%% file: SelectionJournalReady.tex
\documentclass[11pt,letterpaper]{article}
\usepackage{latexsym,color,amssymb,amsthm,fullpage,enumitem,
amsmath,amsfonts,bm}

\usepackage{graphicx}
\usepackage{euscript}

\DeclareFontFamily{OT1}{pzc}{}
\DeclareFontShape{OT1}{pzc}{m}{it}{<-> s * [1.2500] pzcmi7t}{}
\DeclareMathAlphabet{\mathscr}{OT1}{pzc}{m}{it}

\usepackage{graphicx}

\usepackage[left=1in, right=1in, top=1in, bottom=1in]{geometry}

\newcommand{\ignore}[1]{}

\newtheorem{theorem}{Theorem}[section]
\newtheorem{lemma}[theorem]{Lemma}

\newtheorem{proposition}[theorem]{Proposition}

\newtheorem{corollary}[theorem]{Corollary}

\def\eps{{\epsilon}}
\def\aff{{{\mathtt{aff}}}}
\def\conv{{{\mathtt{conv}}}}

\def\1{{(1)}}
\def\2{{(2)}}

\def\deg{{\mathtt{deg}}}

\def\A{{\cal A}}

\def \sign {{\sf sign}}

\def\I{{\cal I}}

\def\F{{\cal F}}

\def\HH{{\cal{H}}}
\def\C{{\cal C}}

\def\E{{{\cal{E}}}}

\def\K{{{\mathcal {K}}}}

\def\N{{\cal N}}
\def\W{{\cal W}}
\def\R{{\cal{R}}}

\def\reals{{\mathbb R}}

\def\G{{\cal G}}

\begin{document}

\begin{titlepage}

\title{Improved Bounds for Point Selections and Halving Hyperplanes in Higher Dimensions}

\author{
Natan Rubin\thanks{Email: {\tt rubinnat.ac@gmail.com}. Ben Gurion University of the Negev, Beer-Sheba, Israel. Supported
by grant 2891/21 from Israel Science Foundation. A preliminary version of this paper has appeared in the Proceedings of the 2024 Annual ACM-SIAM Symposium on Discrete Algorithms (SODA), pp. 4464--4501.
}}

%\date{}
\maketitle

\begin{abstract}Let $(P,E)$ be a $(d+1)$-uniform geometric hypergraph, where $P$ is an $n$-point set in general position in $\reals^d$ and $E\subseteq {P\choose d+1}$ is a collection of $\eps{n\choose d+1}$ $d$-dimensional simplices with vertices in $P$, for $0<\eps\leq 1$. We show that there is a point $x\in \reals^d$ that pierces

\smallskip
\begin{center} 
{\large{
$\Omega\left(\eps^{(d^4+d)(d+1)+\delta}{n\choose d+1}\right)
$}}
\end{center}
 
\noindent simplices in $E$, for any fixed $\delta>0$. This is a dramatic improvement in all dimensions $d\geq 3$, over the previous lower bounds of the general form $\displaystyle\eps^{(cd)^{d+1}}n^{d+1}$, which date back to the seminal 1991 work of Alon, B\'{a}r\'{a}ny, F\"{u}redi and Kleitman.

As a result, any $n$-point set in general position in $\reals^d$ admits only 
\smallskip
\begin{center} 
{\large{$
O\left(n^{d-\frac{1}{d(d-1)^4+d(d-1)}+\delta}\right)
$}}
\end{center}

\noindent halving hyperplanes, for any $\delta>0$, which is a significant improvement over the previously best known bound
$\displaystyle O\left(n^{d-\frac{1}{(2d)^{d}}}\right)$ in all dimensions $d\geq 5$.

An essential ingredient of our proof is the following semi-algebraic Tur\'an-type result of independent interest: Let $(V_1,\ldots,V_k,E)$ be a hypergraph of bounded semi-algebraic description complexity in $\reals^d$ that satisfies $|E|\geq \varepsilon |V_1|\cdot\ldots \cdot |V_k|$ for some $\varepsilon>0$. Then there exist subsets $W_i\subseteq V_i$ that satisfy $W_1\times W_2\times\ldots\times W_k\subseteq E$, and $|W_1|\cdot\ldots\cdots|W_k|=\Omega\left(\varepsilon^{d(k-1)+1}|V_1|\cdot |V_2|\cdot\ldots\cdot|V_k|\right)$. 

%As an immediate by-product, we establish a more efficient variant of the so called 

%This constitutes a drastic improvement over the previous bound of Fox, Pach, and Suk.

%In particular, any $n$-point set $P$ in general position in $\reals^d$ can be subdivided, by the means of Matou\v{s}ek's simplicial partition, into $r$ pairwise disjoint parts $P_1,\ldots,P_r$, of size $\Theta\left(n/r\right)$ each, so that all but $O\left(r^{d+1-1/d}\right)$ of the $(d+1)$-tuples $P_{i_1},\ldots,P_{i_{d+1}}$, with $1\leq i_1<i_2<\ldots< i_k\leq r$, attain a fixed orientation. 

%

\end{abstract}

\maketitle

%\blfootnote{\textup{2010} \textit{Mathematics Subject Classification}: \textup{52A35, 52C10, 52C15, 52C30,  52C35, 52C45, 05D40}}

\end{titlepage}

\section{Introduction} \label{sec:intro}
\subsection {Point selections, and related problems in geometric hypergraphs} 
Let $d\geq 1$ be an integer. 
In this paper we study hypergraphs {\it in the Euclidean $d$-space $\reals^{d}$}, namely, the hypergraphs whose vertices are points in $\reals^d$. 

\medskip
\noindent {\bf Selection in geometric hypergraphs.} A {\it $(d+1)$-uniform simplicial hypergraph $\HH=(P,E)$ in $\reals^d$} is determined by a set $P$ of $n$ points in general position in $\reals^d$ \cite{GeneralPosition}, and a collection $E\subseteq {P\choose d+1}$ of $(d+1)$-size subsets of $P$, which are called the {\it edges} (or rather {\it hyperedges}) of $\HH$.\footnote{In what follows, for any set $A$, and any non-negative integer $k$, we use ${A\choose k}$ to denote the collection of all the possible $k$-size subsets $B\subseteq A$. The general position of $P$ means, in particular, that no $d+1$ of its points lie in the same hyperplane, so that any $(d+1)$-size subset $\tau\in {P\choose d+1}$ indeed yields a proper $d$-dimensional simplex  \cite{JirkaBook}.} The hyperedges $\tau\in E$ are identified with their convex hulls $\conv(\tau)$, which constitute $d$-dimensional {\it simplices} within $\reals^d$. We say that a point $x\in \reals^d$ {\it pierces} the edge $\tau\in E$ if it lies in the relative interior of the convex hull $\conv(\tau)$.

For any positive integer $d$ and $0<\eps\leq 1$, the selection problem concerns the value $F_d(n,\eps)$, which is the largest possible number $f$ with the following property: For any $(d+1)$-uniform simplicial hypergraph $\HH=(P,E)$ in $\reals^d$ with $|V|=n$ vertices and $|E|\geq \eps{|V|\choose d+1}$ hyperedges there exists a point $x\in \reals^d$ (not necessarily in $P$) piercing at least $F$ edges in $E$. The so called {\it First Selection Theorem} states that $F_d(n,1)=\Theta\left({n\choose d+1}\right)$ for any $d\geq 1$.

\begin{theorem}[\cite{BorosFuredi,Barany}]\label{Theorem:FirstSelection} For any $d\geq 1$ there is $c=c(d)>0$ with the following property:
	Let $P$ be a finite point set in general position in $\reals^d$. Then there is a point $x\in \reals^d$ piercing at least $c{n\choose d+1}-o\left(n^{d+1}\right)$ $d$-simplices in ${P\choose d+1}$.
\end{theorem}

%$C_d:=\inf_{n\rightarrow \infty} F_d(n,1)/{n\choose d+1}>0$ in any fixed dimension $d>0$. Namely, for  
The planar variant of Theorem \ref{Theorem:FirstSelection} dates back to Boros and F\"uredi \cite{BorosFuredi} who established the following property: {\it For every set $P$ of $n$ points in the plane, there is a point that belongs to at least $\frac{2}{9}{n\choose 3}-o\left(n^3\right)$ closed triangles induced by the elements of $P$.} This remarkable insight was subsequently generalized by B\'ar\'any \cite{Barany} to all dimensions $d\geq 3$.
While the fraction $2/9$ in the bound of Boros and F\"uredi was shown to be asymptotically tight by Bukh, Matou\v{s}ek and Nivasch \cite{BMN}, it remains an outstanding open problem to determine, in all dimensions $d\geq 3$, the largest possible constant $c=c_d$ that satisfies $F_d(n,1)\geq c_d{n\choose d+1}-o\left(n^{d+1}\right)$. The most recent breakthrough was achieved in 2010 due to Gromov \cite{Gromov}, who attained the lower bound $c_d\geq \frac{2d}{(d+1)!(d+1)}$ via a highly involved topological argument. On the other hand, the stretched grid construction of Bukh, Matou\v{s}ek and Nivasch \cite{BMN} has demonstrated that $c_d\leq \frac{(d+1)!}{(d+1)^{(d+1)}}$.

\medskip
The focus of this paper lies on the more general {\it Second Selection Theorem} which concerns arbitrary $(d+1)$-uniform hypergraphs in $\reals^d$ and yields, for any dimension $d\geq 2$, an exponent $\beta_d$ so that $F_d(n,\eps)=\Omega_d\left({n\choose d+1}\eps^{\beta_d}\right)$ for all $0<\eps\leq 1$. This such result, due to Alon, B\'{a}r\'{a}ny, F\"{u}redi and Kleitman \cite{AlonSelections}, yields an astronomical exponent of $\beta_d=(4d+1)^{d+1}$.

\begin{theorem}[Alon-B\'ar\'any-F\"uredi-Kleitman \cite{AlonSelections}, 1990]\label{Theorem:Main}
	Let $\HH=(P,E)$ be a $(d+1)$-uniform geometric hypergraph with $|E|\geq \eps{n\choose d+1}$ hyperedges. Then there exists a point $x\in \reals^d$ piercing $\displaystyle\Omega\left(\eps^{(4d+1)^{d+1}}{n\choose d+1}\right)$ hyperedges.
\end{theorem}

One of the key ingredients in the proof of Alon, B\'{a}r\'{a}ny, F\"{u}redi and Kleitman, which is briefly sketched in Appendix \ref{Appendix:Tverberg}, is the Erd\H{o}s-Simonovits theorem \cite{ErdosSimon} -- a general T\'uran-type result for hypergraphs which yields  relatively many copies of the complete $(d+1)$-partite $(d+1)$-uniform simplicial hypergraph $K_{t,\ldots,t}$ within $(P,E)$, so that each part is comprised of some $t=t_d=\Theta(d)$ points of $P$. According to the so called Colored Tverberg Theorem \cite{ColoredTverberg,BMZ} -- the other key ingredient in the proof of Alon, B\'{a}r\'{a}ny, F\"{u}redi and Kleitman -- each of these complete $(d+1)$-partite instances encompasses $d+1$ vertex-disjoint simplices $\tau_1,\ldots,\tau_{d+1}\in E$ that can be pierced by a single point $x\in \bigcap_{i=1}^{d+1} \conv(\tau_i)$. 

%The proof of Theorem \ref{Theorem:Main} used the Erd\H{o}s-Simonovits theorem \cite{ErdosSimon} -- a general T\'uran-type result for hypergraphs -- in order to obtain many copies of the complete $t$-partite $(d+1)$-uniform hypergraph $K_{t,\ldots,t}$ within $(P,E)$. Each of these copies involves $d+1$ pairwise disjoint subsets $P_1,\ldots,P_{d+1}$ of $t$ points each, and so that $E$ encompasses all the $t^{d+1}$  simplices $f$ with $|f\cap V_i|=1$ for all $1\leq i\leq d+1$. 
%Provided that $t\geq 4d+1$, the so called Colored Tverberg Theorem  \cite{ColoredTverberg,BMZ} then implies that some $d+1$ of these simplices, with disjoint vertex sets, can always be pierced by a single point. Repeating this argument for each copy of $K_{t,\ldots,t}$, Alon, B\'{a}r\'{a}ny, F\"{u}redi and Kleitman obtained $\eps^{O(1)}{n\choose d+1}$ $(d+1)$-sets $\{f_1,\ldots,f_{d+1}\}\in {E\choose d+1}$ of $d$-simplices so that each of them can be pierced by a single point. 
%which had been derived by \v{Z}ivaljevi\'c and Vre\'cica via the exotic topological framework of equivariant maps \cite{BorsukUlam}, 
%However,the enormous selection exponent $\beta_d$ in Theorem \ref{Theorem:Main} stems from the use of the Erd\H{o}s-Simonovits theorem \cite{ErdosSimon} -- a general T\'uran-type result which is oblivious of the intrinsic geometry of geometric hypergraphs.

The central open question in this regard is to determine the smallest possible exponent $\beta_d$ of $\eps$ so that Theorem \ref{Theorem:Main} would still hold. 
 While the exponent $\beta_d=(4d+1)^{d+1}$ in Theorem \ref{Theorem:Main} depends on the precise constant $t_d=\Theta(d)$ in the Colored Tverberg Theorem -- a rather deep result of independent interest that was derived by \v{Z}ivaljevi\'c and Vre\'cica \cite{ColoredTverberg} via the topological framework of equivariant maps \cite{BorsukUlam}, its overall form $\beta_d=(t_d)^{d+1}$ stems from the Erd\H{o}s-Simonovits theorem, which is oblivious of the implicit geometry of the simplices in $(P,E)$. 
 
Despite the long-known ad-hoc estimates $F_2(n,\eps)=\Omega\left(\eps^{3+o(1)} {n\choose 3}\right)$ in dimension $d=2$ \cite{SelectionPlane,EppsteinSelection},
%the better ad-hoc estimate of $F_2(n,\eps)=$ in dimension $d=2$, Theorem \ref{Theorem:Main} had remained the state of the art for the following 20 years. 
the only improvement of the general selection exponent $\beta_d$, from $(4d+1)^{d+1}$ to $(2d+2)^{d+1}$ (and $(d+1)^{d+1}$ if $d+2$ is a prime larger than $4$), was attained in 2010 by Blagojevi\'c, Matschke and Ziegler \cite{BMZ}, by slightly improving the Colored Tverberg constant and then plugging it into the old analysis of Alon, B\'{a}r\'{a}ny, F\"{u}redi and Kleitman. See the survey articles \cite{BaranyKalai,TopologicalSurvey} and a popular textbook \cite[Chapter 9]{JirkaBook}, for a more complete review of the state of the art.

Let us also mention the existence of interesting {\it colored} variants of Theorem \ref{Theorem:FirstSelection}, due to Pach \cite{PachTheorem} and Karasev \cite{Karasev}, which are formulated in Section \ref{Sec:Prelim}. The latter result is used in the sequel to establish the improved selection exponent $\beta_d$.

\medskip
\noindent{\bf Halving simplices and $k$-sets in $\reals^d$.} One of the best known, and most challenging, open problems in discrete and computational geometry\footnote{Specifically, it figures as Problem 7 on the Open Problems Project \cite{TOPP}.} is to determine the maximum possible number $\psi_d(n)$ of so called {\it halving simplices} that can be determined by a set $P$ of $n$ points in general position in $\reals^d$ \cite[Section 11]{JirkaBook}. These are the simplices in ${P\choose d}$ whose supporting hyperplanes $H$ leave at most $\lceil (n-d)/2\rceil$ of the remaining $n-d$ points of $P$ in each open halfspace of $\reals^d\setminus H$. 

 B\'ar\'any, F\"uredi and Lov\'asz \cite{BFL} used the so called Lov\'asz Antipodality Lemma \cite{Lovasz} to deduce the following fundamental relation between the quantities $F_{d-1}(n,\eps)$ and $\psi_d(n)$; also see \cite[Theorem 11.3.3]{JirkaBook}. %It is a rather direct corollary of the so called antipodality property that was first observed by Lov\'asz in 1971 \cite{Lovasz}. 

\begin{theorem}[\cite{BFL}]\label{Theorem:RelationHalving}
	For any $d\geq 3$, we have that $\psi_d(n)=O(n^{d-1/\beta_{d-1}})$, where $\beta_{d-1}$ denotes the exponent in the Second Selection Theorem. 
\end{theorem}

Thus, the improved selection exponent of Blagojevi\'c, Matschke, and Ziegler \cite{BMZ} yields $\psi_d(n)=O\left(n^{d-1/(2d)^{d}}\right)$. 
While Theorem \ref{Theorem:RelationHalving} yields the so far only known general $o\left(n^d\right)$-upper bound for the number of halving simplices in dimension $d\geq 5$ (see \cite{BaranyKalai} and \cite[Sections 11]{JirkaBook}), vastly better upper bounds of $\psi_2(n)=O\left(n^{4/3}\right)$ \cite{Dey}, $\psi_3(n)=O\left(n^{5/2}\right)$ \cite{ksets3D}, and $\psi_4(n)=O\left(n^{4-1/18}\right)$ \cite{ksets4D} have been derived via ad-hoc arguments -- still a far cry from the best known lower bound $\psi_d(n)=\Omega\left(e^{\sqrt{\log n}} n^{d-1}\right)$, due to T\'{o}th \cite{Toth}.
  
Combined with the Clarkson-Shor sampling argument \cite{CS}, any upper bound of the form $\psi_d(n)=O\left(n^{d-\alpha}\right)$ extrapolates, for $0\leq k\leq (n-d)/2$, to a more general upper bound $O\left(n^{\lfloor d/2\rfloor}(k+1)^{\lceil d/2\rceil-\alpha}\right)$ on the maximum number of
  %$O\left(\min\{(k+1)^d\psi(n/(k+1)),(n-k+1)^d\psi(n/(n-k+1))\}\right)$ 
  \noindent {\it $k$-sets} -- subsets of $k$ points within an $n$-point $P$ in $\reals^d$, that can be cut out by open halfspaces; see \cite{Levels} and \cite[Theorem 11.1.1]{JirkaBook}. By the means of point-hyperplane duality \cite[Section 5]{JirkaBook}, the same upper bound applies to the maximum complexity of the so called {\it $k$-th level} in an arrangement of $n$ hyperplanes within $\reals^d$ -- a fundamental structure in computational geometry which bears significance to geometric range searching problems and optimization in a fixed dimension \cite{ArrangementsSurvey}.

\subsection{Semi-algebraic hypergraphs} %Let $k\geq 2$. We say that a $k$-uniform, $k$-partite hypergraph $(V_1,\ldots,V_k,E)$ has {\it bounded semi-algebraic description complexity in $\reals^d$} if 1. we have that $V_i\subseteq \reals^d$ for all $1\leq i\leq k$, and 2. the edge set $E\subseteq V_1\times V_2\times\ldots\times V_k$ is cut out by a semi-algebraic set $X\subseteq \reals^{d\times k}$ of bounded description complexity that is determined by a Boolean combination bounded number of polyno

Most Ramsey-type and regularity results that exist for general graphs and hypergraphs, are known to hold in a much stronger form for hypergraphs with a {\it bounded semi-algebraic description complexity} \cite{AlonSemi,Overlap,Regularity}.

 We say that a set $Y\subseteq \reals^{d}$ is {\it semi-algebraic} if it is the locus of all the points in $\reals^{d}$ that satisfy a given Boolean combination of a finite number of polynomial equalities and inequalities in the $d$ coordinates of $\reals^{d}$ \cite{BasuBook}. Such a semi-algebraic description of $Y$ has complexity $(D,s)$ if it involves $s$ polynomial constraints of maximum degree $D$. The definition naturally extends to subsets $A$ of matrices in $\reals^{d\times k}$, which can be treated as vectors in $\reals^{dk}$. 

We then say that a $k$-uniform, $k$-partite hypergraph $(V_1,\ldots,V_k,E)$ admits {\it a semi-algebraic description of complexity $(D,s)$} in $\reals^d$ if $\bigcup_{i=1}^k V_i\subseteq \reals^d$, and the edge set $E$ is induced, as a subset of $V_1\times\ldots \times V_k\subseteq \reals^{d\times k}$, by a semi-algebraic set $Y\subseteq \reals^{d\times k}$ of complexity $(D,s)$ (so that $E=Y\cap (V_1\times\ldots \times V_k)$). The definition naturally extends to symmetric $k$-uniform hypergraphs $(V,E)$ that are not apriori $k$-partite.
%it is the locus of all the points in $\reals^{d}$ that satisfy a given Boolean combination of a bounded number of polynomial equalities and inequalities in the $d$ coordinates of $\reals^{d}$, each of bounded degree.\footnote{See Section \ref{Subsec:SemiAlgebraic} for a more general (and also more explicit) framework for handing the description complexity of semi-algebraic sets and semi-algebraic relations.} 

%To extend this definition to the $k$-uniform hypergraphs $(V,E)$ that are not apriori $k$-partite, we require instead that $E=\left
%\{\{v_1,\ldots,v_k\}\in {V\choose k}\mid \left[v_1,\ldots,v_k\right]\subseteq Y\right\}$, where $\left[v_1,\ldots,v_k\right]$ denotes the set of all the permutation sequences of $\{v_1,\ldots,v_k\}$. 
%See Section \ref{Subsec:SemiAlgebraic} for a more rigorous treatment of semi-algebraic sets, semi-algebraic hypergraphs, and their description via polynomial inequalities.

%By restricting our attention to the semi-algebraic subsets $Y\subseteq \reals^{d\times k}$ that are symmetric with respect to any permutation amongst the $k$ columns of their matrices, the notion readily extends to all the $k$-uniform hypergraphs $(V,E)$ in $\reals^d$, with $V\subseteq \reals^d$ and $E\subseteq 2^V$.

\smallskip
\smallskip
\noindent{\bf A Tur\'an-type theorem for semi-algebraic hypergraphs.} The {\it polynomial regularity lemma} of Fox, Pach and Suk \cite{Regularity} states that, for any $0<\varepsilon<1/2$, and
any hypergraph $\G=(V,E)$ with a bounded semi-algebraic description complexity, there is an equitable\footnote{A subdivision is called {\it equitable} if any pair of parts $U_i,U_j$ are of nearly same size, so that $\left||U_i|-|U_j|\right|\leq 1$. } subdivision $V=U_1\uplus \ldots \uplus U_K$ into $K\leq 1/\varepsilon^{c}$ parts, with the property that all by an $\varepsilon$-fraction of the $k$-tuples $(U_{i_1},\ldots,U_{i_k})$ are homogeneous -- we have that either $U_{i_1}\times U_{i_2}\times\ldots\times U_{i_k}\subseteq E$ or $\left(U_{i_1}\times U_{i_2}\times\ldots\times U_{i_k}\right)\cap E=\emptyset$. 

%Given a $k$-uniform hypergraph $\G=(V,E)$, we say that a $k$-tuple of pairwise disjoint vertex subsets $W_1,\ldots,W_k\subseteq V$, with $1\leq i\leq k$, is {\it homogeneous} if 
%we have that either $W_1\times\ldots \times W_k\subseteq E$ or  $\left(W_1\times\ldots \times W_k\right)\cap E=\emptyset$. 
%A subdivision $V=U_1\uplus\ldots\uplus U_K$, into $K$ parts, is called {\it equitable} if any pair of subsets $U_i,U_j$ are of nearly same size, so that $\left||U_i|-|U_j|\right|\leq 1$. 
%We say that a $k$-uniform, $k$-partite hypergraph $(V_1,\ldots,V_k,E)$ is {\it $\varepsilon$-dense} if we have that $|E|\geq \varepsilon |V_1|\cdot\ldots\cdot |V_k|$.%

% While exponent $c$ is not provided explicitly, it is a constant which depends only on $d$, $k$, and the fixed constants $D$ and $s$ that bound the semi-algebraic description complexity of $\G$.

 We say that a $k$-uniform, $k$-partite hypergraph $(V_1,\ldots,V_k,E)$ is {\it $\varepsilon$-dense} if we have that $|E|\geq \varepsilon |V_1|\cdot\ldots\cdot |V_k|$.
The regular partition of Fox, Pach, and Suk relies on the following Tur\'an-type statement of independent interest:
for any $\varepsilon$-dense, $k$-uniform, and $k$-partite hypergraph $(V_1,V_2,\ldots,V_k,E)$ of bounded semi-algebraic description complexity and, in particular, of bounded maximum degree $D$, there exist subsets $W_i\subseteq V_i$ of size $|W_i|=\Omega\left(\varepsilon^{d+D\choose d}|V_i|\right)$, and so that $W_1\times \ldots\times W_k\subseteq E$.
%\footnote{The implicit constant of proportionality in such results depends on $d,k$, and the constant $c$ which bounds the complexity of the semi-algebraic descriptions of the hypergraphs. As a rule, $c$ will be specific to the class of the hypergraphs in $\reals^d$ under consideration, which share the same semi-algebraic description.}

\smallskip
Our improved bound for the point-selection problem will require a more efficient Tur\'an-type theorem for semi-algebraic hypergraphs, which incurs no dependence on the maximum degree $D$ in the exponent. %As a by-product, we derive a more efficient variant of the polynomial regularity lemma.

\subsection{Main results} 
\noindent{\bf 1. An improved selection exponent $\beta_d$ in dimension $d\geq 3$.} The primary contribution of this paper lies in reducing the best selection exponent $\beta_d$ that is known in any dimension $d\geq 3$, from $(2d+2)^{d+1}$ \cite{BMZ} to $d^5+o\left(d^5\right)$.
\begin{theorem}\label{Thm:MainMain}
Let $d\geq 2$ be an integer, and $\delta>0$. Then for any set $P$ of $n$ points in general position in $\reals^d$, and any subset $E\subseteq {P\choose d+1}$ of $d$-dimensional simplices with vertices in $P$, there is a point $x\in \reals^d$ piercing $\displaystyle\Omega\left(\eps^{(d^4+d)(d+1)+\delta}{n\choose d+1}\right)$ of the simplices in $E$; the implicit constants of proportionality may depend on $\delta$ and $d$.
\end{theorem}

\noindent{\bf 2. An improved bound for halving hyperplanes and $k$-sets in dimension $d\geq 5$.} Combined with Theorem \ref{Theorem:RelationHalving}, Theorem \ref{Thm:MainMain} yields the following estimate for the maximum number $\psi_d(n)$ of the halving simplices that can be determined by any $n$ points in general position in $\reals^d$, which improves over the previous state of the art in all dimensions $d\geq 5$.

\begin{theorem}\label{Thm:ksets}
Let $d\geq 2$ and $\delta>0$. Then for any $n\geq 1$ we have that 
$$
\psi_d(n)=O\left(n^{d-\frac{1}{d(d-1)^4+d(d-1)}+\delta}\right).
$$

\noindent More generally, for any $0\leq k\leq  (n-d)/2$, any set of $n$ points in general position in $\reals^d$ determines
$$
O\left(n^{\lfloor d/2\rfloor}(k+1)^{\lceil d/2\rceil-\frac{1}{d(d-1)^4+d(d-1)}+\delta}\right)
$$ 
\noindent $k$-sets.
The implicit constants of proportionality may depend on $\delta$ and $d$.
\end{theorem}

\noindent{\bf 3. Efficient semi-algebraic regularity and Tur\'an-type theorems.}  
 A essential ingredient of our proof of Theorem \ref{Thm:MainMain} is an efficient Tur\'an-type theorem for semi-algebraic hypergraphs.\footnote{Since the hypergraph $(P,E)$ in Theorem \ref{Thm:MainMain} does not necessarily admit a semi-algebraic description of bounded complexity in $d$, our Tur\'an-type result will be applied for a auxiliary semi-algebraic hypergraph, which will be defined in the parametric space $\reals^{d^2}$ -- the natural representation space of $(d-1)$-dimensional simplices in $\reals^d$.}
Note that the implicit maximum degree $D$ for the majority of the ``natural" semi-algebraic hypergraphs, including these that arise in the proof of Theorem \ref{Thm:MainMain}, is not easily seen to be polynomial in the ambient dimension $d$. Hence, using the bound of Fox, Pach and Suk \cite{Regularity} would have resulted in an astronomic selection exponent $\beta_d$. In the sequel we establish the following stronger property.

\begin{theorem}\label{Theorem:NewTuran}
% For any fixed integers $d\geq 1$, $k\geq 2$, $D\geq 0$, and $s\geq 1$, there is a real number $C=C(d,k,D,s)>0$ so that the following statement holds.
Let $\varepsilon>0$, and let $(V_1,\ldots,V_k,E)$ be an $\varepsilon$-dense, $k$-uniform, and $k$-partite hypergraph in $\reals^d$, that admits a semi-algebraic description of bounded complexity $(D,s)$. Then there exist subsets $W_1\subseteq V_1,\ldots,W_k\subseteq V_k$ with the following properties:

\begin{enumerate}
	\item  $W_1\times W_2\times\ldots\times W_k\subseteq E$,
	\item $|W_1|\cdot |W_2|\cdot \ldots\cdot |W_k|=\Omega\left(\varepsilon^{d(k-1)+1}\cdot |V_1|\cdot |V_2|\cdot \ldots\cdot |V_k|\right)$. 
\end{enumerate}

Furthermore, for all $1\leq i\leq k-1$ we have that
\begin{equation*}
|W_i|=\begin{cases}
\Omega\left(\varepsilon^{d}|V_i|\right) &\text{if $d_1\geq 2$}\\
\Omega\left(\frac{\varepsilon|V_i|}{\log^2(1/\varepsilon)}\right) &\text{if $d_1=1$},
\end{cases}
\end{equation*}

\noindent and we have that $|W_k|=\Omega\left(\varepsilon|V_k|\right)$. The implicit constants of proportionality depend only on $d,k,D$ and $s$.
\end{theorem}

\medskip
In contrast to the analysis of Fox, Pach and Suk, who invoked Chazelle's Hyperplane Cutting Lemma \cite{Cuttings} in the Euclidean space $\reals^{d'}$ of dimension $d'={d+D\choose d}-1$, at the heart of the proof of Lemma \ref{Lemma:BipartiteRegularity} lies a more efficient spatial subdivision which combines Haussler's Packing Lemma for geometric range spaces  \cite{Haussler} with the multi-level polynomial partition\footnote{In the companion study of the author \cite{RubinRegularity}, the polynomial method of Guth and Katz \cite{GuthKatz} is further adapted to establish a more efficient version of the polynomial regularity lemma.} 
 of Matou\v{s}ek and P\'at\'akova \cite{MultiLevel}. 

\subsection{Proof overview, and the roadmap to the paper} 
 In contrast to the argument of Alon, B\'{a}r\'{a}ny, F\"{u}redi and Kleitman \cite{AlonSelections}, which combines the Erd\H{o}s-Simonovits theorem with the Colored Tverberg Theorem in an essentially {\it bottom-up} manner so as to force many $(d+1)$-tuples of simplices $\tau_1,\ldots,\tau_{d+1}\in E$ that satisfy $\bigcap_{i=1}^{d+1} \conv(\tau_i)\neq \emptyset$, the proof of Theorem \ref{Thm:MainMain} proceeds in a {\it top-down} fashion. 
 To this end, we consider the almost-even partition $\Pi=\{P_1,\ldots,P_r\}$ of $P$, by the means of the Simplicial Partition Theorem of Matou\v{s}ek \cite{PartitionTrees}, into $r$ pairwise-disjoint subsets $P_1,\ldots,P_r$ of size $\Theta(n/r)$; each part $P_i\in \Pi$ is enclosed in a simplex $\Delta_i$ so that any hyperplane crosses only $O\left(r^{1-1/d}\right)$ of the elements of the family $\Sigma=\{\Delta_1,\ldots,\Delta_r\}$.  
 
We then study the complete $(d+1)$-uniform hypergraph over $\Sigma$, whose edges correspond to $(d+1)$-size subfamilies of simplices within $\Sigma$.
A fine enough simplicial partition $\Pi$ of $P$ guarantees that the vast majority of the sub-families $\K=\{\Delta_{i_1},\ldots,\Delta_{i_{d+1}}\}$  of $d+1$ simplices, with $1\leq i_1<i_2<\ldots<i_{d+1}\leq r$, are ``far apart" and, therefore, resemble points in the following topological sense: the intersections among the convex hulls of the various combinations $\bigcup_{j\in J}\Delta_{i_j}$, with $J\subsetneq [d+1]$, resemble those of the faces on the boundary of a $d$-dimensional simplex. In particular, all the simplices $\conv\left(p_1,\ldots,p_{d+1}\right)$, whose $d+1$ vertices $p_j$ are selected from the $d+1$ enclosed sets $P_{i_j}\subset \Delta_{i_j}$, must have a non-empty common intersection; see Figure \ref{Fig:Tight}.
In the sequel, such families $\K=\{\Delta_{i_1},\ldots,\Delta_{i_{d+1}}\}$ will be referred to as {\it tight}, whereas the remaining $(d+1)$-size families will be called {\it loose}.\footnote{This is a natural $d$-dimensional extension of the notion of {\it tight triples} of convex sets in $\reals^2$ as defined, e.g., in a recent survey of Holmsen \cite{T3}. For the sake of brevity, this outline omits several other classes of families within our partition of $P$, along with certain classes of edges in $E$, which altogether have no effect on the bound in Theorem \ref{Thm:MainMain}.} 

To establish Theorem \ref{Thm:MainMain}, it suffices to show that only $O\left(r^{d+1-1/\alpha_d}\right)=o\left(r^{d+1}\right)$ of the $(d+1)$-size families $\K$ within $\Sigma$ are loose, where $\alpha_d=d^4+d$ denotes a suitable constant whose choice will be explained in the sequel.
Fixing $r=(1/\eps)^{\alpha_d}$ will then guarantee the existence of at least $|E|/2$ edges in $E$ so that each of them can be ``assigned" (via its $d+1$ vertices) to a tight $(d+1)$-size sub-family $\K$ within $\Sigma$. Thus, by the pigeonhole principle, at least $\frac{|E|}{2{r\choose d+1}}=\epsilon^{d^5+O(d^4)}n^{d+1}$ of these edges will be assigned to the {\it same} tight family $\K$, and have a non-empty common intersection.

\begin{figure}
    \begin{center}      
        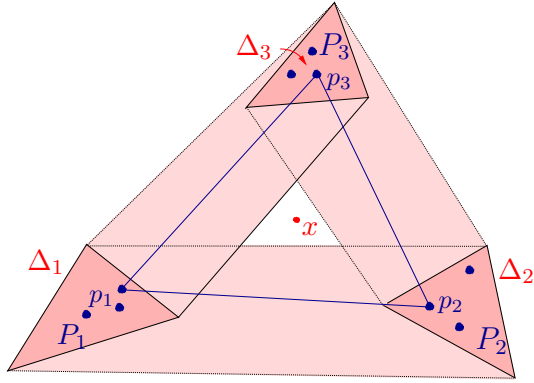%\hspace{2cm}\input{NotSurrounded.pdf_t}
        \caption{\small A tight triple of simplices $\{\Delta_1, \Delta_2,\Delta_3\}\subset \Sigma$ in $\reals^2$, with enclosed subsets $P_1,P_2,P_3$. The point $x\in \reals^3$ pierces every triangle $\triangle p_1p_2p_3$, with vertices $p_1\in P_1,p_2\in P_2$ and $p_3\in P_3$.}
        %Right: Proposition \ref{Prop:NotSurrounded} for $d=3$: The origin $O$ lies in the triangle $\triangle x_1x_2x_3$, for $x_1\in \Delta_1,x_2\in \Delta_3,x_3\in \Delta_3$, yet it is not surrounded by $\Delta_1,\Delta_2,\Delta_3$. The simplices $\Delta_1$ and $\Delta_2$ are crossed by a line through $O$.}
        \label{Fig:Tight}
    \end{center}
\end{figure}

To bound the number of the loose families $\K$ of $d+1$ simplices within $\Sigma$, notice that such $(d+1)$-tuples comprise a hypergraph of bounded semi-algebraic description complexity in $\reals^{d\times (d+1)}$ -- the natural representation space of the simplices of $\Sigma$.\footnote{Using a more careful argument, the dimensionality of this hypergraph can be reduced to $d^2$.} 
The desired upper estimate stems from the following two observations.
\begin{enumerate}
	\item  If the above semi-algebraic hypergraph over $\Sigma$ is $\varepsilon$-dense, then Theorem \ref{Theorem:NewTuran}, together with Karasev's colored selection result, yield a point $x\in \reals^d$ that lies in the convex hulls $\conv\left(\bigcup\K\right)$ of at least $\varepsilon^{d^3+1}{r\choose d+1}$ loose families $\K\in {\Sigma\choose d+1}$.

	\item On the other hand, the fact that no hyperplane crosses more than $O\left(r^{1-1/d}\right)$ of the simplices in $\Sigma$, implies that no point in $\reals^d$ can pierce more than $O\left(r^{d+1-1/d}\right)$ of the above convex hulls $\conv\left(\bigcup \K\right)$.
\end{enumerate}

The rest of this paper is organized as follows.
In Section \ref{Sec:Prelim} we introduce the essential notation, and assemble the key technical ingredients that underly the improved bound of Theorem \ref{Thm:MainMain}: (i) the colored selection theorems of Pach \cite{PachTheorem} and 
Karasev \cite{Karasev}, (ii) the crossing patterns of hyperplanes amid a family of $d+1$ convex sets, including the full characterization of tight families of convex sets, 
(iii) Matou\v{s}ek's simplicial partition theorem, and (iv) a more explicit framework for the study of semi-algebraic sets, and hypergraphs of bounded semi-algebraic description complexity in $\reals^d$. 
The main result of this paper, namely, Theorem \ref{Thm:MainMain} is established in Section \ref{Sec:Main}, whereas the proof of our semi-algebraic Tur\'an-type result -- Theorem \ref{Theorem:NewTuran} -- is relegated to Section \ref{Sec:Polynomial}. 
In Section \ref{Section:Concluding}, we conclude the paper with several by-products of our proof of Theorem \ref{Thm:MainMain}. These include an interesting extension of Pach's theorem, and an improved estimate for the same-type lemma of B\'ar\'any and Valtr \cite{BaranyValtr}.

%The original proof of Theorem \ref{Theorem:Main} by Alon, B\'{a}r\'{a}ny, F\"{u}redi and Kleitman, along the broader relation between Tverberg-type and selection theorems, are sketched in Appendix \ref{Appendix:Tverberg}.

\section{Geometric preliminaries}\label{Sec:Prelim}

\subsection{Notation} \label{Subsec:Notation}

For any point set $A\subseteq \reals^d$ we use $\conv(A)$ to denote the {\it convex hull} of $A$, and we use $\overline{A}$ to denote the closure of $A$ under the standard Euclidean topology in $\reals^d$. In addition, we use $\aff(A)$ to denote the {\it affine closure} of $A$ -- the smallest possible affine space within $\reals^d$ (that is, a translate of a linear sub-space) that contains $A$ \cite[Section 1]{JirkaBook}.

As was previously mentioned, it is assumed that the vertex set $P$ of the $(d+1)$-uniform simplicial hypergraph $(P,E)$ is in {\it general position} \cite{GeneralPosition}, which in particular means that any subset $A\subseteq P$ of at most $d+1$ points is {\it affinely independent}, in the sense that $\aff(A)$ is linearly isomorphic to $\reals^{|A|-1}$ or, equivalently, $\conv(A)$ is a $(|A|-1)$-dimensional simplex with a non-empty relative interior.

For the sake of our analysis, it pays to assume, with no loss of generality, a stronger form of general position, which in particular implies that any $d+1$ pairwise disjoint subsets $A_1,\ldots,A_{d+1}\in {P\choose d}$ satisfy $\bigcap_{i=1}^{d+1}\aff(A_i)=\emptyset$. This property can be enforced, e.g, by applying an infinitesimally small perturbation to the set $P$ \cite{GeneralPosition}.

%\medskip
%\noindent{\bf $\eta$-perturbation.} Let $k,d\geq 1$ be integers. For a pair of points $p=(p(1),\ldots,p(d)),p'=(p'(1),\ldots,p'(d))\in \reals^{d}$, we say that a point $p'$ is an {\it $\eta$-perturbation} of $p$ if we have $|p(i)-p'(i)|\leq \varepsilon$ for all $1\leq i\leq d$.

%Let $(P=\{p_1,\ldots,p_n\},E)$ be a $k$-uniform hypergraph so that $P\subseteq \reals^{d}$. We say that a hypergraph $(P'=\{p'_1,\ldots,p'_n\},E')$ is an {\it $\eta$-perturbation} of $(P,E)$ if each vertex $p'_i\in P'$ is an $\eta$-perturbation of the corresponding vertex $p_i\in P$, for all $1\leq i\leq n$, and the hyperedges $\tau'\in E'$ correspond to the hyperedges $\tau\in E$, so that a vertex $p'_i\in P'$ belongs to $\tau'$ if and only if its counterpart $p_i\in P$ belongs to the respective edge $\tau\in E$.  

%Note that for any $(d+1)$-uniform simplicial hypergraph $(V,E)$ in $\reals^{d}$ whose vertex set $V\subseteq \reals^d$ is in general position, there is an $\varepsilon>0$ so that no $\varepsilon$-perturbation $(V,E)$ of 
%$(V,E)$ affects the topology of the intersections $\left{\bigcap_{f\in F}{\sf int}(\conv(f))\mid F\subseteq E\right\}$, which in particular guarantees that $\max_{x\in \reals^d}\left|\{f\in E\mid x\in {\sf int}(\conv(f))\}\right|=\max_{x\in \reals^d}\left|\{f'\in E'\mid x\in {\sf int}(\conv(f'))\}\right|$ for any $\varepsilon$-perturbation $(V',E')$ of $(V,E)$.

\subsection{The colored selection theorems of Pach and Karasev} 

As was noted in Section \ref{sec:intro}, the original proof of Theorem \ref{Theorem:Main} by Alon, B\'{a}r\'{a}ny, F\"{u}redi and Kleitman \cite{AlonSelections}, which we sketch in Section \ref{Appendix:Tverberg}, employs a combination of the theorem of Erd\H{o}s and Simonovits \cite{ErdosSimon}, the Colored Tverberg Theorem \cite{ColoredTverberg}, and the Fractional Helly's Theorem \cite{FracHelly}.

Let us also mention a closely related Tverberg-type result, due to Pach \cite{PachTheorem}, which stands ``in-between" the First Selection Theorem and the Colored Tverberg Theorem.

\begin{theorem}[Pach's theorem \cite{PachTheorem}] \label{Theorem:Pach}
For any integer $d\geq 2$ there is a constant $c'_d>0$ so that the following statement holds:
	Let $P_1,\ldots,P_{d+1}$ be pairwise disjoint $n$-points sets in general position in $\reals^d$ (also called color classes, or simply colors). Then there is a point $x\in \reals^d$, and subsets  $Q_i\subseteq P_i$ of cardinality $|Q_i|\geq c'_dn$, for all $1\leq i\leq d+1$, so that $x$ pierces every simplex 
	in the set $\{\conv(p_1,\ldots,p_{d+1})\mid p_1\in Q_1,\ldots,p_{d+1}\in Q_{d+1}\}$ (whose simplices are commonly called colorful, or rainbow).
	\end{theorem}

Our argument in Section \ref{Sec:Main} will rely on the following colored variant of First Selection Theorem, which was established by Karasev \cite{Karasev} in the wake of Gromov's breakthrough \cite{Gromov}, and subsequently employed by Fox, Pach and Suk \cite{Regularity} to improve the constant $c'_d$ in Theorem \ref{Theorem:Pach}.

\begin{theorem}[The Colored Selection Theorem of Karasev \cite{Karasev}] \label{Theorem:Karasev}
For any integer $d\geq 2$, and any $d+1$ point sets $P_1,\ldots,P_{d+1}$, whose union $\bigcup_{i=1}^{d+1} P_i$ is in general position in $\reals^d$, 
there is a point $x\in \reals^d$ that lies in at least 
$$
\frac{1}{(d+1)!}|P_1|\cdot |P_{2}|\cdot\ldots\cdot |P_{d+1}|
$$
\noindent	(closed) simplices in $\{\conv(p_1,\ldots,p_{d+1})\mid p_1\in P_1,\ldots,p_{d+1}\in P_{d+1}\}$.\footnote{More precisely, the original statement of Karasev deals with simplices whose $d+1$ vertices are drawn independently from some $d+1$ {\it strictly continuous} probability distributions. However, it immediately extends, via a standard limiting argument, to piercing the {\it closed} rainbow simplices of $P_1\times\ldots\times P_{d+1}$ in the setting of Theorem \ref{Theorem:Karasev}. To this end, we fix an arbitrary small $\eta>0$, and replace each set $P_i$ by a uniformly weighted combination of $|P_i|$ uniform distributions over the $\eta$-discs $B(p,\eta)$, with $p\in P_i$. Applying the continuous statement with ever decreasing radii $\eta_j=1/(j+1)$, for $j\in {\mathbb N}$, yields a sequence of points $x_j=x(\eta_j)$ whose partial limit(s) $x$ meet the criteria of Theorem \ref{Theorem:Karasev}.} 

	\end{theorem}

One of the strengths of Karasev's result is its elementary 1-page proof  which
makes only a nominal use of topology (and, in particular, does not rely on the Colored Tverberg Theorem). More importantly, Karasev's result does not require that $|P_1|=|P_2|=\ldots=|P_{d+1}|$.
Using a more careful analysis, the fraction $\frac{1}{(d+1)!}$ in Theorem \ref{Theorem:Karasev} can be improved to $\frac{2d}{(d+1)!(d+1)}$ \cite{Jiang}.

%Notice that Karasev's formulation does not require that the sets $P_1,\ldots,P_{d+1}$ satisfy $|P_1|=|P_2|=\ldots=|P_{d+1}|$
%; as a matter of fact, its proof is stated for $d+1$ continuous probability distributions, rather than discrete point sets. 
%Hence, it is this latter variant of Pach's Theorem that will be used in our proof of Theorem \ref{Thm:MainMain} in Section \ref{Sec:Main}. 

\subsection{Intersection patterns of hyperplanes and convex sets} \label{Subsec:PrelimTight}

Though the results of this section will be applied primarily to families of simplices of affine dimension $d$ or $d-1$ within $\reals^d$, most of our claims will be stated in full generality for families of compact convex sets.

The space of hyperplanes in $\reals^d$ is commonly identified with the affine Grassmanian $G(d,d-1)$, which is a $d$-dimensional topological manifold \cite{WengerProgress}. Each hyperplane $H$ in $\reals^d$ can be {\it oriented} in two distinct ways which correspond to the possible labelings of the open halfspaces of $\reals^d\setminus H$ as $H^+$ and $H^-$. 

\medskip
\noindent{\bf Arrangements of hyperplanes.} 
Any finite collection $\HH$ of hyperplanes within $\reals^d$ determines the {\it arrangement} $\A(\HH)$ of $\HH$ -- a hierarchical decomposition of $\reals^d$ into $k$-dimensional {\it faces} (or, shortly, {\it $k$-faces}) -- relatively open $k$-dimensional polyhedra, of dimensions $0\leq k\leq d$; see, e.g., \cite[Section 5]{JirkaBook}. The {\it $d$-faces}, or {\it cells}, of $\A(\HH)$ are the connected open components of $\reals^d\setminus \left(\bigcup \HH\right)$.
Each of these cells $\Delta$ yields an orientation of the hyperplanes of $\HH$ so that $\Delta=\bigcap_{H\in \HH} H^-$.
 The boundary complex of each $k$-face in $\A(\HH)$, with $k\geq 1$, is comprised of certain relatively open, $(k-1)$-dimensional faces, within $\A(\HH)$, along with their respective boundary complexes.
 The $(d-1)$-dimensional faces of $\A(\HH)$ are called {\it facets}, and each of them is an open $(d-1)$-dimensional polyhedron which is contained in some hyperplane $H_i$, for $1\leq i\leq d+1$.
 The $0$-dimensional faces of $\A(\HH)$ are called {\it vertices}; each of them is an intersection of some $d$ hyperplanes in $\HH$.  The $1$-dimensional faces of $\A(\HH)$ are called {\it edges}.

If $\HH$ is comprised of $d+1$ hyperplanes $H_1,\ldots,H_{d+1}$ in general position, these determine exactly $d+1$ vertices which lie on the boundary of the only bounded cell $\Delta$ within $\A(\HH)$ -- an open $d$-simplex. 
Each of the unbounded cells of $\A(\HH)$ contains an infinite ray. Hence, within the $(d-1)$-dimensional sphere ${\mathbb S}^{d-1}$ of line directions, such cells of $\A(\HH)$ correspond to the cells that are cut out by the $d+1$ great $(d-2)$-spheres that are parallel to the hyperplanes of $\HH$.

\begin{lemma}\label{Lemma:ArrangementSimplex}
Let $H_1,\ldots,H_{d+1}$ be $d+1$ oriented hyperplanes in general position in $\reals^d$.
Then exactly one of the following cases holds:

\begin{itemize}
	\item[(i)] One of the sets $\bigcap_{i=1}^{d+1} H_i^-$ and $\bigcap_{i=1}^{d+1} H_i^+$ is an open simplex -- the only bounded open cell in $\A(\{H_1,\ldots,H_{d+1}\})$, whereas the other set is empty.
	\item[(ii)] Both sets $\bigcap_{i=1}^{d+1} H_i^-$ and $\bigcap_{i=1}^{d+1} H_i^+$ are non-empty yet unbounded open cells in $\A(\{H_1,\ldots,H_{d+1}\})$.
 \end{itemize}
	
\end{lemma}
\begin{proof}

If $\bigcap_{i=1}^{d+1} H_i^-$ is an open simplex, then the cone $\bigcap_{i=1}^d H_i^+$ which is adjacent to the vertex $v_{d+1}=\bigcap_{i=1}^d H_i$, is contained in $H_{d+1}^-$, so that $\bigcap_{i=1}^{d+1} H_i^+=\emptyset$.

If $\bigcap_{i=1}^{d+1} H_i^-=\emptyset$, the hyperplane $H_{d+1}$ misses the cone $\bigcap_{i=1}^d H_i^-$, whose apex $v_{d+1}=\bigcap_{i=1}^d H_i$ then lies in $H^+_{d+1}$. Notice that the
arrangement $\A(\{H_1,\ldots,H_{d}\})$ encompasses a total of $2^d$ cells. Each of these cells is a cone with apex $v_{d+1}$ and $d$ infinite edges that emanate from $v$. Since $H_{d+1}$ misses $\bigcap_{i=1}^d H_i^-$ (and, in particular, all $d$ of its boundary edges), it must meet the $d$ edges that lie on the boundary of the antipodal cone $\bigcap_{i=1}^d H_i^+$ and, thereby, determine a bounded open simplex $\bigcap_{i=1}^{d+1} H_i^+$.

Lastly, if $\bigcap_{i=1}^d H_i^-$ is unbounded, then so is $\bigcap_{i=1}^d H_i^+$; in particular, $\bigcap_{i=1}^d H_i^+$ must be non-empty. Indeed, there must be a line $L$ in $\reals^d$ whose intersection $L\cap \left(\bigcap_{i=1}^d H_i^-\right)$ is an infinite ray $\vec{\rho}$. Then $L\cap \left(\bigcap_{i=1}^d H_i^+\right)$ too must be an unbounded ray with orientation opposite to that of $\rho$. (Within the sphere ${\mathbb S}^{d-1}$ of directions, the two unbounded cells must correspond to a pair of antipodal facets of the arrangement of the $d+1$ great $(d-2)$-spheres of $S_1,\ldots,S_{d+1}$ that are parallel to, respectively, $H_1,\ldots,H_{d+1}$.)
\end{proof}

%\medskip
%\noindent{\bf Definition.} 
%An {\it ordered $k$-family} in $\reals^d$ is an {\it ordered} $k$-tuple $(X_1,\ldots,X_k)$ of point sets $X_i\subseteq \reals^d$. 
%We say that a $k$-tuple $(x_1,\ldots,x_{k})\in \left(\reals^d\right)^{k}$ is a {\it transversal} to the $k$-family $(X_1,\ldots,X_k)$ if $(x_1,\ldots,x_{k})\in X_1\times\ldots\times X_{d+1}$. 

\medskip
\noindent{\bf Definition -- separated families of convex sets.} We say that a family $\K$ of convex sets in $\reals^d$ is {\it separated} if no $k$ of its members can be crossed by a single $(k-2)$-flat, for $1\leq k\leq d+1$. (If $|\K|\geq d+1$, then it is enough that no $d+1$ sets can be crossed by a single hyperplane.) %We say that a family $\{X_1,\ldots,X_{d+1}\}$ of $d+1$ not necessarily convex, sets in $\reals^d$ is {\it separated} if the family $\{\conv(X_i)\mid 1\leq i\leq d+1\}$ is. 

\smallskip
It is known \cite{WellSeparation,PolWen,WengerProgress} that a family $\K=\{K_1,\ldots,K_{d+1}\}$ of $d+1$ compact convex sets in $\reals^d$ is separated if and only if all the selections $(x_1,\ldots,x_{d+1})\in K_1\times\ldots\times K_{d+1}\subseteq \reals^{d\times (d+1)}$, each with $d$ coordinates $x_i=\left(x_{i}(1),\ldots,x_{i}(d)\right)$, have the same non-zero {\it orientation} $\chi(x_1,\ldots,x_{d+1})\in \{-,+\}$, which is given by\footnote{Notice that the common orientation itself may depend on of the labeling $K_1,\ldots,K_{d+1}$.}

\begin{equation*}
\chi(x_1,\ldots,x_{d+1}):=\sign\:{\sf det}
\begin{pmatrix}
    1 & 1 & \cdots & 1\\
	x_{1}(1) & x_{2}(1)& \cdots & x_{d+1}(1)\\
	x_{1}(2)& x_{2}(2)& \cdots & x_{d+1}(2)\\
	\vdots & \vdots & \vdots & \vdots\\
	x_{1}(d)& x_{2}(d)& \cdots & x_{d+1}(d)\\
\end{pmatrix}.
\end{equation*}

If a family $\K$ of $d+1$ compact convex sets in $\reals^d$ is not separated, we say that it is {\it crossed}.

\medskip
\noindent{\bf Definition.}
We say that a hyperplane $H$ in $\reals^d$ is {\it tangent to} (or {\it supports}) a convex set $K\subseteq \reals^d$ if we have that $H\cap K\neq \emptyset$, and $K$ contained in one of the closed halfspaces that are determined by $H$.

\medskip
\noindent{\bf Definition.} Let $\K$ and $\K'$ be families of convex sets in $\reals^d$. We say that a hyperplane $H$ in $\reals^d$ separates $\K$ from $\K'$ if $\bigcup \K$ and $\bigcup \K'$ lie in the opposite closed halfspaces that are determined by $H$. (Note that $H$ can be tangent to one or more sets of $\K\cup \K'$).

\begin{lemma}\label{Lemma:SeparatingHyperplane}
	 Let $\K=\{K_1,\ldots,K_{d+1}\}$ be a separated family of $d+1$ non-empty compact convex sets in $\reals^d$. Then 
	 for any $1\leq i\leq d+1$ there is exactly one hyperplane, which we denote by $H(\K\setminus \{K_i\},\{K_i\})$, that is tangent to each set $K_j$ with $j\in [d+1]\setminus \{i\}$, and separates $\{K_i\}$ from $\{K_j\mid j\in [d+1]\setminus \{i\}\}$.
\end{lemma}

Lemma \ref{Lemma:SeparatingHyperplane} is an easy corollary of the following property, which was observed by Cappell, Goodman, Pach, Pollack, and Sharir \cite[Theorem 3]{Cappell}.

\begin{theorem}\label{Theorem:Chappell}
	 Let $\K$ be a family of $d$ non-empty compact and strictly convex sets in $\reals^d$, with a partition $\K=\K_1\uplus \K_2$. Then there exist exactly two distinct choices of an oriented common tangent $H$ to the sets of $\K$, both of which satisfy $\bigcup \K_1\subseteq \overline{H^+}$ and $\bigcup \K_2\subseteq \overline{H^-}$.

\end{theorem}

\begin{proof}[Proof of Lemma \ref{Lemma:SeparatingHyperplane}]
We can assume, with no loss of generality, that each set $K_i$, with $1\leq i\leq d+1$ is strictly convex; if this is not the case, it can be replaced by the Minkowski sum of $\conv(K_i)$ and the Euclidean ball of arbitrarily small radius, so that a standard limiting argument applies.

Fix $i\in [d+1]$. By Theorem \ref{Theorem:Chappell}, there exist distinct oriented hyperplanes $H$ and $H'$ so that each of them is tangent to every set $K_j$, for $j\in [d+1]\setminus \{i\}$, and we have that $\bigcup_{j\in [d+1]\setminus \{i\}}K_j\subset H^+\cap (H')^+$.

Let us first assume, for a contradiction, that neither of these hyperplanes separates $\{K_i\}$ from $\{K_j\mid j\in [d+1]\setminus \{i\}\}$. Since the family $\K$ is separated, $K_i$ cannot meet $H_1\cup H_2$; hence, the entire collection $\bigcup_{j\in [d+1]} K_{j}$ lies in the closure of the single cell $H^+\cap (H')^+$ of $\reals^d\setminus (H\cup H')$. (In particular, it is impossible that $H$ and $H'$ correspond to the opposite orientations of the same hyperplane, for then $K_i\subset H^+\cap (H')^+=\emptyset$.)

Suppose, with loss of generality, that $H\cap H'\neq \emptyset$. Consider a continuous rotation of the hyperplane $H(\zeta)$ from $H$ to $H'$, around the $(d-2)$-flat $G=H\cap H'$, so that $H(0):=H$, $H(1):=H'$, and for each $\zeta\in [0,1]$ the hyperplane $H(\zeta)$ lies in the closed double wedge $W:=\overline{(H^+\cap (H')^+)\cup (H^-\cap (H')^-)}$. (If $H$ and $H'$ are parallel, then $G$ lies at infinity in the standard projective extension of $\reals^d$, so the rotation around $G$ amounts to a translation within the slab $H^+\cap (H')^+$.) 
 
 Notice that every hyperplane $H(\zeta)$, with $\zeta\in [0,1]$, meets every set $K_j$, for $j\in [d+1]\setminus \{i\}$. To see this, let $j\in [d+1]\setminus \{i\}$, and $x_j$ and $x'_j$ be points chosen from $H\cap K_j$ and $H'\cap K_j$, respectively.
By the convexity of $K_j$, the segment $x_jx'_j\subseteq K_j$ lies entirely in the closure of $H^+\cap (H')^+$ and, therefore, meets the hyperplane $H(\zeta)$ as it sweeps $H^+\cap (H')^+$ from $H$ to $H'$.

As $W$ is continuously swept by the moving hyperplane $H(\zeta)$, and contains the set $K_i\subseteq H^+\cap (H')^+$, there must exist $\zeta_0\in (0,1)$ so that $H(\zeta_0)\cap K_i\neq \emptyset$. However, then the hyperplane $H(\zeta_0)$ meets all the $d+1$ sets $X_1,\ldots,X_{d+1}$, so the $(d+1)$-family $\F$ cannot be separated.

Lastly, if $\{K_i\}$ is separated by {\it both} hyperplanes $H,H'$ from $\{K_j\mid j\in [d+1]\setminus \{i\}\}$, the we similarly argue that $K_i$ lies in the open cell $H^-\cap H'^-\subset W$, which is too swept by the continuously rotating hyperplane $H(\zeta)$ through $H\cap H'$. (Notice that this time, the hyperplanes $H$ and $H'$ cannot be parallel, or else $H^-\cap H'^-$ would be empty.) Thus, we again derive a hyperplane $H(\zeta_0)$, with $\zeta_0\in (0,1)$, that crosses all the $d+1$ sets $K_1,\ldots,K_{d+1}$, which is contrary to the assumption that the family $\K$ is separated.
\end{proof}

%To better characterize the strongly separated families $\X=(X_1,\ldots,X_{d+1})$, we examine the arrangement of $d+1$ hyperplanes so that each of them is tangent to the convex hulls of some $d$ sets $X_j$ of $\X$, for $j\in [d+1]\setminus \{i\}$, and separates them from the remaining set $X_i$.

{\noindent \bf Definition -- tight families of convex sets.} 
For any finite family $\K$ of sets in $\reals^d$, let us denote
$$
C(\K):=\conv\left(\bigcup_{K\in \K} K\right).
$$

\noindent Let $\K=\{K_1,\ldots,K_{d+1}\}$ be a separated family $\F$ of $d+1$ non-empty compact convex sets in $\reals^d$. 
\begin{enumerate}

\item  We say that the family $\K$ is {\it tight} if $\bigcap_{i=1}^{d+1}C\left(\K\setminus \{K_i\}\right)=\emptyset$ (see Figure \ref{Fig:Tight1} (left)), and we say that $\K$ is {\it loose} otherwise (see Figure \ref{Fig:Split} (left)). 

	\item For each $1\leq i\leq d+1$ we orient  the hyperplane $H_i=H\left(\K\setminus\{K_i\},\{K_i\}\right)$ in Lemma \ref{Lemma:SeparatingHyperplane} so that $H_i^-\supset K_i$ and $\overline{H_i^+}\supset C(\K\setminus \{K_i\})$, and
 refer to these $d+1$ hyperplanes $H_i=H\left(\K\setminus\{K_i\},\{K_i\}\right)$ as the {\it inner tangents of $\K$}.

 \item We denote $\Delta(\K):=\bigcap_{i=1}^{d+1}H_i^-$; see Figure \ref{Fig:Tight1} (right).
 %Accordingly, we say that a separated family $\K$ of $d+1$ sets is {\it tight} (resp., {\it loose}) if it is tight (resp., loose) as an ordered $(d+1)$-family.\footnote{It is immediate to check that the definition is independent of the ordering of the elements in $\K$.}

\end{enumerate}

\begin{figure}
    \begin{center}      
        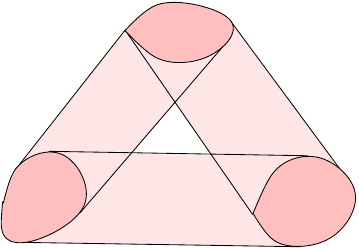\hspace{2cm}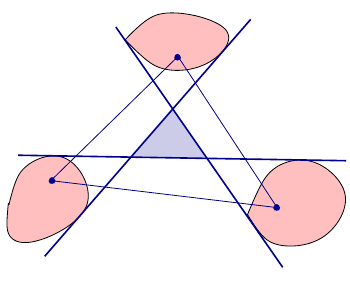%\hspace{2cm}\input{NotSurrounded.pdf_t}
        \caption{\small The convex sets $K_1, K_2$ and $K_3$ comprise a tight triple $\K$ in $\reals^2$.  Left: The convex hulls $C_i=C\left(\K\setminus \{K_i\}\right)$ are depicted for $1\leq i\leq 3$. Note that we have $\bigcap_{i=1}^3C_i=\emptyset$. 
        Right: The simplex $\Delta(\K)=\bigcap H_i^-$ is depicted along with the inner tangents $H_i=H\left(\K\setminus \{K_i\},\{K_i\}\right)$. By Theorem \ref{Theorem:StronglySeparated}, if $\K$ is tight, then $\Delta(\K)$ is non-empty, and is contained in every triangle $\triangle x_1x_2x_3$ with $x_i\in K_i$, for $1\leq i\leq 3$.}
        %Right: Proposition \ref{Prop:NotSurrounded} for $d=3$: The origin $O$ lies in the triangle $\triangle x_1x_2x_3$, for $x_1\in \Delta_1,x_2\in \Delta_3,x_3\in \Delta_3$, yet it is not surrounded by $\Delta_1,\Delta_2,\Delta_3$. The simplices $\Delta_1$ and $\Delta_2$ are crossed by a line through $O$.}
        \label{Fig:Tight1}
    \end{center}
\end{figure}

%\begin{lemma}\label{Lemma:Regular}
%Let $d\geq 2$ be an integer. A $(d+1)$-family $\Sigma=(X_1,\ldots,X_{d+1})$ in $\reals^d$ is separated if and only if 

%$\conv(X_i)$ is disjoint from $C_i(X)$ for each $1\leq i\leq d+1$.
%\end{lemma} 

\begin{theorem}\label{Theorem:StronglySeparated}
	 Let $\K=\{K_1,\ldots,K_{d+1}\}$ be a separated family of $d+1$ compact convex sets in $\reals^d$, so that $K_i\neq \emptyset$ for all $1\leq i\leq d+1$. Then the following statements are equivalent.\footnote{It has come to the authors attention that Theorem \ref{Theorem:StronglySeparated} was independently observed, in part, in the study of Castillo, Doolittle, and Samper \cite{Castillo}. Nevertheless, we supply its full proof for the sake of completeness.}
	 \begin{enumerate}
	 	\item[(i)] $\K$ is tight. 
	 	\item[(ii)] $\Delta(\K)$ is a nonempty open $d$-simplex.
	 	\item[(iii)] The set
	 	$$
	 	\bigcap_{x_1\in K_1,\ldots,x_{d+1}\in K_{d+1}}\conv(x_1,\ldots,x_{d+1})
	 	$$
	 	\noindent has non-empty interior.
	 		
	 		 \end{enumerate}
	 		 
	\noindent  Furthermore, if $\K$ is tight then we have that
	 $$
	\overline{\Delta(\K)}=\bigcap_{x_1\in K_1,\ldots,x_{d+1}\in K_{d+1}}\conv(x_1,\ldots,x_{d+1}).
	 $$

	  \end{theorem}

To facilitate the proof of Theorem \ref{Theorem:StronglySeparated}, we first establish a more elementary property of an arrangement of $d+1$ oriented hyperplanes in general position.

\bigskip
\noindent{\bf Definition.} We say that a $(d+1)$-sequence of points $\sigma=(x_1,\ldots,x_{d+1})\in \reals^{d\times(d+1)}$ is {\it split} by a sequence $(G_1,\ldots,G_{d+1})$ of oriented hyperplanes in $\reals^d$ if (i) $\Delta=\bigcap_{i=1}^{d+1} G_i^-$ is a non-empty open $d$-simplex in $\reals^d$ (in particular, $G_1,\ldots,G_{d+1}$ must be in general position), and (ii) every point $x_i$ lies in the closure of the cone $\bigcap_{j\in [d+1]\setminus \{i\}} G_j^+(\subset G_i^-)$ for all $1\leq i\leq d+1$. See Figure \ref{Fig:Split}.

Accordingly, we say that a sequence $(X_1,\ldots,X_{d+1})$ of $d+1$ sets in $\reals^d$ is {\it split} by a sequence $(G_1,\ldots,G_{d+1})$ of oriented hyperplanes if every selection $(x_1,\ldots,x_{d+1})\in X_1\times\ldots\times X_{d+1}$ is split by $(G_1,\ldots,G_{d+1})$: each set $X_i$ must lie in the closure of the cone $\bigcap_{j\in [d+1]\setminus \{i\}} G_j^+$.

\medskip
The following property was observed by Karasev, Kyn\v{c}l, Pat\'ak, Pat\'akov\'a, and Tancer \cite[Theorem 12]{BoundsPachs}, and is (re-)established below for the sake of compleness.

\begin{lemma}\label{Lemma:StronglySeparatedSystem}
Let $\sigma=(x_1,\ldots,x_{d+1})$ be a $(d+1)$-sequence in $\reals^d$ that is split by a sequence $(G_1,\ldots,G_{d+1})$ of oriented hyperplanes. Then the open $d$-simplex $\Delta=\bigcap_{i=1}^{d+1}G_i^-$ is contained in $\conv(x_1,\ldots,x_{d+1})$.
\end{lemma}

\begin{figure}
    \begin{center}      
      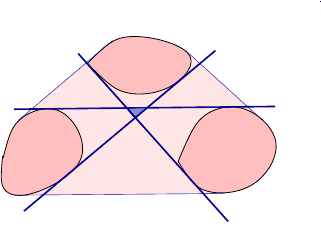\hspace{2cm}  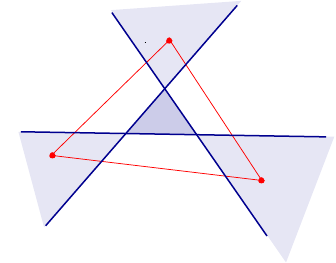
        \caption{\small Left: The family $\K=\{K_1,K_2,K_3\}$ in $\reals^2$ yields $\Delta(\K)=\bigcap_{i=1}^{d+1} H_i^-=\emptyset$ so that, by Theorem \ref{Theorem:StronglySeparated}, $\K$ is loose. Right: The point sequence $\left(x_1,x_2,x_{3}\right)$ in $\reals^2$ is split by the sequence $\left(G_1,G_2,G_{3}\right)$ of hyperplanes. The $2$-simplex (i.e., triangle) $\Delta=\bigcap_{i=1}^{3}G_i^-$ is depicted.}
        %Right: Proposition \ref{Prop:NotSurrounded} for $d=3$: The origin $O$ lies in the triangle $\triangle x_1x_2x_3$, for $x_1\in \Delta_1,x_2\in \Delta_3,x_3\in \Delta_3$, yet it is not surrounded by $\Delta_1,\Delta_2,\Delta_3$. The simplices $\Delta_1$ and $\Delta_2$ are crossed by a line through $O$.}
        \label{Fig:Split}
    \end{center}
\end{figure}

%Since any strongly separated $(d+1)$-family $\X$ is strongly separated by the system $\left(H_1(\X),\ldots,H_{d+1}(\X)\right)$ of its inner hyperplanes, Lemma \ref{Lemma:StronglySeparated} implies Lemma \ref{Lemma:GoodFamilyPierced}.

\begin{proof}[Proof of Lemma \ref{Lemma:StronglySeparatedSystem}]
Clearly, it is necessary and sufficient to establish the claim if every point $x_i$ in the split $(d+1)$-sequence $\sigma$ lies in the {\it open} cone $\bigcap_{j\in [d+1]\setminus \{i\}}G_i^+$.
To this end, we use induction in the dimension $d\geq 1$. 

For $d=1$, we have that $\sigma=(x_1,x_2)$, both hyperplanes $H_i$ in $(G_1,G_2)$ are points separating $x_1$ from $x_2$ and, therefore, piercing the interval $x_1x_2\subset \reals^1$.

%Thus, $H_1(\sigma)$ and $H_2(\sigma)$ are, respectively, the leftmost point of $X_2$ and the rightmost element of $X_1$. Hence, the closed interval $H_2(\Psi)H_1(\Psi)$ is contained in every transversal simplex $\tau=x_1x_2$ so that $(x_1,x_2)\in X_1\times X_2$. 

Now assume that $d\geq 2$.
Denote the vertices of $\Delta=\bigcap_{i=1}^{d+1} G_i^-$ by $v_1,\ldots,v_{d+1}\in \reals^d$ so that $v_i\not\in G_i$ for all $i\in [d+1]$. Up to a relabeling of the points $x_i$ and the hyperplanes $G_i$, it is sufficient to show that the simplex $\tau=\conv\{x_1,\ldots,x_{d+1}\}$ contains the first $d$ vertices $v_1,\ldots,v_d$.

To this end, we identify the hyperplane $G_{d+1}$ with $\reals^{d-1}$ and, for each $1\leq i\leq d$, denote $x'_i=\conv(x_i,v_{d+1})\cap H_{d+1}$, and use $G'_i$ to denote the hyperplane $G_i\cap G_{d+1}$ whose orientation satisfies $(G')_i^+=G_i^+\cap G_{d+1}$. Since each open oriented segment $x_ix_{d+1}$, for $1\leq i\leq d$, crosses only the hyperplanes $G_{d+1}$ and $G_{i}$, and in this order, it follows that the sequence $\sigma'=(x'_1,\ldots,x'_d)$ too is strongly separated in $\reals^{d-1}$ by the system $(G'_1,\ldots,G'_d)$ which attains the vertices $v_i=\bigcap_{j\in [d]\setminus \{i\}} G'_j$. Thus, the induction assumption for dimension $d-1$ readily implies that $\{v_1,\ldots,v_d\}\subseteq \conv(x'_1,\ldots,x'_d)\subseteq \conv(x_1,\ldots,x_{d+1})$.
\end{proof}

 % The reverse inclusion follows from Lemma \ref{Lemma:StronglySeparatedSystem}, as the $(d+1)$-family $\X$ is strongly separated by the system $(H_1(\X),\ldots,H_{d+1}(\X))$.

\begin{proof}[Proof of Theorem \ref{Theorem:StronglySeparated}] Let us denote $\K_i:=\K\setminus \{K_i\}$, $H_i:=H(\K_i,\{K_i\})$, and $C_i:=C(\K_i)$ for all $1\leq i\leq d+1$, so that $\Delta(\K)=\bigcap_{i=1}^{d+1}H^-_i$.

To see that (ii) yields (i), suppose that $\Delta(\K)$ is a non-empty open $d$-dimensional simplex. Lemma \ref{Lemma:ArrangementSimplex} implies, then, that $\bigcap_{i=1}^{d+1} \overline{H_i^+}=\emptyset$ (using that the hyperplanes $H_1,\ldots,H_{d+1}$ are clearly in general position). Since $C_i\subset \overline{H_i^+}$ for all $i\in [d+1]$, it follows that $\bigcap_{i=1}^{d+1}C_i=\emptyset$. Thus, the family $\K$ is tight.

To see that (i) yields (iii), suppose that the family $\K$ is tight, that is, $\bigcap_{i=1}^{d+1} C_i=\emptyset$. 
Since each convex set $C_i$, for $1\leq i\leq d+1$ is the common intersection of all the (closed) halfspaces that contain it, Helly's Theorem yields $d+1$ oriented hyperplanes $G_i$, and $d+1$ indices $j_i\in [d+1]$, for $1\leq i\leq d+1$, so that
\begin{enumerate}
	\item  each closed halfspace $\overline{G_i^+}$ contains the set $C_{j_i}$, and
	\item  $\bigcap \overline{G_i^+}=\emptyset$. 
\end{enumerate}
	 
	  As the intersection of any $d$ among sets $C_i$, for $1\leq i\leq d+1$, is non-empty (for it contains the set $K_j$ that is labeled by the missing index $j$), the $d+1$ indices $j_i$, for $1\leq i\leq d+1$, must be distinct. Thus, it can be assumed up to relabeling that $j_i=i$ for all $1\leq i\leq d+1$.
	  
	  Since $\bigcap_{i=1}^{d+1}\overline{G_i^+}=\emptyset$, Lemma \ref{Lemma:ArrangementSimplex} implies that $\Delta=\bigcap_{i=1}^{d+1}G_i^-$ is a non-empty open simplex. Furthermore, since each subset $K_i$ is contained in 
	  $$
	  \bigcap_{j\in [d+1]\setminus \{i\}}C_j\subset\bigcap_{j\in [d+1]\setminus \{i\}}\overline{G^+_j}
	  $$ 
	  \noindent for each $1\leq i\leq d+1$, the set sequence $(K_1,\ldots,K_{d+1})$ is split by the sequence $(G_1,\ldots,G_{d+1})$.
Thus, Lemma \ref{Lemma:StronglySeparatedSystem} yields that 
	$$
	\emptyset\neq \Delta=\bigcap_{i=1}^{d+1} G_i^-\subseteq \bigcap_{x_1\in K_1,\ldots,x_{d+1}\in K_{d+1}}\conv(x_1,\ldots,x_{d+1}).
	$$

\medskip
To see that (iii) yields (ii), suppose that the family $\K=\{K_1,\ldots,K_{d+1}\}$ yields the intersection

$$
\bigcap_{x_1\in K_1,\ldots,x_{d+1}\in K_{d+1}}\conv(x_1,\ldots,x_{d+1}).
$$

\noindent whose interior is non-empty. As each inner hyperplane $H_i$, for $1\leq i\leq d+1$, meets each convex hull $\conv(K_j)$, with $j\leq [d+1]\setminus \{i\}$, at a point of $K_j$, the closed halfspace $\overline{H_i^-}$, which contains $K_i$, must contain 
at least one simplex of the form $\conv\left(x_1,\ldots,x_{d+1}\right)$, with $x_j\in K_j$ for all $j\in [d+1]$. It, therefore, follows that

\begin{equation}\label{Eq:Delta1}
\bigcap_{x_1\in K_1,\ldots,x_{d+1}\in K_{d+1}}\conv(x_1,\ldots,x_{d+1})\subseteq \overline{\Delta(\K)}=\bigcap_{i=1}^{d+1}\overline{H_i^-}.
\end{equation}

It suffices to show that $\overline{\Delta(\K)}$ is a bounded simplex that is equal to
$\bigcap_{x_1\in K_1,\ldots,x_{d+1}\in K_{d+1}}\conv(x_1,\ldots,x_{d+1})$. To this end, we fix any $(d+1)$-tuple $\left(y_1,\ldots,y_{d+1}\right)\in  K_1\times \ldots \times K_{d+1}$ and observe that 

$$
\emptyset \neq \bigcap_{x_1\in K_1,\ldots,x_{d+1}\in K_{d+1}}\conv(x_1,\ldots,x_{d+1})\subseteq \overline{\Delta(\K)}\cap \conv(y_1,\ldots,y_{d+1}).
$$

\noindent Note that $\Delta(\K)$ cannot intersect any the sides $\conv\left(y_j\mid j\in [d+1]\setminus \{i\}\right)$, with $1\leq i\leq d+1$, which lie in the respective halfspaces $\overline{H_j^+}$. Hence, we have that 
\begin{equation}\label{Eq:Delta2}
	\overline{\Delta(\K)}\subseteq \conv\left(y_1,\ldots,y_{d+1}\right), 
\end{equation}
and, in particular, $\Delta(K)$ is bounded. 
Therefore, Lemma \ref{Lemma:ArrangementSimplex} again implies that $\Delta(\K)$ is an open simplex, so that condition (ii) holds. 

Furthermore, since the choice of $y_1\in K_1,\ldots,y_{d+1}\in K_{d+1}$ was arbitrary, the combination of (\ref{Eq:Delta1}) and (\ref{Eq:Delta2}) implies that
$$
\overline{\Delta(\K)}=\bigcap_{x_1\in K_1,\ldots,x_{d+1}\in K_{d+1}}\conv(x_1,\ldots,x_{d+1}).
$$
\end{proof}

%\begin{lemma}
%	Let $\X=(X_1,\ldots,X_{d+1})$ be a weakly separated $(d+1)$-family in $\reals^d$. Then we have that
%	$\bigcap_{i=1}^{d+1}C_i(\X)\subseteq \Delta^+(\X)$.
%\end{lemma}
%\begin{proof}
	%By definition, we have that $C_i(\X)\subseteq H^+_i(\X)$, so that $C(\X)=\bigcap_{i=1}^{d+1}C_i(\X)\subseteq \bigcap_{i=1}^{d+1}H_i^+=\Delta^+(\X)$.
%\end{proof}

\subsection{Matou\v{s}ek's Simplicial Partition Theorem}\label{Subsec:Simplicial}

\medskip
\noindent{\bf Definition.} Let $P$ be a set of $n$ points in general position in $\reals^d$, and $r>0$ be an integer. A {\it simplicial $r$-partition $\Pi$} of $P$ is a collection $\{(P_i,\Delta_i)\mid 1\leq i\leq r\}$ of $r$ pairs, where for each $1\leq i\leq r$ we have that $P_i\subset P$ and $\Delta_i$ is a simplex of dimension at most $d$ in $\reals^d$, so that the following properties are satisfied:\footnote{For the sake of brevity, we require that each partition encompasses exactly $r$ sets $P_i$, for $1\leq i\leq r$, some of which can be empty. For each empty set $P_i$ we introduce an arbitrary small simplex so that no hyperplane in $\reals^d$ crosses more than $d$ such dummy simplices. %Though a point of $P$ may belong to several simplices $\Delta_i$, for $1\leq i\leq r$, it is assigned to exactly one of their respective sets $P_i$.
} 

\begin{enumerate}
\item $P=\biguplus_{i=1}^r P_i$.

\item For each $1\leq i\leq r$ so that $P_i\neq \emptyset$, the cardinality $n_i:=|P_i|$ of $P_i$ satisfies
$
 \lceil n/r\rceil \leq n_i<2\lceil n/r\rceil.
$
 
\item For each $1\leq i\leq r$, the set $P_i$ is contained in the relative interior of $\Delta_i$.
\end{enumerate}

\medskip
\noindent{\bf Definition.} For each set $P_i$ in an $r$-partition $\Pi=\{(P_i,\Delta_i)\mid 1\leq i\leq r\}$ of $P$, and each point $p\in P_i$, we refer to $\Delta_i$ as the {\it ambient simplex} of $p$, and denote it by $\Delta(p)$.
\footnote{Notice that the simplices in the partition $\Pi$ need not necessarily be pairwise disjoint, or even cover $\reals^d$. Though a point of $P$ may lie in several simplices $\Delta_i$, it is assigned to a unique ambient simplex, by the means of the partition $P_1\uplus\ldots\uplus P_r$ of $P$.}

For any hyperplane in $\reals^d$, we say that a point $p\in P$ {\it lies in the zone of $H$ within $\Pi$} if its ambient simplex is crossed by $H$.

\begin{theorem}[The Simplicial Partition Theorem \cite{PartitionTrees}]\label{Theorem:Simplicial}
For any $d\geq 2$ there is a constant $c(d)$ with the following property.
For any $n$-point set, and any $1\leq r\leq n$, there is a simplicial $r$-partition $\Pi=\{(P_i,\Delta_i)\mid 1\leq i\leq r\}$ so that any hyperplane crosses at most $c(d)  r^{1-1/d}$ of the simplices $\Delta_i$, for $1\leq i\leq r$. 
%Furthermore, such a partition $\Pi$ can be computed in time $\tilde{O}(n)$.\footnote{The $\tilde{O}(x)$ notation hides asymptotic factors that are bounded by $x^{\gamma}$, for arbitrary small $\gamma>0$. The running time was improved by Chan \cite{Chan} to $O(n\log n)$ (with high probability).}
\end{theorem}

If the points of the underlying set $P$ are in a general position, then we can assume that all the simplices $\Delta_i$ in Theorem \ref{Theorem:Simplicial}, are $d$-dimensional; furthermore, their vertices can be perturbed in a general position with respect to one another, and with respect to the point set $P$.

\medskip
\noindent{\bf Definition.} For any $d\geq 2$, any point set $P\subset \reals^{d}$, and any $r>0$, we fix a unique $r$-partition $\Pi=\Pi_d(P,r)$ that meets the criteria of Theorem \ref{Theorem:Simplicial}. We then use $\Sigma_d(P,r)$ to denote the family $\{\Delta_1,\ldots,\Delta_r\}$ of the $r$ simplices that enclose the subsets $P_i$, for $1\leq i\leq r$.

Finally, for each subset $A=\{p_1,\ldots,p_{d+1}\}\in {P\choose d+1}$, we refer to the family 
$$
\Sigma_A=\{\Delta(p_1),\ldots,\Delta(p_{d+1})\}\subseteq \Sigma_d(P,r)
$$ 
\noindent of the ambient simplices $\Delta(p_i)$ of the points $p_i\in A$, as {\it the ambient family of $A$ in $\Sigma_d(P,r)$}.

\medskip
\noindent{\bf Definition.} We say that the subset $A\in {P\choose d+1}$ is {\it crowded in $\Pi_d(P,r)$} if its ambient family $\Sigma_A\subseteq \Sigma_d(P,r)$ has cardinality at most $d$, in which case some pair of its points $p_i,p_j\in A$, with $1\leq i\neq j\leq d+1$, must fall into the same subset $P_i$ of $\Pi_d(P,r)$. Otherwise, we say that $A$ is {\it split} in $\Pi_d(P,r)$.

% which we briefly denote by $\Pi(P,r)$ when the ambient dimension of $P$ is clear from the context.

%\medskip
%

%\begin{corollary}
%Let $\X=(X_1,\ldots,X_{d+1})$ be a weakly separated $(d+1)$-family in $\reals^d$, and let $\Y=(Y_1,Y_2,\ldots,Y_{d+1})$ be a $(d+1)$-family so that $Y_i\subseteq X_i$ for all $1\leq i\leq d+1$. Then $\Y$ is too weakly separated and we have that $\$
%\end{corollary}
%Theorem \ref{Theorem:Main} can be easily deduced from the following two properties.
 \subsection{Semi-algebraic sets and hypergraphs}\label{Subsec:SemiAlgebraic}

Though the hypergraph $(P,E)$ in Theorem \ref{Thm:MainMain} is {\it not} necessarily semi-algebraic, an essential ingredient of its proof is Theorem \ref{Theorem:NewTuran}, which yields a huge complete $k$-partite sub-hypergraph in any sufficiently dense $k$-uniform hypergraph of bounded semi-algebraic description complexity in $\reals^{d}$, for all fixed and positive integers $k$ and $d$. To facilitate the proof of Theorem \ref{Thm:MainMain} in Section \ref{Sec:Main}, as well as the proof of Theorem \ref{Theorem:NewTuran} in Section \ref{Sec:Polynomial}, let us lay down a more comprehensive framework for describing semi-algebraic sets and semi-algebraic hypergraphs.

\medskip
\noindent{\bf Definition.} Let $d$ and $k$ be positive integers.

\begin{enumerate}
\item  A {\it real $d$-variate polynomial $f:\reals^{d}\rightarrow \reals$}, in real variables $x_1,\ldots,x_{d}$, is a function of the form
$$
f(x_1,\ldots,x_{d})=\sum_{i_1,\ldots,i_{d}\in {\mathbb N}}a_{i_1,\ldots,i_{d}}x^{i_1}\cdot \ldots \cdot x^{i_{d}},
$$
\medskip
with real coefficients $a_{i_1,\ldots,i_{d}}$.

In the sequel, we use $\reals[x_1,\ldots,x_{d}]$ denote the space of all such real polynomials.

The {\it degree} of $f\in \reals[x_1,\ldots,x_{d}]$ is $deg(f)=\max\left\{\sum_{j=1}^{d}i_j\mid  a_{i_1,\ldots,i_{d}}\neq 0\right\}$.
Thus, the real polynomials $f:\reals^{d}\rightarrow \reals$ with $deg(f)\leq D$ comprise a vector space of dimension ${d+D\choose d}$ -- the number of possible monomials $x_1^{i_1}\ldots x_{d}^{i_{d}}$ with $0\leq i_1+\ldots +i_{d}\leq D$.

	\item A {\it semi-algebraic description $(f_1,\ldots,f_s;\Phi)$} within $\reals^{d}$ is comprised of a finite sequence $f_1,\ldots,f_s\in \reals[x_1,\ldots,x_{d}]$ of real polynomials, and a Boolean formula $\Phi$ in $s$ variables (where $s$ is also the number of the real polynomials in the sequence). 
	 The {\it complexity} of this description is the pair $(D,s)$, where $D=\max\{deg(f_i)\mid 1\leq i\leq s\}$.

\item A subset $A\subseteq \reals^{d}$ {\it has semi-algebraic  description $(f_1,\ldots,f_s;\Phi)$} if we have that $A=\{x\in \reals^{d}\mid \Phi(f_1(x)\leq 0,\ldots,f_s(x)\leq 0)\}$.

 In what follows, we use $\Gamma_{d,D,s}$ to denote the family of the subsets of $\reals^d$ that admit a semi-algebraic description whose complexity is at most $(D,s)$.

\item Let $(f_1,\ldots,f_s;\Phi)$ be a semi-algebraic description within $\reals^{d\times k}$. We say that a $k$-partite $k$-uniform hypergraph $(V_1,\ldots,V_k,E)$ in $\reals^{d}$
 {\it admits the semi-algebraic description $(f_1,\ldots,f_s;\Phi)$} if  $V_i \subset \reals^{d}$ for all $1\leq i\leq k$, and
any $k$-tuple of points $(p_1,\ldots,p_k)\in P_1\times \ldots\times P_k$ (treated as a coordinate vector in $\reals^{d\times k}$) determines a hyperedge $f=\{p_1,\ldots,p_k\}$ if and only if 
$$
\Phi\left(f_1\left(p_1,\ldots ,p_{k}\right)\leq 0;\ldots;f_s\left(p_1,\ldots ,p_{k}\right)\leq 0\right)=1.
$$

In other words, $E$ is cut out (as a subset of $V_1\times \ldots\times V_k$) by the set $Y\subseteq \reals^{d\times k}$ that meets the description $(f_1,\ldots,f_s;\Phi)$.\footnote{To this end, the vertex sets $V_1,\ldots,V_k$ need not be pairwise disjoint.}

%As was mentioned in the Introduction, this definition naturally extends to the $k$-uniform hypergraphs $(V,E)$ in $\reals^d$ that are not apriori $k$-partite, by insisting that $E=\left\{\{p_1,\ldots,p_k\}\in {V\choose k}\mid \overline{\left[v_1,\ldots,v_k\right]}\subseteq Y\right\}$.

%and semi-algebraic representations $(f_1,\ldots,f_s;\Phi)$ 
%whose induced subsets $A=\{x\in \reals^{d\times k}\mid \Phi(f_1(x)\leq 0,\ldots,f_s(x)\leq 0)\}$ are invariant to any permutation of the $k$ columns.

%\item We say that this description of the $k$-uniform hypergraph $(V_1,\ldots,V_k,E)$ (or $(V,E)$) in $\reals^d$ by $(f_1,\ldots,f_s;\Phi)$ is {\it sharp} if there is $\eta>0$ so that any $\eta$-perturbation $(V'_1,\ldots,V'_k,E')$ %(resp.,  $(V',E')$) 
%too meets the description $(f_1,\ldots,f_s;\Phi)$. (That is, the edge set $E$, treated as a point set in $\reals^{d\times k}$, is contained in the interior of the semi-algebraic set $Y$ which is determined by the description $(f_1,\ldots,f_s;\Phi)$.)

\end{enumerate}

Our proof of Theorem \ref{Theorem:NewTuran} will rely on a certain ``semi-algebraic analogue" of Theorem \ref{Theorem:Simplicial} which was established by Matou\v{s}ek and P\'at\'akova \cite{MultiLevel}.
To this end, we say that a semi-algebraic set $X$ is {\it crossed} by another set $Y$ if we have that $X\cap Y\neq \emptyset$ yet $X\not\subseteq Y$.

\begin{theorem}\label{Theorem:MultiLevel}
For every integers $d\geq 1$, $D\geq 0$, and $s\geq 1$, there exist constants $a=a(d)$, $b=b(d)$ and $c=c(d,D,s)$ such that the following holds. 

For any $n$-point set $P\subset \reals^d$ and any parameter $r>1$, there exist numbers $r_1,r_2,\ldots,r_d\in \left[r,r^{a}\right]$, positive integers $t_1,\ldots,t_d\leq br^b$, and a partition
$$
 P=P^*\uplus\biguplus_{i=1}^d\biguplus_{j=1}^{t_i} P_{i,j},
$$

\noindent along with connected semi-algebraic subsets $X_{i,j}$, for $1\leq i\leq d$ and $1\leq j\leq t_i$, whose respective description complexities are bounded by $\left(br^{b},br^{b}\right)$, so that the following conditions are satisfied:
\begin{enumerate}
    \item For all $1\leq i\leq d$ and $1\leq j\leq t_i$, we have that $P_{i,j}=P\cap X_{i,j}$.
	\item We have that $|P^*|\leq r^{a}$, and $|P_{i,j}|\leq n/r_i$ for all $1\leq i\leq d$ and $1\leq j\leq t_i$.
	\item For any $1\leq i\leq d$, any semi-algebraic set $Y\in \Gamma_{d,D,s}$ crosses at most $c\cdot r_i^{1-1/d}$ among the sets in $\{X_{i,j}\mid 1\leq j\leq t_i\}$.
\end{enumerate}

\end{theorem}

\section{Proof of Theorem \ref{Thm:MainMain}}\label{Sec:Main}

At the heart our proof lie the following properties, of independent interest, of the simplex set $\Sigma_d(P,r)$ of the partition $\Pi_d(P,r)$ in Theorem \ref{Theorem:Simplicial}. 

%(See Section \ref{Section:Concluding} for immediate implications of these results in connection with polynomial regularity, Pach's Theorem \ref{Theorem:Pach}, and the so called same-type lemma of B\'ar\'any and Valtr \cite{BaranyValtr}.)

\begin{theorem} \label{Theorem:Crossed} 
Let $P$ be a set of $n$ points in general position, and $0<r\leq n$ an integer. Then the simplicial partition $\Pi_d(P,r)$ yields only $O\left(r^{d+1-1/d}\right)$ crossed $(d+1)$-size families $\K=\left\{\Delta_{j_1},\ldots,\Delta_{j_{d+1}}\right\}\subseteq \Sigma_d(P,r)$, so that $1\leq i_1<i_2<\ldots<i_{d+1}\leq r$.
\end{theorem}

\begin{theorem}\label{Theorem:Loose} 
Let $P$ be a finite point set in general position, and $r$ a positive integer, that satisfy $r\leq n=|P|$. Then the simplicial $r$-partition $\Pi_d(P,r)=\{(P_i,\Delta_i)\mid 1\leq i\leq r\}$ of $P$ yields only $\displaystyle O\left(r^{d+1-\frac{1}{d^4+d}}\right)$ loose families $\{\Delta_{i_1},\ldots,\Delta_{i_{d+1}}\}\subseteq \Sigma_d(P,r)$ of $d+1$ simplices, with $1\leq i_1,\ldots,i_{d+1}\leq r$. 

%Then at least one of the following conditions is satisfied.
%\begin{enumerate}
	%\item There exist a separated $(d+1)$-family $\P=\left(P_1,P_2,\ldots,P_{d+1}\right)$ of pairwise disjoint sub-sets $P_i\subseteq P$, each of cardinality $n/r\leq |P_i|\leq 2n/r$ and so that $|E\cap \left(P_1\times P_2\times\ldots\times P_{d+1}\right)|=\Omega\left(tn^{d+1}/r^{d+1}\right)$.
	%\item There exists a point $x\in \reals^d$ that pierces at least $\Omega\left(F_d(t\cdot r^{1/d})(n/r)^{d+1}\right)$ simplices in $E$. 
 %\end{enumerate}
\end{theorem}

While Theorem \ref{Theorem:Crossed} is derived via a simple charging argument,\footnote{As a matter of fact, an earlier implicit proof of Theorem \ref{Theorem:Crossed} can be found in a previous study of weak $\eps$-nets by the author; see \cite[Lemma 5.2 with $\eps=1$]{Rubin}.} which assigns the crossed $(d+1)$-tuples to certain $d$-tuples of simplices,
the proof of Theorem \ref{Theorem:Loose} -- the most challenging ingredient in the whole proof Theorem \ref{Thm:MainMain} -- combines the new Tur\'an-type result (Theorem \ref{Theorem:NewTuran}) and Karasev's selection theorem (Theorem \ref{Theorem:Karasev}). 
Put together, Theorems \ref{Theorem:Crossed} and \ref{Theorem:Loose} yield the immediate conclusion.

\begin{corollary}\label{Corol:BoundBadSubsets}
Let $P$ be a set of $n$ points in general position in $\reals^d$, and $0<r\leq n$ an integer. Then 
there exist only $O\left(n^{d+1}/r\right)$ subsets $A\in {P\choose d+1}$ that are crowded in $\Pi_d(P,r)$.
Furthermore, there exist a total of $O\left(n^{d+1}/r^{\frac{1}{d^4+d}}\right)$ split subsets $A\in {P\choose d+1}$ whose ambient  $(d+1)$-size families $\Sigma_A$ within $\Sigma_d(P,r)$ are either crossed or loose.
\end{corollary}

\begin{proof}[Proof of Corollary \ref{Corol:BoundBadSubsets}.] 
By the definition of a crowded subset in Section \ref{Subsec:Simplicial}, the number of such subsets $A\in {P\choose d+1}$ with respect to the partition $\Pi_d(P,r)$ is 
$$
O\left(r^d\left(\frac{n}{r}\right)^{d+1}\right)=O\left(n^{d+1}/r\right).
$$

Furthermore, according to Theorems \ref{Theorem:Crossed} and \ref{Theorem:Loose}, the overall number of the split subsets $A\in {P\choose d+1}$ of the latter two types is 
$$
O\left(r^{d+1-1/d}\left(\frac{n}{r}\right)^{d+1}+r^{d+1-\frac{1}{d^4+d}}\cdot \left(\frac{n}{r}\right)^{d+1}\right)
=O\left(\frac{n^{d+1}}{r^{\frac{1}{d^4+d}}}\right).
$$
\end{proof}

Before proving Theorems \ref{Theorem:Crossed} and \ref{Theorem:Loose}, let us demonstrate how their combination yields, via Corollary \ref{Corol:BoundBadSubsets}, an immediate and dramatic improvement in the selection exponent $\beta_d$; an additional factor of $\eps$ will be ``shaved off" in Section \ref{Subsec:ImprovedRecurrence} through recurrence in the density $|E|/{n\choose d+1}$.

\begin{theorem}\label{Theorem:EasyBound}
	Let $d\geq 2$. Then we have that $F_d(n,\eps)=\Omega\left(\eps^{(d^4+d)(d+1)+1}n^{d+1}\right)$ for any integer $n>0$ and any $\eps>0$.
\end{theorem}
\begin{proof}
Let $(P,E)$ be a $(d+1)$-uniform simplicial hypergraph in $\reals^d$ with $|E|\geq \eps {n\choose d+1}$. %We conveniently treat $(V,E)$ as a $(d+1)$-partite hypergraph $(V,\ldots,V,E')$, so that any hyperedge $f\in E'$ is represented by $(d+1)!$ copies in $(V,\ldots,V,E')$; thus $|E'|\geq (d+1)! \eps {n\choose d+1}$.

Denote $\alpha=d^4+d$. Let $r:=\left\lceil c\eps^{-\alpha}\right\rceil$ where $c>0$ is a suitably large constant that may depend on the dimension $d$, and consider the $r$-partition $\Pi_d(P,r)$. Let $E_1\subseteq E$ denote the subset of all the hyperedges that are crowded in $\Pi_d(P,r)$, and let $E_2\subseteq E$ denote the subset of all the split hyperedges $\tau=\{p_1,\ldots,p_{d+1}\}\in E$ whose ambient families $\Sigma_\tau=\left\{\Delta(p_1),\ldots,\Delta(p_{d+1})\right\}\subseteq \Sigma_d(P,r)$ are either crossed or loose.

By Corollary \ref{Corol:BoundBadSubsets}, we have that
$$
|E_1|+|E_2|=O\left(\frac{n^{d+1}}{r^{\frac{1}{d^4+d}}}\right)=O\left(\frac{n^{d+1}}{r^{1/\alpha}}\right).
$$

\noindent Thus, fixing a large enough constant $c$ guarantees that the cardinality of $E_1\cup E_2$ is smaller than $\eps{n \choose d+1}/2$, so there must remain at least $\eps{n\choose d+1}/2$ hyperedges $\tau=\{p_1,\ldots,p_{d+1}\}$ within $E=E\setminus \left(E_1\cup E_2\right)$ whose vertices determine a tight $(d+1)$-size family $\Sigma_\tau=\left\{\Delta\left(p_1\right),\ldots,\Delta\left(p_{d+1}\right)\right\}\subseteq \Sigma_d(P,r)$.

By the pigeonhole principle, there is a tight $(d+1)$-size family $\K=\left\{\Delta_{i_1},\ldots,\Delta_{i_{d+1}}\right\}$, with $1\leq i_1<\ldots<i_{d+1}\leq r$, so that at least 
$\eps {n\choose d+1}/\left(2{r\choose d+1}\right)$ of the hyperedges $\tau=\{p_1,\ldots,p_{d+1}\}\in E$ satisfy $\Sigma_\tau=\K$.
By Lemma \ref{Lemma:StronglySeparatedSystem}, each of these simplices $\tau=\conv(p_1,\ldots,p_{d+1})$ must contain the closure of $\Delta(\K)\neq \emptyset$. Thus, {\it any} point $x\in \Delta(\K)$ must pierce 
$$
\Omega\left(\frac{\eps n^{d+1}}{r^{d+1}}\right)=\Omega\left(\eps^{\alpha (d+1)+1}n^{d+1}\right)=\Omega\left(\eps^{(d^4+d)(d+1)+1}n^{d+1}\right)
$$ 
\noindent simplices in $E$. 
\end{proof}

	\noindent{\bf Definition.} For any $k$-dimensional simplex $\Delta$ in $\reals^d$, with $0\leq k\leq d$, let $V(\Delta)$
	denote the vertex set of $\Delta$, so that $\Delta=\conv(V(\Delta))$.
	
	For any finite collection $\Sigma$ of simplices of dimension $0\leq k\leq d$, let $V(\Sigma)$ denote the set $\bigcup_{\Delta\in \Sigma}V(\Delta)$ of the at most $(d+1)|\Sigma|$ vertices of these simplices. 

\medskip
The proof of both Theorems \ref{Theorem:Crossed} and \ref{Theorem:Loose} will use the following elementary property of crossed families of $d+1$ simplices which was established in the previous study by the author \cite[Lemma 2.6]{Rubin}; for the sake of completeness we include its proof in Appendix \ref{App:ElementaryRubin}.

\begin{lemma}\label{Lemma:PinHyperplane}
	Let $\{\Delta_1,\ldots,\Delta_{d+1}\}$ be a family of $d+1$ closed $d$-dimensional simplices in $\reals^d$ that are crossed by a single hyperplane. Then there exists such a hyperplane that crosses $\Delta_1,\ldots,\Delta_{d+1}$ and is supported by $d$ distinct vertices, in general position, of $\bigcup_{i=1}^{d+1}V(\Delta_i)$ (which need not necessarily come from distinct simplices).
	\end{lemma}
	
	\begin{figure}
    \begin{center}      
        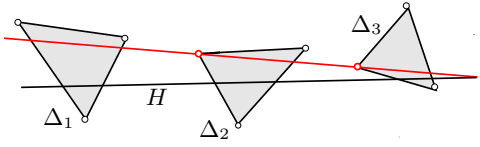%\hspace{2cm}\input{NotSurrounded.pdf_t}
        \caption{\small Lemma \ref{Lemma:PinHyperplane} in dimension $d=2$: moving the hyperplane $H$ crossing the simplices (i.e., triangles) $\Delta_1,\Delta_2$ and $\Delta_3$ to a position in which it contains a pair of vertices of $V(\{\Delta_1,\Delta_2,\Delta_3\})$ while still intersecting each $\Delta_i$. }
        %Right: Proposition \ref{Prop:NotSurrounded} for $d=3$: The origin $O$ lies in the triangle $\triangle x_1x_2x_3$, for $x_1\in \Delta_1,x_2\in \Delta_3,x_3\in \Delta_3$, yet it is not surrounded by $\Delta_1,\Delta_2,\Delta_3$. The simplices $\Delta_1$ and $\Delta_2$ are crossed by a line through $O$.}
        \label{Fig:ExtremalHyperplane}
    \end{center}
\end{figure}

%	\begin{corollary}
	%	A $(d+1)$-size family $\K=\{\Delta_1,\ldots,\Delta_{d+1}\}$ of $d$-simplices $\Delta_i$, each with vertices $\{p_{i,1},\ldots,p_{i,d+1}\}$, is separated if and only if
		%all the simplices of the form $\tau=\left(p_{1,i_1},\ldots,p_{d+1,i_{d+1}}\right)$, with $1\leq i_j\leq d+1$ for all $1\leq j\leq d+1$ (i.e, with one vertex from each simplex $\Delta_i$), have the same orientation.
	%\end{corollary}

%We will establish a stronger claim, namely, that all but some $O\left(r^{d+1-c/d}\right)$ $(d+1)$-families $(\Delta_{i_1},\ldots,\Delta_{i_{d+1}})$ families of the {\it simplices} that enclose the sets $P_{i_j}\subset \Delta_{j_i}$, are strongly separated. 

\subsection{Proof of Theorem \ref{Theorem:Crossed}}\label{Subsec:CrossedProof}
For the sake of brevity, denote $\Sigma:=\Sigma_d(P,r)=\{\Delta_1,\ldots,\Delta_r\}$.
Since the point set $P$ is in general position, it can be assumed in the sequel that the $(d+1)r$ vertices of $V(\Sigma)$ are too in general position.
Let $\K=\{\Delta_{j_1},\ldots,\Delta_{j_{d+1}}\}$ be a crossed family of $d+1$ simplices within $\{\Delta_1,\ldots,\Delta_r\}$, with $1\leq i_1<i_2<\ldots<i_{d+1}\leq r$. Then its elements $\Delta_{j_i}$ can be crossed by a single hyperplane and,
	 furthermore, Lemma \ref{Lemma:PinHyperplane} yields such a hyperplane $H$ that contains a subset $V_H=H\cap V(\K)$ of $d$ vertices in general position. 
	 
	 For each such non-separated family $\K$ we fix a unique such hyperplane with the subset $V_H\subset V(\K)$, along with a subset $\I_\K\subsetneq \K$ of $d$ simplices so that $V_H\subsetneq V(I_\K)$.\footnote{Notice that $\I_\K$ need not be unique, as the vertices of $V_H$ may come from the boundaries of fewer than $d$ simplices.}
	  
	  As there exist only ${r\choose d}=O\left(r^d\right)$ possible $d$-size sub-sets $\I=\I_\K$ within $\Sigma$, it suffices to show that that any of them is shared by $O\left(r^{1-1/d}\right)$ distinct $(d+1)$-size families $\K$ which satisfy $\I=\I_\K$. 
	   Indeed, for each non-separated $(d+1)$-size family $\K$, the only remaining simplex in $\K\setminus \I_\K$ must be crossed by the hyperplane $H=\aff(V_H)$ through some $d$-size subset $V_H=V\left(\I_\K\right)\cap H$. The bound on $|\{\K\in {\Sigma\choose d+1}\mid \I_\K=\I\}|$ now follows since 1. such a subset $V_H$ can be guessed from within $V\left(I_\K\right)$ in ${d(d+1)\choose d}$ ways, and 2. the hyperplane $\aff\left(V_H\right)$ crosses $O\left(r^{1-1/d}\right)$ simplices in $\Sigma$. $\Box$

\subsection{Proof of Theorem \ref{Theorem:Loose}}
We consider the $(d+1)$-uniform hypergraph whose edges correspond to the loose $(d+1)$-size families $\K=\{\Delta_{i_1},\ldots,\Delta_{i_{d+1}}\}$ within $\{\Delta_1,\ldots,\Delta_r\}$, with $1\leq i_1<\ldots<i_{d+1}\leq r$. 

\bigskip
\noindent{\bf Definition.} Let $\Sigma=\{\Delta_1,\ldots,\Delta_r\}$ be a family of $r$ $d$-simplices in $\reals^d$. 

\begin{enumerate}
\item We use $E^*(\Sigma)$ to denote the collection of all the loose $(d+1)$-size sub-families $\K\subseteq \Sigma$.
		\item We say that a loose family $\K=\{\Delta_{i_1},\ldots,\Delta_{i_{d+1}}\}\in E^*(\Sigma)$ is {\it pinned} by a point $x\in \reals^d$ if $x$ lies in 
		
		\begin{equation}\label{Eq:Union}
			C(\K)=\conv\left(\bigcup \K\right)=\bigcup\{\conv(x_1,\ldots,x_{d+1})\mid x_1\in \Delta_{i_1},\ldots,x_{d+1}\in \Delta_{i_{d+1}}\},
		\end{equation}
						
		\noindent which is the union of all the ``colorful" simplices, so that each of their vertices $x_j$ is chosen from a distinct element $\Delta_{i_j}$ of $\K$.\footnote{The last equality in (\ref{Eq:Union}) stems from the convexity of the elements of $\K$, and is not necessarily true, e.g., for families of finite point sets.}
		%$C(\K)=\conv\left(\bigcup_{i=1}^{d+1}\Delta_{j_i}\right)$.
	\item For each $x\in \reals^d$ we use $E^*(\Sigma,x)$ to denote the subset of all the families $\K\in E^*(\Sigma)$ that are pinned by $x$.
\end{enumerate}

Theorem \ref{Theorem:Loose} is an immediate corollary of the following two statements, which yield an upper and a lower bound on the maximum number of  the loose $(d+1)$-size families $\K\in E^*(\Sigma)$ (with $\Sigma=\Sigma_d(\Sigma,r)$) that can be simultaneously pinned by a point $x\in \reals^d$.

\begin{lemma}\label{Theorem:UpperBound} The following statement holds true for any constant $c>0$.

Let $\Sigma=\{\Delta_1,\ldots,\Delta_r\}$ be a family of $r$ closed $d$-dimensional simplices in $\reals^d$ so that the set $V(\Sigma)$ is comprised of $r(d+1)$ vertices in general position. Suppose that any hyperplane crosses at most $cr^{1-1/d}$ simplices in $\Sigma$.
Then any point $x\in \reals^d$ pins only $O\left(r^{d+1-1/d}\right)$ $(d+1)$-size subsets within $\Sigma$, where the constant of proportionality may depend on $c>0$ and the dimension $d$.
\end{lemma}

\begin{lemma}\label{Theorem:LowerBound}
	  		Let $\Sigma=\{\Delta_1,\ldots,\Delta_r\}$ be a set of $r\geq d+1$ closed $d$-dimensional simplices within $\reals^d$ so that the set $V(\Sigma)$ is comprised of $r(d+1)$ distinct vertices in general position. Suppose that $|E^*(\Sigma)|=\varepsilon {r\choose d+1}$ for some $\varepsilon>0$. Then there is a point $x\in \reals^d$ that pins $\Omega\left(\varepsilon^{d^3+1}r^{d+1}\right)$ families $\K\in {\Sigma\choose d+1}$.
\end{lemma}

To establish Theorem \ref{Theorem:Loose} via Lemmas \ref{Theorem:UpperBound} and \ref{Theorem:LowerBound}, let $\Sigma=\Sigma_d(P,r)$. Suppose with no loss generality that $r\geq d+1$, and that $\left|E^*(\Sigma)\right|=\varepsilon {r\choose d+1}$ for some $\varepsilon>0$.
Let $x^*=\arg\max_{x\in \reals^d} |E^*(\Sigma,x)|$ be the point that pins the largest number of hyperedges in $E^*(\Sigma)$. 
By Lemma \ref{Theorem:LowerBound}, there is a point $x\in \reals^d$ that pins $\Omega\left(\varepsilon^{d^3+1}r^{d+1}\right)$ $(d+1)$-size families $\K\in E^*(\Sigma)$. On the other hand, Lemma \ref{Theorem:UpperBound} implies that $|E^*(\Sigma,x^*)|=O\left(r^{d+1-1/d}\right)$. Combining these two estimates on $|E^*\left(\Sigma,x^*\right)|$ yields $\varepsilon^{d^3+1}r^{d+1}=O\left(r^{d+1-1/d}\right)$, or $\varepsilon=O\left(1/r^{\frac{1}{d^4+d}}\right)$.	 $\Box$

\subsection{Proof of Lemma \ref{Theorem:UpperBound}}
%Lemma \ref{Theorem:UpperBound} is established via an instant charging argument similar to Theorem \ref{Theorem:Crossed}. 

We will use the following elementary property.

\begin{lemma}\label{Claim:HyperplaneThroughPoint}
Let $\K=\{\Delta_1,\ldots,\Delta_{d+1}\}$ be a loose family of  $d+1$ $d$-dimensional simplices, so that $V(\K)$ is comprised of $(d+1)^2$ points in general position, and $x\in \reals^d$ a point in $C(\K)=\bigcup\{\conv(x_1,\ldots,x_{d+1})\mid x_1\in \Delta_{1},\ldots,x_{d+1}\in \Delta_{d+1}\}$. There is a hyperplane through $x$ that crosses at least $d$ elements of $\K$.	See Figure \ref{Fig:Pinned}.
\end{lemma}

\begin{figure}
    \begin{center}      
        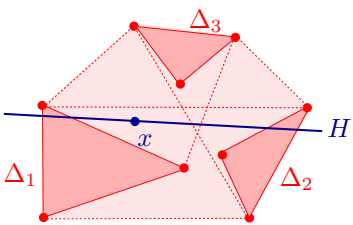%\hspace{2cm}\input{NotSurrounded.pdf_t}
        \caption{\small The loose family $\K=\{\Delta_1,\ldots,\Delta_{d+1}\}$ of simplices in general position within $\reals^d$ is pinned by a point $x\in C(\K)$. By Lemma \ref{Claim:HyperplaneThroughPoint}, there is a hyperplane $H$ through $x$ that crosses at least $d$ elements of $\K$.}
        %Right: Proposition \ref{Prop:NotSurrounded} for $d=3$: The origin $O$ lies in the triangle $\triangle x_1x_2x_3$, for $x_1\in \Delta_1,x_2\in \Delta_3,x_3\in \Delta_3$, yet it is not surrounded by $\Delta_1,\Delta_2,\Delta_3$. The simplices $\Delta_1$ and $\Delta_2$ are crossed by a line through $O$.}
        \label{Fig:Pinned}
    \end{center}
\end{figure}

\begin{proof}[Proof of Lemma \ref{Claim:HyperplaneThroughPoint}.]
We follow an argument that was previously used by the author to establish a more general statement \cite[Lemma 2.4]{Rubin}, and readily extends to all loose families of $d+1$ compact convex sets in $\reals^d$. A different (and, in our opinion) more instructive proof is supplied in Appendix \ref{App:ElementaryRubin}.
	
	Since $x\in C(\K)$, there must exist $y_1\in \Delta_1,\ldots,y_{d+1}\in \Delta_{d+1}$ so that $x\in \conv\left(y_1,\ldots,y_{d+1}\right)$.
On the other hand, since the family $\K=\{\Delta_1,\ldots,\Delta_{d+1}\}$ is loose, we have that 
$$
\bigcap_{x_1\in \Delta_1,\ldots,x_{d+1}\in \Delta_{d+1}}\conv(x_1,\ldots,x_{d+1})=\emptyset.
$$ 

\noindent Hence, there must also exist $y'_1\in \Delta_1,\ldots,y'_{d+1}\in \Delta_{d+1}$ so that $x\not\in \conv\left(y'_1,\ldots,y'_{d+1}\right)$. 

For each $t\in [0,1]$, and each $1\leq i\leq d+1$, we define the point $y_i(t):=(1-t)y_i+ty'_i$ which clearly lies within $\Delta_i$. Denote $\tau(t):=\conv(y_1(t),\ldots,y_{d+1}(t))$ for all $t\in [0,1]$. 
Since $\K$ is loose, it is, in particular, separated, so that $\tau(t)$ is a $d$-dimensional simplex for each $t\in [0,1]$.
As $\tau(0)=\conv\{y_i\mid 1\leq i\leq d+1\}$  contains $x$, whereas $\tau(1)=\conv\{y'_i\mid 1\leq i\leq d+1\}$ no longer does, there must be $t^*\in [0,1]$ so that $x$ lies on the boundary of $\tau(t^*)$. We can assume with no loss of generality that $x$ lies in the closed $(d-1)$-face $\kappa\left(t^*\right)=\conv\left(y_1(t^*),\ldots,y_{d}(t^*)\right)$ of $\tau\left(t^*\right)$. Then the hyperplane $\aff\left(\kappa\left(t^*\right)\right)$ contains $x$, and also intersects $\Delta_1,\ldots,\Delta_{d}$ at the respective points $y_1(t^*),\ldots,y_{d}(t^*)$.
\end{proof}

\begin{proof}[Back to the proof of Lemma \ref{Theorem:UpperBound}.] It can be assumed, with no loss of generality, that $r=|\Sigma|\geq d+1$.
Fix a point $x\in \reals^d$, and
let $\K=\{\Delta_{i_1},\ldots,\Delta_{i_{d+1}}\}$ be a $(d+1)$-size family that is pinned by $x$, so that $x\in C(\K)=\bigcup\{\conv(x_1,\ldots,x_{d+1})\mid x_1\in \Delta_{i_1},\ldots,x_{d+1}\in \Delta_{i_{d+1}}\}$. According to Lemma \ref{Claim:HyperplaneThroughPoint}, there exists a hyperplane $H$ through $x$ that crosses some $d$ of the simplices in $\K$.

Assume with no loss of generality $H$ crosses the first $d$ sets $\Delta_{i_j}$, with $1\leq j\leq d$.
 Applying Lemma \ref{Lemma:PinHyperplane} to the family $\{\Delta_{i_1},\ldots,\Delta_{i_d},\{x\}\}$ yields a hyperplane that passes through $x$ and some subset $V_\K$ of $d-1$ vertices within $\bigcup_{j=1}^d V(\Delta_{i_j})$ which altogether are in general position. We then label $\K$ with a set $\I_\K\subset \K$ of $d-1$ simplices so that $V_\K\subset V\left(\I_\K\right)$. Notice that there exist at most ${r\choose d-1}=O\left(r^{d-1}\right)$ distinct labels $\I_\K\in {\Sigma\choose d-1}$, and the same label $\I\subseteq \Sigma$ can be shared by at most $O\left(r^{2-1/d}\right)$ $(d+1)$-size families $\K=\{\Delta_{i_1},\ldots,\Delta_{i_{d+1}}\}$ with $\I_\K=\I$, as at least one of the remaining 2 elements $\Delta_{i_j}\in \K\setminus \I_\K$ must be selected from the subset of $O\left(r^{1-1/d}\right)$ simplices that are crossed by a hyperplane $\aff(V_\K\cup \{x\})$ through $x$ and some $(d-1)$-size subset $V_\K\in {V(\I)\choose d-1}$. Thus, the family $\K$ can be selected in only $O\left(r^{d-1}\cdot r^{2-1/d}\right)=O\left(r^{d+1-1/d}\right)$ ways.
\end{proof}

\subsection{Proof of Lemma \ref{Theorem:LowerBound}} \label{Subsec:LowerBound}

The crucial observation is that the $(d+1)$-uniform hypergraph $(\Sigma,E^*(\Sigma))$ has a semi-algebraic representation of bounded complexity within the space $\reals^{d\times (d+1)}$ -- the natural representation space of $d$-simplices as ordered sequences of $d+1$ points. Then the desired point $x\in \reals^d$ pinning $\Omega\left(\varepsilon^{d^3+o(d^3)}r^{d+1}\right)$ $(d+1)$-size families $\K$, will be obtained through a rather straightforward combination of Theorem \ref{Theorem:NewTuran} (say, in dimension $d\times (d+1)$) and Karasev's Theorem \ref{Theorem:Karasev}.

To facilitate the use of our efficient Tur\'an-type result -- Theorem \ref{Theorem:NewTuran}, it is more convenient to work with hypergraphs that are $(d+1)$-partite.
Note that the $(d+1)$-uniform hypergraph $(\Sigma,E^*(\Sigma))$ encompasses a $(d+1)$-partite sub-hypergraph $(\Sigma_1,\ldots,\Sigma_{d+1},E^*(\Sigma_1,\ldots,\Sigma_{d+1}))$, where $E^*(\Sigma_1,\ldots,\Sigma_{d+1})$ is comprised of all the (ordered) loose families $\left(\Delta_{i_1},\ldots,\Delta_{i_{d+1}}\right)\in \Sigma_1\times\ldots\times\Sigma_{d+1}$, so that
%\footnote{As a result, the elements of each loose family $\K$ within $E^*(\Sigma_1,\ldots,\Sigma_{d+1})$ can be treated as an ordered {\it sequence}.}

$$
E^*(\Sigma_1,\ldots,\Sigma_{d+1})\geq \frac{(d+1)!}{(d+1)^{d+1}}|E^*(\Sigma)|\geq \frac{(d+1)!}{(d+1)^{d+1}}\cdot \varepsilon\cdot {r\choose d+1}.
$$ 

To see this, consider a random assignment of the simplices of $\Sigma$ into $d+1$ pairwise disjoint classes $\Sigma_1,\ldots,\Sigma_{d+1}$, and notice that the simplices in a given hyperedge $\{\Delta_{i_1},\ldots,\Delta_{i_{d+1}}\}\in E$ fall into $d+1$ distinct classes with probability $(d+1)!/(d+1)^{d+1}$. Clearly, each class $\Sigma_i$ in this subdivision would satisfy
$$
|\Sigma_i|\geq \frac{(d+1)!}{(d+1)^{d+1}}|\Sigma|=\frac{(d+1)!}{(d+1)^{d+1}}r.
$$

\medskip
As the exponent in Theorem \ref{Theorem:NewTuran} depends on the ambient dimension of the vertices, it pays to first reduce the representation size of our simplices (which comprise the vertices of our hypergraph).

\medskip
\noindent{\bf Definition.} For any $d$-dimensional simplex $\Delta$, let $\Sigma'(\Delta)$ denote the set of the ${d+1\choose d}=d+1$ closed $(d-1)$-dimensional simplices that comprise the boundary complex of $\Delta$.

\medskip
For each $1\leq i\leq d+1$, let us ``replace" the collection $\Sigma_i$ with the set $\Sigma'_i:=\bigcup_{\Delta\in \Sigma_i}\Sigma'(\Delta)$ of $(d+1)|\Sigma_i|$ $(d-1)$-dimensional simplices.\footnote{Though no pair of simplices in $\Sigma$ share a vertex, each simplex $\Delta\in \Sigma_i$ adds to $\Sigma'_i$ a subset $\Sigma'(\Delta)$ of $d+1$ simplices so that any $2$ of them share $d-1$ vertices, and any $d$ of them share a common vertex.}  Accordingly, instead of $E^*(\Sigma_1,\ldots,\Sigma_{d+1})$ we consider the set $E^*\left(\Sigma'_1,\ldots,\Sigma'_{d+1}\right)$ which is comprised of all the loose families $(\sigma_1,\ldots,\sigma_{d+1})\in \Sigma'_1\times\ldots\times \Sigma'_{d+1}$.%The subdivision $\Sigma=\Sigma'_1\uplus\ldots\uplus \Sigma'_{d+1}$ of $\Sigma'$, determines the subset $E(\Sigma'_1,\ldots,\Sigma'_{d+1})\subseteq E(\Sigma')$ which is comprised of all the loose families $\K=\{\sigma_1,\ldots,\sigma_{d+1}\}$ with $\sigma_i\in \Sigma'_i$ for all $1\leq i\leq d+1$.
% (Note that no pair of simplices $\sigma_i,\sigma_j$ within the same loose $(d+1)$-family $\left\{\sigma_1,\ldots,\sigma_{d+1}\right\}$ of $E(\Sigma'_1\times\ldots\times \Sigma'_{d+1})$ can come from the boundary of the same $d$-simplex of $\Sigma$ or, else, such a family $\K$ would be crossed.)

\medskip
The following property guarantees that the resulting hypergraph $\left(\Sigma'_1,\ldots,\Sigma'_{d+1},E^*\left(\Sigma'_1,\ldots,\Sigma'_{d+1}\right)\right)$ encompasses at least $\frac{(d+1)!}{(d+1)^{d+1}}\cdot \varepsilon\cdot {r\choose d+1}$ hyperedges.

\begin{proposition}
Let $\K$ be a $(d+1)$-size loose family $\K=(\Delta_{i_1},\ldots,\Delta_{i_{d+1}})\in E^*(\Sigma_1,\ldots,\Sigma_{d+1})$. Then there is a loose $(d+1)$-size family $\K'=(\sigma_1,\ldots,\sigma_{d+1})$ in $E^*(\Sigma'_1,\ldots,\Sigma'_{d+1})$, so that $\sigma_{j}\in \Sigma'\left(\Delta_{i_j}\right)$ for all $1\leq j\leq d+1$.

\end{proposition}

\begin{proof}
Every simplex $\Delta_{i_j}\in \K$, with $1\leq j\leq d+1$, supports exactly $d$ of the oriented inner tangents $H_k:=H\left(\K\setminus \{\Delta_{i_k}\},\{\Delta_{i_k}\}\right)$, namely, those with $k\in [d+1]\setminus \{j\}$; see Section \ref{Subsec:PrelimTight}. 

Hence, there exist $d$-size subsets $V_1\subset V\left(\Delta_{i_1}\right),\ldots,V_{d+1}\subset V\left(\Delta_{i_{d+1}}\right)$ so that each hyperplane $H_j$ is tangent to every $(d-1)$-simplex $\sigma_k=\conv(V_k)$, with $k\in [d+1]\setminus\{k\}$. Since the vertices of $V(\Sigma)$ are in general position, the families $\K$ and $\K'=\{\sigma_1,\ldots,\sigma_{d+1}\}$ have the same set of oriented inner tangents (namely, $H_1,\ldots,H_{d+1}$, with $H_j^-\supset \Delta_{i_j}\supset \sigma_j$ for all $1\leq j\leq d+1$) which satisfy $\Delta(\K)=\bigcap_{j=1}^{d+1} H_j^-=\emptyset$ according to Theorem \ref{Theorem:StronglySeparated}. The claim now follows by applying Theorem \ref{Theorem:StronglySeparated} to the set $\K'$.
\end{proof}

%\medskip
%\noindent{\bf Semi-algebraic hypergraphs.} We say that a $k$-uniform $k$-partite hypergraph $(V_1\uplus\ldots\uplus V_k,E)$ is {\it semi-algebraic with complexity at most $(s,D)$ in $\reals^d$} if $V=V_1\uplus\ldots\uplus V_k\subseteq \reals^d$, and there exist polynomial $f_1,\ldots,f_s\in \reals[x_1,\ldots,x_{kd}]$ with a Boolean function $\Phi$ such that
%a $k$-tuple of points $(p_1,\ldots,p_k)\in P_1\times \ldots\times P_k$ determines a hyperedge $f=\{p_1,\ldots,p_k\}$ if and only if 
%$$
%\Phi\left(f_1\left(p_1,\ldots,p_k\right)\geq 0;\ldots;f_s\left(p_1,\ldots,p_k\right)\geq 0\right)=1.
%$$

Though each family $\Sigma'_i$ contains many pairs of simplices that share at least one vertex, an arbitrary small perturbation \cite{GeneralPosition} guarantees that the perturbed set $V\left(\biguplus_{i=1}^{d+1}\Sigma'_i\right)$ encompasses $d\left|\biguplus_{i=1}^{d+1}\Sigma'_i\right|$ vertices in general position. To this end, each closed $(d-1)$-dimensional simplex $\sigma\in \biguplus_{i=1}^{d+1}\Sigma'_i$ is shrunk away from the boundary of its ``parent" $d$-dimensional simplex $\Delta\in \biguplus_{i=1}^{d+1} \Sigma_i$, as each vertex $v$ of $\sigma$ is replaced by a generic point $v'$ in the $\eta$-vicinity of $v$ (in the $L_\infty$-norm) yet within the old $\conv(\sigma)$. By Theorem \ref{Theorem:Loose}, each loose family $\K'=\{\sigma_1,\ldots,\sigma_{d+1}\}\in E^*(\Sigma'_1,\ldots,\Sigma'_{d+1})$ yields a non-empty open simplex $\bigcap_{i=1}^{d+1}H^+_i$, while a suitably small choice of $\eta>0$ guarantees that no loose families in $E^*(\Sigma'_1,\ldots,\Sigma'_{d+1})$ cease to be loose as a result of this shrinking.

\medskip
The following lemma yields a semi-algebraic description of bounded complexity within $\reals^{d^2}$ for the perturbed hypergraph
$(\Sigma'_1,\ldots,\Sigma'_{d+1},E^*(\Sigma'_1,\ldots,\Sigma'_{d+1}))$.

\begin{lemma}\label{Lemma:SemiAlgebraic} %For any dimension $d\geq 2$, there exist constants $D=D(d)$ and $s=s(d)$ with the following property.
Let $d\geq 2$, and $\Sigma'_1,\ldots,\Sigma'_{d+1}$ be pairwise disjoint and finite collections of $(d-1)$-simplices within $\reals^d$, so that the vertex set $V(\Sigma')$ of their union $\Sigma':=\biguplus_{i=1}^{d+1}\Sigma'_i$ is comprised of $d|\Sigma'|$ points in general position.\footnote{In particular, no pair of simplices in $\Sigma'$ can share a vertex, and no pair of vertices in $V(\Sigma')$ can share a coordinate.}
Denote the coordinates of a point $x\in \reals^d$ by $(x(1),\ldots,x(d))$, and
suppose that each $(d-1)$-simplex $\sigma=\conv\left\{v_1,\ldots,v_{d}\right\}\in \Sigma'$ is represented by the point\footnote{To make the representation unique, we can require that $v_1(1)<v_2(1)<\ldots<v_{d+1}(1)$.} 
$$
(v_1(1),v_1(2),\ldots,v_1(d);\ldots;v_{d}(1),\ldots,v_{d}(d))\in \reals^{d^2}.
$$

Then the $(d+1)$-partite hypergraph $\left(\Sigma'_1,\ldots,\Sigma'_{d+1},E\left(\Sigma'_1,\ldots,\Sigma'_{d+1}\right)\right)$ in $\reals^{d^2}$ admits a semi-algebraic representation whose complexity is bounded in $d$.
\end{lemma}

\noindent{\bf Proof of Lemma \ref{Theorem:LowerBound} -- wrap up.} Applying Theorem \ref{Theorem:NewTuran} to the $(d+1)$-uniform hypergraph at hand 
$$
(\Sigma'_1,\ldots,\Sigma'_{d+1},E^*(\Sigma'_1,\ldots,\Sigma'_{d+1}))
$$ 
\noindent {within $\reals^{d^2}$}, yields subsets $\Sigma''_1\subseteq \Sigma'_1,\ldots, \Sigma''_{d+1}\subseteq \Sigma'_{d+1}$ that satisfy
$$
|\Sigma''_1|\cdot|\Sigma''_2|\cdot\ldots\cdot |\Sigma''_{d+1}|=\Omega\left(\varepsilon^{d^3+1}\cdot |\Sigma'_1|\cdot|\Sigma'_2|\cdot\ldots\cdot |\Sigma'_{d+1}|\right)=\Omega\left(\varepsilon^{d^3+1}r^{d+1}\right).
$$

	  Let us choose a point $x_\sigma$ in every simplex $\sigma\in \bigcup_{1\leq i\leq d+1}\Sigma''_{d+1}$, and denote $X_i=\left\{x_\sigma\mid \sigma\in \Sigma''_i\right\}$. %Fix a point $x_i$ in each $\Delta_i$.
	  According to Theorem \ref{Theorem:Karasev}, there exists a point $x\in \reals^d$ that lies in at least
	  
	  $$
	  \frac{1}{(d+1)!}|X_1|\cdot\ldots\cdot |X_{d+1}|=  \frac{1}{(d+1)!}|\Sigma''_1|\cdot\ldots\cdot|\Sigma''_{d+1}|=\Omega\left(\varepsilon^{d^3+1}r^{d+1}\right)
	  $$ 
	among the ``colorful" simplices $\tau\in X_1\times\ldots\times X_{d+1}$.
	  
	  %Hence, it suffices to show that each of these simplices $\tau\in X_1\times\ldots\times X_{d+1}$, that is pierced by $x$, yields a distinct loose family $\K'\in E^*\left(\Sigma'_1,\ldots,\Sigma'_{d+1}\right)$ whose ``parent" family $\K\in E^*\left(\Sigma_1,\ldots,\Sigma_{d+1}\right)$ is loose and pinned by $x$.
	  
Fix a ``colorful" simplex $\tau=\{x_{\sigma_1},\ldots,x_{\sigma_{d+1}}\}$ in $X_1\times \ldots\times X_{d+1}$ that contains $x$. Then each of its vertices $x_{\sigma_i}\in X_i$ belongs to a $(d-1)$-simplex $\sigma_i\in \Sigma''_i$, which lies on the boundary of a unique $d$-simplex $\Delta_{j_i}\in \Sigma_i$. 
Let $\K'=(\sigma_1,\ldots,\sigma_{d+1})$ and $\K=(\Delta_{j_1},\ldots,\Delta_{j_{d+1}})$.
Then we clearly have that $x\in C(\K')\subseteq C(\K)$; hence, the family $\K$ is pinned by $x$.

%As the family $\K'=\{\sigma_1,\ldots,\sigma_{d+1}\}$ belongs to $\Sigma''_1\times \ldots\times \Sigma''_{d+1}\subseteq E^*\left(\Sigma'_1,\ldots,\Sigma'_{d+1}\right)$, it is loose. Then the characterization of Theorem \ref{Theorem:StronglySeparated} implies that the ``parent" family $\K=\{\Delta_{j_1},\ldots,\Delta_{j_{d+1}}\}$ of $\K'$, in $E^*(\Sigma_1,\ldots,\Sigma_{d+1})$, must be loose as well, as each simplex $\sigma_{i}\in \K'$ is replaced by an even larger simplex $\Delta_{j_i}\supset \sigma_i$ within $\K$. Lastly, the choice of $x\in \conv\left(x_{\sigma_1},\ldots,x_{\sigma_{d+1}}\right)$, and of the points $x_{\sigma_i}\in \sigma_i\subseteq \Delta_{j_i}$, guarantees that 
%	  $x\in C(\K)$; that is, the family $\K\in E^*(\Sigma_1,\ldots,\Sigma_{d+1})$ is pinned by $x$.
	  
	   Lastly, note that any pinned family $\K=(\Delta_{j_1},\ldots,\Delta_{j_{d+1}})$ can be ``reached" via at most $(d+1)^{d+1}$ subsidiary loose families $\K'=(\sigma_1,\ldots,\sigma_{d+1})$ in $E^*(\Sigma'_1,\ldots,\Sigma'_{d+1})$.
	  Hence, repeating the above argument for at least $\left(1/(d+1)!\right)|X_1|\cdot \ldots\cdot |X_{d+1}|=|\Sigma''_1|\cdot\ldots\cdot |\Sigma''_{d+1}|$ ``colorful" simplices $\tau\in X_1\times\ldots\times X_{d+1}$ that contain $x$, yields at least
	  
	  $$
	  \frac{|X_1|\cdot\ldots\cdot |X_{d+1}|}{(d+1)!(d+1)^{d+1}}=\Omega\left(|\Sigma''_1|\cdot\ldots\cdot |\Sigma''_{d+1}|\right)=\Omega\left(\varepsilon^{d^3+1}r^{d+1}\right)
	  $$
	  
	  \noindent {\it distinct} loose families $\K=\{\Delta_{j_1},\ldots,\Delta_{j_{d+1}}\}$ in ${\Sigma\choose d+1}$ that are pinned by $x$.
	  $\Box$

\medskip

\begin{proof}[Proof of Lemma \ref{Lemma:SemiAlgebraic}.] 
For any sequence of $d+1$ points $(x_1,\ldots,x_{d+1})\in \reals^{d\times(d+1)}$ we set 

\begin{equation*}
f(x_1,\ldots,x_{d+1}):={\sf det}
\begin{pmatrix}
    1 & 1 & \cdots & 1\\
	x_{1}(1) & x_{2}(1)& \cdots & x_{d+1}(1)\\
	x_{1}(2)& x_{2}(2)& \cdots & x_{d+1}(2)\\
	\vdots & \vdots & \vdots & \vdots\\
	x_{1}(d)& x_{2}(d)& \cdots & x_{d+1}(d)\\
\end{pmatrix}.
\end{equation*}

Notice that a hyperplane $H=\aff(x_1,\ldots,x_{d})$ through $d$ points $x_1,\ldots,x_{d}$ in general position in $\reals^d$ separates a pair of points $y,z\in \reals^d\setminus H$ if and only if the sequences $(x_1,\ldots,x_d,y)$ and $(x_1,\ldots,x_d,z)$ have opposite (and non-zero) orientations \cite{WellSeparation,PolWen,WengerProgress}
$$
\chi(x_1,\ldots,x_d,y)=\sign f(x_1,\ldots,x_d,y)\neq \sign f(x_1,\ldots,x_d,z)=\chi(x_1,\ldots,x_d,z).
$$

Let $\Sigma'=\Sigma'_1\uplus\ldots\uplus\Sigma'_{d+1}$ be a family of $(d-1)$-simplices as prescribed in the lemma. Fix a $(d+1)$-size sub-family $\K=\{\sigma_1,\ldots,\sigma_{d+1}\}$, with $\sigma_i=\conv\{v_{i,1},\ldots,v_{i,d+1}\}\in \Sigma'_i$ for all $1\leq i\leq d+1$.
Since the $|\Sigma|(d+1)$ vertices of $V\left(\Sigma'\right)$ are in general position, Theorem \ref{Theorem:Loose} implies that $\K$ is loose if and only if the following two conditions are satisfied.

\begin{enumerate}
	\item $\K$ is separated, which happens if any only if the sign of $f(x_1,\ldots,x_{d+1})$ is invariant over all choices $(x_1,\ldots,x_{d+1})\in \sigma_1\times\ldots\times \sigma_{d+1}$.  
	\item The set $\Delta(\K)=\bigcap_{i=1}^{d+1} H^-(\K\setminus \{\sigma_i\},\{\sigma_i\})$ is either empty or unbounded, which according to Lemma \ref{Lemma:ArrangementSimplex} is equivalent to the property that $\bigcap_{i=1}^{d+1} H^+(\K\setminus \{\sigma_i\},\{\sigma_i\})\neq \emptyset$.
\end{enumerate}

Since the function $f(x_1,\ldots,x_{d+1})$ is affine in point $x_i\in \reals^d$ of the sequence $(x_1,\ldots,x_{d+1})\in \reals^{d\times (d+1)}$, the first condition is satisfied if and only if the sign of the polynomials\footnote{To this end, we regard each $g_{j_1,\ldots,j_{d+1}}$ as a polynomial in the $d'=d^2\times (d+1)$ coordinates $x_1,\ldots,x_{d'}$ of the sequence $(\sigma_1,\ldots,\sigma_{d+1})$.} $g:\reals^{d^2\times (d+1)}\rightarrow \reals$

$$
g_{j_1,\ldots,j_{d+1}}(\sigma_1,\ldots,\sigma_{d+1}):=f\left(v_{1,j_1},\ldots,v_{d+1,j_{d+1}}\right)
$$ 
is invariant over all the $d^{d+1}$ possible choices of $1\leq j_i\leq d$ for $1\leq i\leq d+1$.
Hence, there is a Boolean function $\Phi_1$ in $d^{d+1}$ variables so that the first condition is equivalent to the combination

$$
\Psi_1(\sigma_1,\ldots,\sigma_{d+1}):=\Phi_1\left(g_{j_1,\ldots,j_{d+1}}(\sigma_1,\ldots,\sigma_{d+1}): 1\leq j_1,\ldots,j_{d+1}\leq d\right)
$$

\noindent of $d^{d+1}$ polynomial inequalities $g_{j_1,\ldots,j_{d+1}}(\sigma_1,\ldots,\sigma_{d+1})\leq 0$, with $1\leq j_1,\ldots,j_{d+1}\leq d$, each of degree $d$.\footnote{Since the vertices of $V(\Sigma')$ are in general position, an equality $f\left(v_{1,i_1},\ldots,v_{d+1,i_{d+1}}\right)=0$ is never attained.}

%Furthermore, each of these inequalities involves that is at most linear in each  coordinate $p_{i,j}(k)$ of each oriented simplex $\hat{\Delta}_i=(p_{i,1},\ldots,p_{i,d+1})$ (and none of these polynomials involves two such coordinates that come from the same simplex). 
 %Hence, the hypergraph of the separated $(d+1)$-families has semi-algebraic complexity $\left(d,d^{d+1}\right)$.\footnote{}

In other words, the families $\K$ that meet the first condition constitute a semi-algebraic set 
$$
Y_1=\left\{y\in \reals^{d^2\times (d+1)}\mid \Psi_1\left(\sigma_1,\ldots,\sigma_{d+1}\right)\right\},
$$
\noindent whose description complexity is bounded by $\left(d,d^{d+1}\right)$.

For the second property, suppose that the family $\K$ is separated, and denote $H_i:=H(\K\setminus \{\sigma_i\},\{\sigma_i\})$ for all $1\leq i\leq d+1$. 
For any $1\leq i\leq d+1$, each assignment $\phi_i:[d+1]\setminus \{i\}\rightarrow [d]$ determines the unique (non-oriented) hyperplane 
$$
H_{\phi_i}:=\aff\left\{v_{j,\phi_i(j)}\mid j\in [d+1]\setminus \{i\}\right\},
$$ 

\noindent which coincides with $H_i$ if and only if it separates every vertex $v_{i,h}$ of $\sigma_i$ from every vertex $v_{k,l}$ of $\sigma_{k}$, with $k\in [d+1]\setminus \{i\}$ and $l\in [d]\setminus \phi_i(k)$.
Equivalently, the polynomials $g_{i,\phi_i,h},g_{i,\phi_i,k,l}:\reals^{d^2\times (d+1)}\rightarrow \reals$
 $$
g_{i,\phi_i,h}(\sigma_1,\ldots,\sigma_{d+1}):=f\left(v_{1,\phi_i(j)},\ldots,v_{i-1,\phi_i(i-1)},v_{i,h},v_{i+1,\phi_i(i+1)},\ldots,v_{d+1,\phi_i(d+1)}\right)
 $$ and  
 $$
g_{i,\phi_i,k,l}(\sigma_1,\ldots,\sigma_{d+1}):=f\left(v_{1,\phi_i(j)},\ldots,v_{i-1,\phi_i(i-1)},v_{k,l},v_{i+1,\phi_i(i+1)},\ldots,v_{d+1,\phi_i(d+1)}\right)
 $$ 
 must attain opposite signs over $(\sigma_1,\ldots,\sigma_{d+1})\in \reals^{d^2\times (d+1)}$ for all $1\leq h\leq d$, $k\in [d+1]\setminus \{i\}$, and $l\in [d]\setminus \{\phi_i(k)\}$. 
 Therefore, the property that $H_i=H_{\phi_i}$, is determined, for any given $1\leq i\leq d+1$, by a Boolean combination of $d$ inequalities of the form $g_{i,\phi_i,h}(\sigma_1,\ldots,\sigma_{d+1})\leq 0$, and $d(d+1)$ inequalities of the form $g_{i,\phi_i,k,l}(\sigma_1,\ldots,\sigma_{d+1})\leq 0$.
 
 Lastly, let $\phi_1,\ldots,\phi_{d+1}$ be $d+1$ assignments, each of the form $\phi_i:[d+1]\setminus\{i\}\rightarrow [d]$, so that the non-oriented hyperplane $H_{\phi(i)}$ coincides with $H_i$.
Then a point $x\in \reals^d$ lies in $H^+_i$ if and only if both points $x$ and $v_{i,1}$ lie to the same side of $H_{\phi(i)}$. Equivalently, the functions $g_{i,\phi_i,1}:\reals^{d^2\times(d+1)}\rightarrow \reals$ and $z_{i,\phi_i}(\sigma_1,\ldots,\sigma_{d+1};x):\reals^{d^2\times(d+1)+d}\rightarrow \reals$, with
 $$
z_{i,\phi_i}(\sigma_1,\ldots,\sigma_{d+1};x):= f\left(v_{1,\phi_i(1)},\ldots,v_{i-1,\phi_i(i-1)},x,v_{i+1,\phi_i(i+1)},\ldots,v_{d+1,\phi_i(d+1)}\right),
 $$ 
must attain identical signs. 

By repeating the above argument for all $1\leq i\leq d+1$, and all the $d^d$ possible assignments $\phi_i:[d+1]\setminus \{i\}\rightarrow [d]$, the property that $\K=\left\{\sigma_1,\ldots,\sigma_{d+1}\right\}$ is loose can be expressed as
\begin{equation}
	\Psi_1(\sigma_1,\ldots,\sigma_{d+1})\wedge[\exists x\in \reals^d: \Psi_2(\sigma_1,\ldots,\sigma_{d+1};x)].
\end{equation}

\noindent Here
\begin{equation}
	\Psi_2\left(\sigma_1,\ldots,\sigma_{d+1};x\right):=
\end{equation}
\begin{equation*}
\Phi_2\begin{pmatrix}
g_{i,\phi_i,h}\left(\sigma_1,\ldots,\sigma_{d+1}\right)\leq 0: & i\in [d+1], \phi_i:[d+1]\setminus \{i\}\rightarrow [d], h\in [d];\\
g_{i,\phi_i,k,l}\left(\sigma_1,\ldots,\sigma_{d+1}\right)\leq 0: & i\in [d+1], \phi_i:[d+1]\setminus \{i\}\rightarrow [d], k\in [d+1]\setminus \{i\},l\in [d]\setminus\{\phi_i(k)\};\\
z_{i,\phi_i}\left(\sigma_1,\ldots,\sigma_{d+1};x\right)\leq 0: & i\in [d+1], \phi_i:[d+1]\setminus \{i\}\rightarrow [d]
\end{pmatrix}
\end{equation*}

 \noindent is a Boolean combination of $(d+1)d^d\left(d+d(d-1)+1\right)=d^{O(d)}$ polynomial inequalities of the form $g_{i,\phi_i,h}(\sigma_1,\ldots,\sigma_{d+1})\leq 0,g_{i,\phi_i,k,l}(\sigma_1,\ldots,\sigma_{d+1})\leq 0$ or $z_{i,\phi_i}(\sigma_1,\ldots,\sigma_{d+1};x)\leq 0$, each of degree $d$, in the $d^2\times(d+1)$ coordinates that we use to represent the sequence $(\sigma_1,\ldots,\sigma_{d+1})$ of $d+1$ $(d-1)$-dimensional simplices.
 
 \medskip
 By the Tarski-Seidenberg theorem \cite{BasuBook}, the subset 
 $$
Y_2:= \{\left(y_1,\ldots,y_{d+1}\right)\in \reals^{d^2\times (d+1)}\mid \exists x\in \reals^d: \Psi_2(y_1,\ldots,y_{d+1};x)\},
 $$ 
 \noindent which describes the loose families of $d+1$ $(d-1)$-dimensional simplices, is semi-algebraic.\footnote{Using the singly-exponential quantifier elimination \cite[Theorem 2.27]{BasuSurv} (see also \cite[Proposition 2.6]{EstherSearching}), the expression $[\exists x\in \reals^d: \Psi_2(y_1,\ldots,y_{d+1};x)]$ can be replaced by a quantifier-free Boolean combination $\Psi'_2(y_1,\ldots,y_{d+1})$ of $d^{O(d^2)}\cdot d^{O(d^3(d+1))}=d^{O(d^4)}$ polynomial inequalities in the variables $y_1,\ldots,y_{d+1}
 \in \reals^{d^2}$, each of degree at most $d^{O(d)}$.} As a result, the hypergraph $(\Sigma'_1,\ldots,\Sigma'_{d+1},E^*(\Sigma'_1,\ldots,\Sigma'_{d+1}))$, whose edge set $E^*(\Sigma'_1,\ldots,\Sigma'_{d+1})$  is described by $Y_1\cap Y_2$, admits a semi-algebraic description that is bounded in $d$. \end{proof}

	\subsection{A (slightly) better bound for $F_d(n,\eps)$}\label{Subsec:ImprovedRecurrence}

	Denote 
	$$
	G_d(\eps)=\inf_{n\geq d+1}\frac{F_d(n,\eps)}{{n\choose d+1}}.
	$$
	
	\noindent Namely, $G_d(\eps)$ is the largest possible number $0\leq G\leq 1$ with the following property: {\it For any $(d+1)$-uniform simplicial hypergraph $(V,E)$ within $\reals^d$ with $|E|=\eps{n\choose d+1}$ edges, and whose $n\geq d+1$ vertices are in general position, there is a point $x\in \reals^d$ piercing at least $G{n\choose d+1}$ simplices of $E$.}
	
	\medskip
	Let $\delta>0$ be an arbitrary small constant.
	In what follows, we amplify the argument of Theorem \ref{Theorem:EasyBound} so as to show that
	
	$$
	G_d(\eps)=\Omega\left(\eps^{(d^4+d)(d+1)+\delta}\right).
	$$

	Let $n\geq d+1$, and $(P,E)$ be a $(d+1)$-uniform geometric hypergraph in $\reals^d$, where $P\subset \reals^d$ is an $n$-point set, $E\subseteq {P\choose d+1}$, and $|E|\geq \eps {n\choose d+1}$. Similar to the proof of Theorem \ref{Theorem:EasyBound}, we consider the simplicial $r$-partition $\Pi_d(P,r)$ of the vertex set $P$ as described in Theorem \ref{Theorem:Simplicial}. However, now $r=\lceil 1/\eps^\eta\rceil$ is chosen to be arbitrary small, albeit fixed and positive degree of $1/\eps$.
		
	%\noindent Indeed, if this not the case, then Theorem \ref{Theorem:EasyBound} yields $G_d(\eps)=\Theta(1)$.

	Let $E_1$ denote the subset of all the crowded hyperedges $\{p_1,\ldots,p_{d+1}\}\in E$, and let $E_2$ denote the subset of all the split hyperedges $\tau=\{p_1,\ldots,p_{d+1}\}\in E$ whose ambient families 
	$$
	\Sigma_\tau=\left\{\Delta\left(p_1\right),\ldots,\Delta\left(p_{d+1}\right)\right\}\subseteq \Sigma_d(P,r)
	$$ 
	\noindent are either crossed or loose.
By Corollary \ref{Corol:BoundBadSubsets}, we have that that 

\begin{equation}\label{Eq:BadEdges}
|E_1\cup E_2|=O\left(\frac{n^{d+1}}{r^{\frac{1}{d^4+d}}}\right).
\end{equation}
	
	\noindent However, since $r$ is now a constant, the right-hand side of (\ref{Eq:BadEdges}) can be much larger than the overall cardinality of $E$.
	We thus distinguish between two possible scenarios.
	
	\medskip
	\noindent{\bf Case 1.} If $|E_1\cup E_2|\leq |E|/2$ then there remain at least $|E|/2$ edges $\tau=\{p_1,\ldots,p_{d+1}\}\in E\setminus (E_1\cup E_2)$ that are split in $\Pi_d(P,r)$, and whose respective ambient families $\Sigma_\tau=\left\{\Delta\left(p_1\right),\ldots,\Delta\left(p_{d+1}\right)\right\}$ are tight.
	The pigeonhole principle, there is a subset $E'\subseteq E\setminus (E_1\cup E_2)$ of cardinality 
	$$
	|E'|\geq \frac{|E|}{2{r\choose d+1}}\geq \frac{\eps}{{r\choose d+1}}{n\choose d+1}.
	$$ 
	\noindent and a tight $(d+1)$-family $\K$ of $d+1$ simplices within $\Sigma_d(P,r)$, so that all the hyperedges $\tau\in E\setminus (E_1\cup E_2)$ share $\Sigma_\tau=\K$ as their ambient family.
	By Theorem \ref{Theorem:StronglySeparated}, {\it any} point $x\in \Delta(\K)$ pierces all the edges in $E'$.
	
	\medskip
	\noindent{\bf Case 2.} If $|E_1\cup E_2|\geq |E|/2$ 
	we use the fact that, according to Theorems \ref{Theorem:Crossed} and \ref{Theorem:Loose}, the family $\Sigma_d(P,r)$ encompasses only 
	$$
	O\left(r^d+r^{d+1-1/d}+r^{d+1-\frac{1}{d^4+d}}\right)=O\left(r^{d+1-\frac{1}{d^4+d}}\right)
	$$ 
	\noindent distinct ambient families $\Sigma_\tau\subseteq \Sigma_K(P,r)=\{\Delta_1,\ldots,\Delta_r\}$ of the simplices $\tau\in E_1\cup E_2$. Indeed, each of these families $\Sigma_\tau$ has cardinality smaller than $d+1$ (if $\tau\in E_1$), or else is either crossed or loose (if $\tau\in E_2$).
Hence, another application of the pigeonhole principle yields a subset $E''\subseteq E$, of 
	$$
	\Omega\left(\frac{|E|}{r^{d+1-\frac{1}{d^4+d}}}\right)=	\Omega\left(\frac{\eps {n\choose d+1}}{r^{d+1-\frac{1}{d^4+d}}}\right),
	$$
	\noindent hyperedges $A\in E_1\cup E_2$ that share the same ambient a family $\Sigma_\tau=\K=\{\Delta_{i_1},\ldots,\Delta_{i_k}\}\subseteq \Sigma_d(P,r)$, where $k\leq d+1$ and $1\leq i_1<\ldots<i_k\leq r$.

	Let $P'':=\bigcup_{j=1}^{k}P_{i_j}$, then the induced hypergraph $\left(P'',E''\right)$ contains $\Theta\left(n/r\right)$ vertices
	and
	$$
	\left|E''\right|=\Omega\left(\frac{\eps {n\choose d+1}}{r^{d+1-\frac{1}{d^4+d}}}\right)=\Omega\left(\frac{\eps r^{d+1}|P''|^{d+1}}{r^{d+1-\frac{1}{d^4+d}}}\right)=\Omega\left(\eps r^{\frac{1}{d^4+d}} |P''|^{d+1}\right)
	$$
	\noindent hyperedges. Hence, there must be a point that pierces at least 
	$$
	\Omega\left(\left|P''\right|^{d+1}G_d\left(c\eps\cdot r^{\frac{1}{d^4+d}}\right)\right)=\Omega\left(n^{d+1}\cdot \frac{G_d\left(c\eps\cdot r^{\frac{1}{d^4+d}}\right)}{r^{d+1}}\right)
	$$
	\noindent simplices in $E''\subseteq E$, where $c>0$ is a constant that depends only on the dimension $d$. Since $(P,E)$ is an arbitrary $(d+1)$-uniform geometric hypergraph in $\reals^d$ with density $|E|/{|P|\choose d+1}\geq \eps$, this yields the recurrence 
	\begin{equation}\label{Eq:RecurrenceG}
		G_d(\eps)\geq \min\left\{\frac{\eps}{{r\choose d+1}},\frac{G_d\left(c \eps\cdot r^{\frac{1}{d^4+d}}\right)}{r^{d+1}}\right\}.
	\end{equation}
	
	By substituting $T_d(\eps):=1/G_d(\eps)$, we can rewrite (\ref{Eq:RecurrenceG}) as 
	
	\begin{equation}\label{Eq:RecurrenceT}
		T_d(\eps)\leq \frac{{r\choose d+1}}{\eps}+r^{d+1}\cdot T_d\left(c \eps\cdot r^{\frac{1}{d^4+d}}\right).
	\end{equation}

	The recurrence bottoms out when $\eps$ bypasses a certain constant threshold $0<\eps_0<1$, so that the slightly weaker Theorem \ref{Theorem:EasyBound} yields the bounds $G_d(\eps_0)=\Omega(1)$ and $T_d(\eps)=O(1)$, for all $\eps>\eps_0$. By fixing a sufficiently small constants $\eps_0=\eps_0(\delta,c,d)$ and $\eta=\eta(\delta,c,d)$ in $r=\lceil 1/\eps^\eta\rceil$, and following the standard methodology that is described, e.g., in \cite[Section 7.3.2]{SA}, the 
	recurrence solves to 
	
	\begin{equation}\label{Eq:RecurrenceT}
		T_d(\eps)=O\left(1/\eps^{(d^4+d)(d+1)+\delta}\right)
	\end{equation}

	\noindent so that	$G_d(\eps)=\Omega\left(\eps^{(d^4+d)(d+1)+\delta}\right)$; see, e.g., \cite{MatWag04,Rubin2D,Rubin} for solutions of similar recurrences. $\Box$

%\section{Tight and loose families of convex sets -- proof of Theorem \ref{Theorem:StronglySeparated}}
 
\section{Proof of Theorem \ref{Theorem:NewTuran}} \label{Sec:Polynomial}

%As was mentioned in the Introduction, the proof of Theorem \ref{Theorem:NewTuran} largely resembles that of a similar Tur\'an-type result by Fox, Pach, and Suk \cite{Regularity}.
\subsection{A reduction to bi-partite graphs}

Notice that any $k$-uniform $k$-partite hypergraph $(V_1,\ldots,V_k,E)$, with $k\geq 3$, can be represented as a bipartite graph $(V_1,V'_2,E)$, where $V'_2=V_2\times\ldots\times V_k$, so that a pair $(v_1,(v_2,\ldots,v_k))\in V_1\times V'_2$ belongs to $E$ if and only if $(v_1,\ldots,v_k)$ belongs to $E$.
%Similar to the Tur\'an-type result of Fox, Pach, and Suk \cite{Regularity}, Theorem \ref{Theorem:NewTuran} will be derived by $d-1$ iterations of its asymmetric analogue for semi-algebraic {\it bi-partite graphs}.

\medskip
\noindent{\bf Definition.} We say that a bi-partite graph $(V_1,V_2,E)$ has semi-algebraic complexity $(D,s)$ in $\left(\reals^{d_1},\reals^{d_2}\right)$ if it meets the following criteria:
\begin{enumerate}
	\item $V_1\subset \reals^{d_1}$, $V_2\subset \reals^{d_2}$, and
	\item There is a subset $Y\subseteq \reals^{d_1+d_2}$ that admits a semi-algebraic description $(f_1,\ldots,f_s,\Phi)$ of complexity $(D,s)$ (with polynomials $f_1,\ldots,f_s\in \reals[x_1,\ldots,x_{d_1},y_1\ldots,y_{d_2}]$ of degree $\deg(f_i)\leq D$ each), so that $E=(V_1\times V_2)\cap Y$.
\end{enumerate}

In the sequel we establish the following auxiliary bi-partite result which is parallel to Lemma 2.3 of Fox, Pach, and Suk \cite{Regularity}.

\begin{lemma} \label{Lemma:BipartiteRegularity}
For any choice of integers $d_1>0$, $d_2>0$, $D>0$ and $s\geq 0$, there exists integers $D'=D'(d_1,d_2,D,s)>0$ and $s'=s'(d_1,d_2,D,s)$ with the following property.

Let $(V_1,V_2,E)$ by a bi-partite graph of semi-algebraic description complexity at most $(D,s)$ in $(\reals^{d_1},\reals^{d_2})$ so that $|E|\geq \varepsilon\cdot |V_1|\cdot |V_2|$. Then there exist subsets $W_1\subseteq V_1,W_2\subseteq V_2$, along with a set $Y\subseteq \reals^{d_2}$ of semi-algebraic description complexity at most $\left(D',s'\right)$, that meet the following criteria:
\begin{enumerate}
	\item  $W_1\times W_2\subseteq E$, 
	\item $|W_1|\cdot |W_2|=\Omega\left(\varepsilon^{d_1+1}\cdot |V_1|\cdot |V_2|\right)$, and
	\item $W_2=V_2\cap Y$.
\end{enumerate}

\noindent Furthermore, we have that

\begin{equation*}
|W_1|=\begin{cases}
\Omega\left(\varepsilon^{d_1}|V_1|\right) &\text{$d_1\geq 2$}\\
\Omega\left(\frac{\varepsilon |V_1|}{\log^2(1/\varepsilon)}\right) &\text{$d_1=1$}
\end{cases}
\end{equation*}

\noindent and  $|W_2|=\Omega\left(\varepsilon |V_2|\right)$.

The constants implicit in our $O(\cdot)$-notation, may depend on $d_1,d_2$, $D$ and $s$.
\end{lemma}

\noindent{\bf Proof of Theorem \ref{Theorem:NewTuran} via Lemma \ref{Lemma:BipartiteRegularity}.} The theorem is established through induction in $k\geq 2$, where the ground case $k=2$ is provided by Lemma \ref{Lemma:BipartiteRegularity} (with $d_1=d_2=d$). Let us assume, then, that $k>2$, and that Theorem \ref{Theorem:NewTuran} holds for all the $(k-1)$-uniform instances.
As was previously described, any $k$-uniform hypergraph $G=\left(V_1,\ldots,V_k,E\right)$ can be ``converted" to the bi-partite graph $G'=(V_1,V'_2=V_2\times\ldots\times V_k,E')$  in $\left(\reals^{d},\reals^{d\times (k-1)}\right)$, with the property that $(v_1,\{v_2,\ldots,v_k\})\in E'$ holds if and only if $(v_1,v_2,\ldots,v_k)\in E$. Notice that $G'$ admits the essentially semi-algebraic description whose complexity is at most $(D,s)$. 

Applying Lemma \ref{Lemma:BipartiteRegularity} to $G'$ yields subsets $W_1\subseteq V_1$ and $W'_2\subseteq V'_2$ with the property $W_1\times W'_2\subseteq E'$, and a semi-algebraic set $Y\subseteq \Gamma_{d(k-1),D',s'}$ within $\reals^{d\times (k-1)}$, with the property that $W'_2=V'_2\cap Y$, such that the following inequalities hold 

\begin{equation}\label{Eq:Bipartite}
|W_1|\cdot |W'_2|=\Omega\left(\varepsilon^{d_1+1}|V_1|\cdot|V'_2|\right),
\end{equation}

\begin{equation*}
|W_1|=\begin{cases}
\Omega\left(\varepsilon^{d_1}|V_1|\right) &\text{$d_1\geq 2$}\\
\Omega\left(\frac{\varepsilon|V_1|}{\log^2(1/\varepsilon)}\right) &\text{$d_1=1$}
\end{cases}
\end{equation*}

\noindent and  $|W'_2|=\Omega\left(\varepsilon |V'_2|\right)$.
%Here $\tilde{D}=\tilde{D}(d,d(k-1),D,s)$ and $\tilde{s}=\tilde{s}(d,d(k-1),D,s)$ are the constants in Lemma \ref{Lemma:BipartiteRegularity}.
Denote
$$
\gamma:=\frac{|W'_2|}{|V'_2|}=\frac{|W'|}{|V_2|\cdot\ldots\cdot |V_k|}=\Omega(\varepsilon).
$$
Notice that $\tilde{G}=(V_2,\ldots,V_k,W'_2)$ is a $\gamma$-dense, $(k-1)$-uniform hypergraph in $\reals^{d}$ that admits a semi-algebraic decription whose complexity is bounded by $\left(D',s'\right)$. Invoking the inductive assumption for $\tilde{G}$ yields subsets $W_2\subseteq V_2,\ldots,W_k\subseteq V_k$ with the property that $W_2\times\ldots\times W_k\subseteq W'_2$ and 
\begin{equation}\label{Eq:Induction}
|W_2|\cdot\ldots\cdot|W_k|=\Omega\left(\gamma^{d(k-2)+1}|V_2|\cdot\ldots\cdot |V_k|\right),
\end{equation}

\noindent and so that

\begin{equation*}
|W_i|=\begin{cases}
\Omega\left(\gamma^{d_1}|V_i|\right) &\text{$d_1\geq 2$}\\
\Omega\left(\frac{\gamma|V_i|}{\log^2(1/\gamma)}\right) &\text{$d_1=1$}
\end{cases}
\end{equation*}
\noindent holds for all $2\leq i\leq k-1$, and $|W_k|=\Omega\left(\gamma|V_k|\right)$. Notice that $W_1,\ldots,W_k$ meet the

To complete the proof of Theorem \ref{Theorem:NewTuran}, notice that

$$
 W_1\times W_2\times\ldots\times W_k\subseteq E
$$

\noindent (as $W_1\times W'_2\subseteq E$ and $W_2\times\ldots\times W_k\subseteq W'_2$). Futhermore, the combination of (\ref{Eq:Bipartite}) and (\ref{Eq:Induction}), together with the fact that $\gamma=\Omega\left(\varepsilon\right)$, implies that

$$
|W_1|\cdot\ldots\cdot |W_k|=\Omega\left(\frac{\varepsilon^{d+1}}{\gamma}|V_1|\cdot \gamma^{d(k-2)+1}|V_2|\cdot\ldots\cdot |V_k|\right)=\Omega\left(\varepsilon^{d(k-1)+1}|V_1|\cdot\ldots\cdot |V_k|\right).
$$

\noindent $\Box$

\subsection{Proof of Lemma \ref{Lemma:BipartiteRegularity}}

To translate the problem to the framework of {geometric range spaces} \cite{HW87},
note that $(V_1,V_2,E)$ determines the hypergraph $(V_2,\E)$ where $\E$ is comprised of the neighborhood sets $\N(v_1)=\{v_2\mid (v_1,v_2)\in E\}$ of the vertices $v_1\in V_1$. (It is possible for a pair of distinct vertices $v_1,v'_1\in V_1$ to yield the same hyperedge $\N(v_1)=\N(v'_1)$ in $\E$.)

Let $(\Phi,f_1,\ldots,f_s)$ be the semi-algebraic description of $(V_1,V_2,E)$, whose complexity is bounded by $(D,s)$. Consider the family $\R=\{R_x\mid x\in \reals^{d_1}\}$ whose elements are the semi-algebraic sets of the form $R_x:=\{y\in \reals^{d_2}\mid \Phi(f_1(x,y),\ldots,f_s(x,y))\}$, which clearly belong to $\Gamma_{d_2,D,s}$. 
Let
 $[\R]_{W}$ denote the restriction $\{R\cap W\mid R\in \R\}$ of $\R$ to a subset $W\subset \reals^{d_2}$.
Then the pair $(\reals^{d_2},\R)$ constitutes a {\it range space} which includes the hypergraph $(V_2,\E)$ (as we have that $\E\subseteq [\R]_{V_2}$).

 More generally, for any $m\in {\mathbb N}$ we use
$\varphi_2(m)$ to denote the maximum cardinality of $[\R]_{W}$ that is taken over all the $m$-size subsets $W\subset \reals^{d_2}$. The resulting function $\varphi_2:{\mathbb N}\rightarrow {\mathbb N}$ is called the {\it shatter function} of $(\reals^{d_2},\R)$.
 
 \begin{lemma}\label{Lemma:Shatter}
  	For any integer $m\geq 1$ we have that $\varphi_2(m)=O\left(m^{d_1}\right)$, where the constants of proportionality depend on $d_1,d_2,D$ and $s$.
 \end{lemma}
 
 \begin{proof}
  Fix an $m$-size subset $W\subseteq \reals^{d_2}$. To bound the cardinality of $[\R]_W$, we consider the family $\F'_W=\{R'_w\mid w\in W\}$, 
  which is comprised of all the {\it dual} semi-algebraic ranges of the form
  $R'_w:=\{x\in \reals^{d_1}\mid  \Phi(f_1(x,w),\ldots,f_s(x,w))\}$, with $w\in W$. The {\it arrangement} of $\F'_W$ \cite{BasuBook,SA} is a decomposition of $\reals^{d_1}$ into maximal connected regions with the property that all the points that lie in the same such region, belong to the same ranges of $\F'_W$. According to a variant of Milnor-Thom Theorem \cite{Warren}, this arrangement $\A\left(\F'_W\right)$
  is comprised of $O\left(m^{d_1}\right)$ faces, with implicit constants of proportionality that depend on $d_1,d_2,D$ and $s$. (To see this, it suffices to observe that every face of $\A(\F'_W)$ is a union of one or more faces in the finer arrangement of the surface family $\{\xi_{w,i}\mid w\in W,1\leq i\leq s\}$, where $\xi_{w,i}:=\{x\in \reals^{d_1}\mid f_i(x,w)=0\}$.) 
 The claim now follows since any two points $x_1,x_2\in \reals^{d_1}$ that lie in the same face of $\A\left(\F'_W\right)$, determine identical subsets $R_{x_1}\cap W=R_{x_2}\cap W$ within $[\R]_W$.
 \end{proof}

%Let $0<\varepsilon'\leq 1$. We say that an edge $f\in E$ is {\it $\varepsilon'$-heavy} if we have that $|f|\geq \varepsilon' |V_2|$. We say that a vertex $v_1\in V_1$ is {\it $\varepsilon'$-heavy} if its respective edge $f=\Gamma(v_1)$ is $\varepsilon'$-heavy.

 For any $1\leq i\leq \log (10/\varepsilon)$, let $V_{1,i}$ denote the subset of all such vertices $v_1\in V_1$ that satisfy\footnote{For the sake of brevity, we use $\log x$ to denote the quantity $\lceil\log_2 x\rceil$.}
 $$
\frac{2^{i-1}\varepsilon|V_2|}{10} \leq |\Gamma(v_1)|<\frac{2^i \varepsilon |V_2|}{10},
 $$
 
 \noindent and let $\E_{i}$ denote the subset $\{\N(v_1)\mid v_1\in V_{1,i}\}\subseteq \E$.

\begin{lemma}\label{Lemma:DenseVertices}
There is $1\leq i\leq \log (10/\varepsilon)$ so that $|V_{1,i}|=\Omega\left(\frac{|V_1|}{2^ii^2}\right)$.
\end{lemma}
\begin{proof}
 Since the vertices of $V_1\setminus \left(\bigcup_{i=1}^{\log (10/\varepsilon)}V_{1,i}\right)$ account for at most $\varepsilon|V_1||V_2|/10$ edges of $E$, the vertices of $\bigcup_{i=1}^{\log (10/\varepsilon)}V_{1,i}$ must be adjacent to at least $|E|/2\geq \frac{\varepsilon}{2}|V_1|\cdot |V_2|$ edges $(v_1,v_2)\in E$. 
However, if the contrapositive statement holds, then each subset $V_{1,i}$ accounts for at most
 $$
  \frac{|V_{1}|}{2^{i}i^2}\cdot \frac{2^i\varepsilon|V_2|}{10}\leq \varepsilon\cdot \frac{|V_1|\cdot |V_2|}{10i^2}
 $$
such edges, for a total of at most $\displaystyle \sum_{i=1}^{\log (10/\varepsilon)}\frac{\varepsilon |V_1|\cdot |V_2|}{10i^2}\leq \varepsilon |V_1|\cdot |V_2|<|E|/2$ edges.
\end{proof}

Fix $1\leq i\leq \log (10/\varepsilon)$ that meets the criteria of Lemma \ref{Lemma:DenseVertices}. Let $\E_i=\{\N(v_1)\mid v_{1}\in V_{1,i}\}$. Denote $\varepsilon_i=\frac{2^{i-1}\varepsilon}{10}$, so that every hyperedge $\sigma\in \E_i$ satisfies $|\sigma|\geq \varepsilon_i |V_2|$. 

\begin{lemma}\label{Lemma:Mnet}
With the previous choice of $1\leq i\leq \log(10/\varepsilon)$, there exists a collection $\W_i$ of $t=O\left(1/\varepsilon_i^{d_1}\right)$ subsets $\omega\subseteq V_2$ with the following properties:

\begin{enumerate}
\item We have that $|\omega|=\Omega(\varepsilon_i |V_2|)$ for all $\omega\in \W_i$.
\item every edge $\sigma\in \E_{i}$ contains at least one of these subsets $\omega\in \W_i$.
\item  every $\omega\in \W_i$ can be assigned a set $Y_\omega\subseteq \reals^{d_2}$ with the property that $\omega=V_2\cap Y_\omega$, and whose semi-algebraic description complexity is bounded in $d_1,d_2,D$ and $s$.
\end{enumerate}

\end{lemma}
 
 \begin{proof}
  In the notation of Mustafa and Ray \cite{MustafaRay}, any family $\W_i$ of subsets $\omega\subseteq V_2$ that meet the first two criteria, is called an {\it $\varepsilon_i$-Mnet} for the hypergraph $(V_2,\E_{i})$.
Though the Mnet property is not enough to assign a semi-algebraic set $Y_\omega$ to each $\omega\in \W_i$, our construction is overly inspired by the argument of Dutta {\it et al.} \cite{Dutta}, which yields an {$\varepsilon_i$-Mnet} $\W_i$ of cardinality $t=O\left(1/\varepsilon_i^{d_1}\right)$ via the multilevel polynomial partition of Theorem \ref{Theorem:MultiLevel}.

To this end, let us fix the constants $a=a(d_2),b=b(d_2)$ and $c=c(d_2,D,s)$ in accordance with Theorem \ref{Theorem:MultiLevel}.
If $\varepsilon_i\leq 100(100d_2c)^{d_2a}/|V_2|$, then $\W_i=\{\{v\}\mid v\in V_2\}$ is the desired set of cardinality $O(1/\varepsilon_i)$, and each singleton element $\{v\}\in \W_i$ admits a semi-algebraic description whose complexity is bounded by $(1,d_2)$. Otherwise, the family $\W$ is defined in two basic steps.

We first construct a maximal {\it $(\varepsilon_i/10)$-packing $\F_i\subseteq \E_{i}$} for the hypergraph $(V_2,\E_{i})$, that is, a subset $\F_i\subseteq \E_i$ with the following property: for any $\sigma\in \E_{i}$, there is $\rho\in \F_i$ so that $|\rho\triangle \sigma|\leq \varepsilon_i|V_2|/10$. Combining Haussler's Theorem with the upper bound of Lemma \ref{Lemma:Shatter} on the shatter function $\varphi_2$ of the ambient range space $(\reals^{d_2},\R)$, yields that $|\F_i|=O\left(1/\varepsilon_i^{d_1}\right)$.  Note that every edge $\rho\in \F_i$ is ``cut out", as a subset of $V_2$, by a semi-algebraic range $R_\rho\in \R$.
 
 In the second step, we invoke the multilevel polynomial partition of Theorem \ref{Theorem:MultiLevel} for each edge $\rho\in \F_i$. To this end, we fix a suitably large, albeit constant $r=(100d_2c)^{d_2}$, which guarantees that $r^{a}\leq \varepsilon_i|V_2|/100$. Consider the resulting ``secondary" partition 
 $$
 \rho=\omega^*\uplus \biguplus_{j=1}^{d_2} \biguplus_{l=1}^{t_j} \omega_{jl},
 $$ 
 
\noindent where $|\omega^*|\leq r^a$, and $|\omega_{jl}|\leq |\rho|/r_j$ for all $1\leq l\leq t_j$. Recall that we have that $r\leq r_j\leq r^a$ for all $1\leq j\leq d$. Furthermore, each set $\omega_{jl}$ is ``assigned" a semi-algebraic set $X_{\omega_{jl}}\subseteq \reals^{d_2}$ whose semi-algebraic description complexity is bounded by $\left(br^b,br^b\right)$, with the property that $\omega_{jl}=\rho\cap X_{\omega_{jl}}$, and so that any semi-algebraic set $X\in \Gamma_{d_2,D,s}$ crosses at most $cr_j^{1-1/d_2}$ of the sets $X_{\omega_{jl}}$, for $1\leq l\leq t_j$.

Let $\W_{\rho}$ denote the sub-family of all such subsets $\omega_{jl}\subset \rho$ that satisfy
 $$
 |\omega_{jl}|\geq \frac{\varepsilon_i |V_2|}{100bd_2r^{b}}=\Omega(\varepsilon_i |V_2|).
 $$
 
\noindent Since the overall number of sets $\omega_{jl}$ in the partition of $\rho$ does not exceed $d_2\cdot br^{b}$, we have that 
\begin{equation}\label{Eq:ManyPoints}
	\sum_{\omega_{jl}\in \W_\rho} |\omega_{jl}|\geq |\rho|-\varepsilon_i|V_2|/100-r^a\geq |\rho|-\varepsilon_i|V_2|/50.
\end{equation}

\medskip
To obtain the family $\W_i$, we repeat the previous construction for each $\rho\in \F_i$, and denote
$$
\W_i:=\bigcup_{\rho\in \F_i}\W_{\rho}.
$$
 
\noindent Notice that, for any $\rho\in \W_i$, every element $\omega\in \W_\rho$ is contained in the set $Y_\omega:=R_\rho\cap X_{\omega}$ which satisfies $\omega=V_2\cap Y_\omega$, and whose semi-algebraic description complexity is bounded by $(\max\{br^b,D\},br^b+s)$.

To show that any edge in $\E_{i}$ contains one or more subsets $\omega\in \W_i$, let us fix such an edge $\sigma\in \E_{i}$. Since $\F_i$ is a maximal packing in $(V,\E_{i})$, there must exist $\rho\in \F_i$ so that $|\rho\cap \sigma|\geq |\sigma|-|\rho\triangle \sigma|\geq |\sigma|-\varepsilon |V_2|/10$. Together with (\ref{Eq:ManyPoints}), this implies that at least $|\sigma|-\varepsilon_i|V_2|/5\geq 4\varepsilon_i|V_2|/5$ points in $\sigma$ fall in  $\bigcup\W_\rho$. 
Recall that $\sigma$ is determined by a range $R_{\sigma}$ in $\R\subseteq \Gamma_{d_2,D,s}$, with the property that $\sigma=R\cap V_2$.
Since $R$ crosses at most $cr_j^{1-1/d_2}$ of the semi-algebraic sets $X_{\omega}$ that arise in the multi-level polynomial decomposition of $\rho$, it follows that $\sigma$ crosses at most $cr_j^{1-1/d_2}$ of the respective elements $\omega\in \W_\rho$, whose overall cardinality is at most 
$$
\sum_{j=1}^{d_2} cr_j^{1-1/d_2}\cdot \frac{|\rho|}{r_j}\leq cd_2|\rho|/r^{1/d_2}\leq \varepsilon_i|V_2|/10.
$$

\noindent Hence, there remain at least $\varepsilon_i|V_2|/2$ points in $\sigma$ that fall into such subsets $\omega\in \W_\rho$ that are contained in $\sigma$. In particular, $\sigma$ must contain at least {\it one} element of $\W_\rho\subseteq \W_i$.
\end{proof}

\noindent {\bf Proof of Lemma \ref{Lemma:BipartiteRegularity} -- wrap up.} Let $\W_i$ be a family that meets the criteria of Lemma \ref{Lemma:Mnet}. By the pigeonhole principle, there must exist a subset $W_1\subset V_{1,i}$ of cardinality $\Omega\left(\varepsilon_i^{d_1}|V_{1,i}|\right)$ and an element $\omega\in \W_i$ so that $\omega\subseteq \N(v_1)$ holds for all $v_1\in W_1$. Thus, choosing $W_2=\omega$ guarantees that $W_1\times W_2\subseteq E$. 
Furthermore, notice that we have that
$$
|W_1|\cdot |W_2|=\Omega\left(\varepsilon_i^{d_1}|V_{1,i}|\cdot \varepsilon_i |V_2|\right)=\Omega\left(\frac{(2^i\varepsilon)^{d_1+1}|V_1|}{2^ii^2}\right)=\Omega\left(\varepsilon^{d_1+1}|V_1|\cdot |V_2|\right).
$$

To complete the proof of Lemma \ref{Lemma:BipartiteRegularity}, note that we have that $W_2=V_2\cap Y$, where $Y=Y_{\omega}\subseteq \reals^{d_2}$ is a subset whose semi-algebraic description complexity is bounded in the terms of $d_1,d_2,D$ and $s$, and that $|W_2|=\Theta\left(\varepsilon_i|V_2|\right)=\Omega\left(\varepsilon |V_2|\right)$, and we have that

$$
 |W_1|\geq \varepsilon_i^{d_1}|V_{1,i}|=\Omega\left(2^{id_1}\varepsilon^{d_1}\cdot \frac{|V_1|}{2^ii^2}\right),
$$

\noindent where the right hand side is $\Omega\left(\varepsilon^{d_1}|V_1|\right)$ for $d_1>0$ and $\Omega\left(\varepsilon|V_1|/\log^2 (1/\varepsilon)\right)$ for $d_1=1$. $\Box$

\section{Concluding remarks}\label{Section:Concluding}

Let us conclude the paper with several corollary observations and open problems.

\subsection{A strengthening of Pach's theorem}
%Pach's Theorem yields, for any $d+1$ pairwise sets $P_1,\ldots,P_{d+1}$, each comprised of $n$ points in general position in $\reals^d$, a tight $(d+1)$-tuple $\{Q_i\mid 1\leq i\leq d\}$ with the property that $Q_i\subseteq P_i$, and $|Q_i|\geq c'_d|P_i|$. 
In view of the characterization in Theorem \ref{Theorem:StronglySeparated}, one can recast Pach's Theorem \ref{Theorem:Pach} in the following form: {any non-empty $d+1$ $n$-point sets $P_1,\ldots,P_{d+1}$ in general position in $\reals^d$ contain subsets $Q_i\subseteq P_i$ of cardinality $|Q_i|\geq c'_d n$, whose respective convex hulls comprise a tight family $\K=\{\conv(Q_1),\ldots,\conv(Q_{d+1})\}$}.

\medskip
An immediate by-product of our machinery is the following strengthening of Pach's Theorem.

%\begin{theorem}\label{Theorem:MultiPach}
	% Let $P_1,\ldots,P_{d+1}$ be $d+1$ pairwise disjoint point sets in general position in $\reals^d$, $r$ be an integer that satisfies $r\leq \min\{|P_i|\mid 1\leq i\leq d+1\}$. 
	%Let $\Pi_d(P_1,r),\ldots,\Pi_d(P_{d+1},r)$ denote their respective simplicial $r$-partitions via Theorem \ref{Theorem:Simplicial}. Then all but $O\left(r^{d+1-\frac{1}{d^4+d}}\right)$ the $(d+1)$-tuples $\left\{P_{1,j_1},\ldots,P_{d+1,j_{d+1}}\right\}$, with $P_{i,j_i}\in \Pi_d(P_i,r)$, are tight.
%\end{theorem}

%Hence, the combination of Theorems \ref{Theorem:PartiteCrossed} and \ref{Theorem:PartiteLoose} yields the following statement.

 \begin{theorem} \label{Theorem:MultiPach} 
Let $P_1,\ldots,P_{d+1}$ be $d+1$ finite sets, each in general position in $\reals^d$, and $r$ a positive integer that satisfies $r\leq \min\{|P_i|\mid 1\leq i\leq d+1\}$. Then there exist partitions
$P_i=\biguplus_{j=1}^{r_i}P_{i,j}$, for $1\leq i\leq d+1$, each into $r_i\leq r$ parts of cardinality $\lceil |P_i|/r\rceil\leq |P_{i,j}|\leq 2\lceil |P_i|/r\rceil$, with the following property:
All but $O\left(r^{d+1-\frac{1}{d^4+d}}\right)$ of the $(d+1)$-tuples $\left(P_{1,j_1},\ldots,P_{d+1,j_{d+1}}\right)$ are tight.\footnote{We say that a family $\K=\{X_1,\ldots,X_{d+1}\}$ of $d+1$ finite point sets in $\reals^d$ is {\it tight} if $\{\conv(X_1),\ldots,\conv(X_{d+1})\}$ is. }

%Moreover, such partitions can be obtained by applying Theorem \ref{Theorem:Simplicial} to each set $P_i$.
\end{theorem}

 %As a result, Theorem \ref{Theorem:MultiPach} yields Theorem \ref{Theorem:Pach}, albeit with a poor constant.

\smallskip
Theorem \ref{Theorem:MultiPach} is an immediate corollary of the following $(d+1)$-partite formulations of Theorems \ref{Theorem:Crossed} and \ref{Theorem:Loose} which involve simplicial $r$-partitions  
$$
\Pi_d(P_1,r)=\{\left(P_{1,j},\Delta_{1,j}\right) \}_{j=1}^{r_1},\ldots,\Pi_d(P_{d+1},r)=\{\left(P_{d+1,j},\Delta_{d+1,j}\right)\}_{j=1}^{r_{d+1}}.
$$ 

\noindent of some $d+1$ sets $P_1,\ldots,P_{d+1}$, whose respective cardinalities satisfy $|P_i|\geq r$.

%, along with respective families $\Sigma_d(P_1,r)=\{\Delta_{1,j}\}_{j=1}^{r_1},\ldots,\Sigma_d(P_{d+1},r)=\{\Delta_{d+1,j}\}_{j=1}^{r_{d+1}}$ of enclosing simplices.\footnote{Though the sets $P_i$ need not necessarily be pairwise disjoint, it can be safely assumed that the vertex set $\bigcup_{i=1}^{d+1}V\left(\Sigma_d(P_i,r)\right)$ is comprised of $(d+1)\sum_{i=1}^{d+1}|V(\Sigma_d(P_i))|$ points in general position, as long as each set $P_i$ is in general position.} %Notice that, in contrast to Pach's theorem, it is nowhere assumed that $|P_1|=\ldots=|P_{d+1}|$.

\begin{theorem} \label{Theorem:PartiteCrossed} 
Let $P_1,\ldots,P_{d+1}$ be $d+1$ finite sets, each in general position in $\reals^d$, and $0<r\leq \min\{|P_i|\mid 1\leq i\leq d+1\}$ an integer. Then the $r$-partitions $\Pi_d(P_1,r),\ldots,\Pi_d(P_{d+1},r)$ 
determine only $O\left(r^{d+1-1/d}\right)$ crossed $(d+1)$-size families $\left(\Delta_{1,j_1},\ldots,\Delta_{d+1,j_{d+1}}\right)\in \Sigma_d(P_1,r)\times\ldots\times \Sigma_d(P_{d+1},r)$. 

%Thus, all but $O\left(r^{d+1-1/d}\right)$ of the $(d+1)$-tuples $(P_{1,j_1},\ldots,P_{d+1,j_{d+1}})$ are homogeneous with respect to the orientation function $\chi$, in the sense that all the $(d+1)$-point sequences $(p_1,\ldots,p_{d+1})\in P_{1,j_1}\times\ldots\times P_{d+1,j_{d+1}}$ attain the same orientation $\chi(p_1,\ldots,p_{d+1})\in \{+,-\}$.
\end{theorem}

To ``convert" the proof of Theorem \ref{Theorem:Crossed} to the colorful setting of Theorem \ref{Theorem:PartiteCrossed}, note that any crossed family $\K=\{\Delta_{1,j_1},\ldots,\Delta_{d+1,j_{d+1}}\}$ can still be charged to a ``signature subset" $\I_K\subsetneq \K$ so that $V(\I_K)$ contains a subset $V_H$ of $d$ vertices which support a common hyperplane transversal $H$ to $\K$. The fact that $H$ crosses at most $O\left(r^{1-1/d}\right)$ simplices in each family $\Sigma_d(P_i,r)$, implies that any such signature $\I_K$, of at most $d$ simplices, is shared by at most $O\left(r^{1-1/d}\right)$ crossed families $\K$.

 The proof of Theorem \ref{Theorem:Loose} is modified even more easily, by directly considering the {\it already $(d+1)$-partite} semi-algebraic hypergraph $\left(\Sigma_1,\ldots,\Sigma_{d+1},E^*(\Sigma_1,\ldots,\Sigma_{d+1})\right)$, with $\Sigma_i=\Sigma_d(P_i,r)$ for all $1\leq i\leq d+1$, thus yielding the following property.

\begin{theorem}\label{Theorem:PartiteLoose} 
Let $P_1,\ldots,P_{d+1}$ be $d+1$ finite sets, each in general position in $\reals^d$, and $r$ is a sufficiently large positive integer which satisfies $r\leq \min\{|P_i|\mid 1\leq i\leq d+1\}$. Then the $r$-partitions $\Pi_d(P_1,r),\ldots,\Pi_d(P_{d+1},r)$ 
determine only $O\left(r^{d+1-\frac{1}{d^4+d}}\right)$ loose $(d+1)$-size families $\left(\Delta_{1,j_1},\ldots,\Delta_{d+1,j_{d+1}}\right)\in \Sigma_d(P_1,r)\times\ldots\times \Sigma_d(P_{d+1},r)$.

% and, thereby,  only $O\left(r^{d+1-\frac{1}{d^4+d}}\right)$ loose $(d+1)$-size families $\left\{P_{1,j_1},\ldots,P_{d+1,j_{d+1}}\right\}$. 

%Then at least one of the following conditions is satisfied.
%\begin{enumerate}
	%\item There exist a separated $(d+1)$-family $\P=\left(P_1,P_2,\ldots,P_{d+1}\right)$ of pairwise disjoint sub-sets $P_i\subseteq P$, each of cardinality $n/r\leq |P_i|\leq 2n/r$ and so that $|E\cap \left(P_1\times P_2\times\ldots\times P_{d+1}\right)|=\Omega\left(tn^{d+1}/r^{d+1}\right)$.
	%\item There exists a point $x\in \reals^d$ that pierces at least $\Omega\left(F_d(t\cdot r^{1/d})(n/r)^{d+1}\right)$ simplices in $E$. 
 %\end{enumerate}
\end{theorem}

 In particular, applying Theorem \ref{Theorem:MultiPach} with a sufficiently large, albeit constant parameter $r$ yields a somewhat stronger form of Theorem \ref{Theorem:Pach}, which does {\it not} assume that $|P_1|=|P_2|=\ldots=|P_{d+1}|$. Indeed, if $r\leq \min\{|P_i|\mid 1\leq i\leq d+1\}$ then the vast majority of the $(d+1)$-tuples $(P_{1,j_1},\ldots,P_{d+1,j_{d+1}})$ are tight, and their respective cardinalities satisfy $|P_{i,j_i}|=\Theta(|P_i|/r)=\Theta(|P_i|)$. 
Otherwise, if the cardinality of any set $P_i$ is smaller than $r$, we replace its partition $\Pi_d(P_i,r)$ with a collection of $|P_i|$ singleton sets $\{p_{i,j}\}$, for $1\leq j\leq |P_i|$, and set each $\Delta_{i,j}$ to be an infinitesimally small simplex that surrounds $p_{i,j}$. As a result, the pigeonhole argument in Section \ref{Subsec:CrossedProof} still implies that the vast majority of the $(d+1)$-size families $\K=\{\Delta_{1,j_1},\ldots,\Delta_{d+1,j_{d+1}}\}$ are separated. Furthermore, since each of these separated families $\K$ includes a singleton (or, more precisely, a very small simplex) it cannot be loose in view of Theorem \ref{Theorem:StronglySeparated}.

%Notice that Pach's Theorem yields, for any $d+1$ pairwise sets $P_1,\ldots,P_{d+1}$, each comprised of $n$ points in general position in $\reals^d$, a tight $(d+1)$-tuple $\{Q_i\mid 1\leq i\leq d\}$ with the property that $Q_i\subseteq P_i$, and $|Q_i|\geq c'_d|P_i|$. 
%In contrast, Corollary \ref{Corollary:PartiteTight} implies, for any sufficiently large constant $r$ and $|P_1|=\ldots=|P_{d+1}|=n\geq r$, that {\it the vast majority} of the ``colorful" $(d+1)$-tuples $\{P_{1,j_1},\ldots,P_{d+1,j_{d+1}}\}$ of subsets, which are chosen from the simplicial partitions $\Pi(P_1,r),\ldots,\Pi(P_{d+1},r)$, are tight and, therefore, meet the criteria of Theorem \ref{Theorem:Pach}. (In particular, Corollary \ref{Corollary:PartiteTight} yields Theorem \ref{Theorem:Pach}, albeit with a poor constant.)

\subsection{An efficient same-type lemma} 

\noindent{\bf Definition.} Let $d\geq 1$ and $k\geq 1$ be integers. The {\it order-type} $\nu(p_1,\ldots,p_k)$ of a $k$-point sequence $(p_1,\ldots,p_k)$ in general position in $\reals^d$ is the vector in $\{+,-\}^{k\choose d+1}$ that describes the orientations $\chi(p_{i_1},\ldots,p_{i_{d+1}})$ of all the  $(d+1)$-point subsequences of the form $(p_{i_1},\ldots,p_{i_{d+1}})$, with $1\leq i_1<i_2<\ldots<i_{d+1}\leq k$. 
 
\medskip
In 1998 B\'ar\'any and Valtr \cite{BaranyValtr} established the following Ramsey-type statement with respect to order types.

\begin{theorem}[Same Type Lemma]\label{Theorem:SameType}
 For any positive integers $d$ and $k$, there is $c_d(k)>0$ with the following property: {\it For any finite point sets $P_1,\ldots,P_k\subset \reals^d$, whose union is in general position, there exist pairwise disjoint subsets $Q_1\subseteq P_1,\ldots,Q_k\subseteq P_k$, each of cardinality
$|Q_i|\geq c_d(k)|P_i|$, so that the order type $\nu(p_1,\ldots,p_k)$ of any $k$-point sequence $(p_1,\ldots,p_k)\in Q_1\times \ldots\times Q_k$ is invariant in the choice of the representatives $p_i\in Q_i$.} 
\end{theorem}

\medskip
 An important corollary of the polynomial Tur\'an-type result of Fox, Pach and Suk \cite{Regularity} was the lower estimate $c_d(k)=2^{-O(d^3k\log k)}$.
For $P_1=P_2=\ldots=P_k=P$, a polynomial bound of the form $c_d(k)=\Omega(m^{a_d})$ has been established by Mirzaei and Suk \cite[Lemma 3.2]{MirzaeiSuk}, where $a_d$ is a very large constant that depends on $d$.  
To obtain a better estimate, let us re-formulate Theorem \ref{Theorem:Crossed} as an ``approximate same-type lemma", which fixes the orientation for all but $O\left(r^{d+1-1/d}\right)$ of the $(d+1)$-tuples.

\begin{corollary}\label{Corol:OrderType}\label{Corol:RegularityOrderType}
Let $P$ be a set of $n$ points in general position in $\reals^d$, and $r\leq n$. Then $P$ admits a partition $P=P_1\uplus \ldots\uplus P_{r'}$, by the means of Matou\v{s}ek's simplicial partition, into $r'=\Theta(r)$ pairwise disjoint subsets, each of size $\lceil n/r\rceil \leq |P_i|\leq 2\lceil n/r\rceil$, and with the following property:

 All but $O\left(r^{d+1-1/d}\right)$ of the $(d+1)$-tuples $(P_{i_1},\ldots,P_{i_{d+1}})$, with $1\leq i_1<i_2<\ldots<i_{d+1}\leq r'$ are homogeneous with respect to the orientation function $\chi$, in the sense that all the $(d+1)$-point sequences $(p_1,\ldots,p_{d+1})\in P_{i_1}\times\ldots\times P_{i_{d+1}}$ attain the same orientation $\chi(p_1,\ldots,p_{d+1})\in \{+,-\}$.
 \end{corollary}

 As was mentioned by Mirzaei and Suk (and suggested by Jacob Fox), any fractional statement of this type instantly implies a comparably efficient same-result, by the means of the following general graph-theoretic property.

\begin{theorem}[de Caen Theorem \cite{DeCaen}]
	Let $n\geq k\geq l$, and let $(V,E)$ be an $l$-uniform hypergraph on $n$ vertices without a copy of $K^l_k$ -- the complete $l$-uniform hypergraph on $k$ vertices. Then we have that
	$$
	{n\choose l}-|E|\geq \frac{n-k+1}{n-l+1}\cdot \frac{{n\choose l}}{{k-1\choose l-1}}.
	$$
\end{theorem}

Indeed, consider the $(d+1)$-uniform hypergraph whose vertices $r'=\Theta(r)$ vertices correspond to the sets $P_i$ in Corollary \ref{Corol:RegularityOrderType}, and that describes all the homogeneous (i.e., separated) $(d+1)$-tuples $\{P_{i_1},\ldots,P_{i_{d+1}}\}$.
Since the number of the ``missing" hyperedges is $\displaystyle O(r^{d+1-1/d})$, this hypergraph must encompass a copy of $K^{d+1}_k$, with $k=\Omega\left(r^{1/d^2}\right)$. In other words, the (symmetric) same-type lemma holds with $c_d(k)=\Omega\left(k^{-d^2}\right)$. 

\begin{corollary}
Let $k\geq d+1$ be a positive integer, and $P$ be a finite set of at least $k$ points in general position in $\reals^d$. Then there exist pairwise disjoint subsets $Q_1,\ldots,Q_k\subset P$, each of cardinality $
|Q_i|=\Omega\left(k^{-d^2}|P|\right)$, so that all the sequences $(p_1,\ldots,p_k)\in Q_1\times\ldots\times Q_k$ attain the same order type $\nu(p_1,\ldots,p_k)$.\footnote{The author recently became aware of a similar lower estimate $
|Q_i|=\Omega\left(k^{-d^2}|P_i|\right)$, which was obtained by Bukh and Vasileuski \cite{BukhVas} in the $k$-partite setting of Theorem \ref{Theorem:SameType} using the polynomial partition of Guth and Katz \cite{GuthKatz}.}
 Moreover, such subsets $Q_i$ can be found via the standard simplicial $r$-partition $\Pi(P,r)$ of Matou\v{s}ek, with $r=\Theta\left(k^{d^2}\right)$. The implicit constants of proportionality depend on $d$.
\end{corollary}

\subsection{Open problems} Let us wrap-up the discussion with a few open questions.

\begin{enumerate}
	\item It remains a major challenge is to devise a more explicit (and, hopefully, also more efficient) proof of the key Theorem \ref{Theorem:Loose}, which would not resort to any topological, real-algebraic or other advanced machinery. Though our proof of Theorem \ref{Theorem:Loose} no longer uses the Colored Tverberg Theorem, it still relies on the colored selection result of Karasev \cite{Karasev}, which makes a nominal use of topology. 
	    Furthermore, the proof of Theorem \ref{Theorem:NewTuran} makes a rudimentary use of the polynomial partition of Guth and Katz \cite{GuthKatz}, in the form of multi-level partition of Theorem \ref{Theorem:MultiLevel}. %With some care, a slightly more restricted variant of Theorem \ref{Theorem:NewTuran} can be established through direct use of 
	% the Ham-Sandwich Theorem in the proof of Theorem \ref{Theorem:PolynomialPerturbed} in Section \ref{Subsec:PolynomialPerturbed}, which stems from the Borsuk-Ulam Theorem \cite{BorsukUlam}. As was observed by Agarwal, Matou\v{s}ek and Sharir \cite{AgMaSa}, the use of Borsuk-Ulam Theorem can be replaced by a simple random sampling argument.  
	
	\item The selection exponent $\beta_d=d^5+O(d^4)$ in Theorem \ref{Thm:MainMain}, which stems from our proof of Theorem \ref{Theorem:Loose}, is primarily an artifact of applying our general semi-algebraic Tur\'an-type result -- Theorem \ref{Theorem:NewTuran} -- in the ambient space $\reals^{d^2}$, which we use to represent $(d-1)$-dimensional simplices in $\reals^d$. 
	Can the estimates in Theorem \ref{Theorem:NewTuran} be significantly improved? Any progress in this direction would yield an immediate improvement in the selection exponent $\beta_d$.
	
\item There are remarkably few known upper bounds on the quantity $F_d(n,\eps)$. Specializing to triangles in the plane, Eppstein \cite{EppsteinSelection} proposed a construction which yields $F_2(n,\eps)=O\left(\max\{\eps^2n^3,\eps n^2\}\right)$. Is there a simple upper bound of the form $F_d(n,\eps)=O\left(\eps^{d}n^{d+1}\right)$ for $\eps\geq 1/n$?\end{enumerate}

\subsection{Acknowledgement} The author thanks the anonymous SODA referees for helpful suggestions which helped to improve the presentation. In addition, the author thanks Gil Kalai for insightful discussions, and he especially thanks Shakhar Smorodinsky for pointing out the connection between Tur\'an-type problems and Mnets \cite{Dutta,ZaraNets,MustafaRay}, which proved instrumental to considerably shortening the proof of Lemma \ref{Lemma:BipartiteRegularity}.

%(As was observed by Karasev \cite{Karasev}, his theorem yields a variant of the Colored Tverberg Theorem, albeit with a rather poor constant. Its worth noting that Karasev's result makes a very nominal use of topological tools.)
 
% A significant open problem is to find a 

\appendix
\section{Selection via Tverberg-type theorems}\label{Appendix:Tverberg}
All the known and truly elementary proofs of the First Selection Theorem (the case of $\eps=1$) make use of the Tverberg's Theorem, which implies that any set $A$ of $t(d+1)$ points in $\reals^d$ can be subdivided into $t$ pairwise disjoint subsets $A_1,\ldots,A_{t}$, each of cardinality $d+1$ and so that $\bigcap_{i=1}^{t}\conv(A_i)\neq \emptyset$. Choosing $t=d+1$ and repeating this for ${n\choose (d+1)^2}$ choices of $A\in {P\choose t}$ yields a total of $\Omega_d\left(n^{d+1}\right)$ vertex disjoint $(d+1)$-tuples $(A_1,\ldots,A_{d+1})$ with non-empty intersections $\bigcap_{i=1}^{d+1}\conv(A_i)$. Thus, the Fractional Helly Theorem \cite{FracHelly} yields a point that pierces a fixed fraction of the $d$-simplices within ${P\choose d+1}$.

However, Tverberg's Theorem is of little use unless the hypergraph $(P,E)$ satisfies $|E|/{n\choose d+1}=1-o_d(1)$: given that $n$ is a multiple of $d+1$, the complete $(d+1)$-partite hypergraph $K_{\frac{n}{d+1},\ldots,\frac{n}{d+1}}$ is dense yet it contains no {\it complete sub-hypergraphs} with more than $d+1$ vertices.
Instead, the proof of Theorem \ref{Theorem:Main} by Alon, B\'{a}r\'{a}ny, F\"{u}redi and Kleitman \cite{AlonSelections} is based on the following $(d+1)$-partite extension of Tverberg's Theorem.

\begin{theorem}[The Colored Tverberg Theorem \cite{ColoredTverberg}]\label{Theorem:ColorTverberg}For any integers $d,r\geq 2$ there is a number $t=t(r,d)$, so that the following statement holds:

Let $P_1,\ldots,P_{d+1}$ be point sets in $\reals^d$ of cardinality $t(r,d)$ each. Then there exist $r$ pairwise disjoint subsets $A_1,\ldots,A_{r}$, with $A_i\cap P_j$ for all $1\leq i\leq r$ and $1\leq j\leq d+1$, and so that $\bigcap_{i=1}^{r}\conv\left(A_i\right)\neq \emptyset$.  
\end{theorem}

In other words, any complete geometric $(d+1)$-uniform hypergraph in $\reals^d$ with sides $P_i$ of cardinality $t=t(r,d)$ must contain $r$ vertex disjoint $d$-simplices that admit a common intersection.
The original proof of \v{Z}ivaljevi\'c and Vre\'cica yields $t(r,d)\leq 2p(r)-1$, where $p(r)$ is the smallest prime that is not smaller than $r$. Using the trivial bound $p(r)\leq 2r-1$ shows that $t(r,d)\leq 4r-3$. In the sequel, let $t=t_d$ denote the smallest possible number $t(d+1,d)$, which satisfies $t_d\leq 4d-3$.

To obtain many copies of $K_{t,\ldots,t}$ and, thereby, many intersecting $(d+1)$-tuples of vertex disjoint $d$-simplices within $E$ (that can be ``plugged" into the Fractional Helly's Theorem), Alon, B\'{a}r\'{a}ny, F\"{u}redi and Kleitman invoked the following general graph-theoretic result of Erd\H{o}s and Simonovits.

\begin{theorem}[\cite{ErdosSimon}]\label{Theorem:ErdSim}
For all positive integers $d$ and $t$ there is a positive constant $b=b(d,t)$ such shat the following statement holds.

Every $(d+1)$-uniform hypergraph $(P,E)$ on $|V|=n$ vertices and $|E|=\eps{n\choose d+1}$ edges, and with $n^{-t^{-d}}\ll \eps\leq 1$, contains at least 
$b\eps^{t^{d+1}}n^{(d+1)t}$ copies of $K_{t,\ldots,t}$.
\end{theorem}

Unfortunately, Theorem \ref{Theorem:ErdSim} can yield only $\Omega\left(\eps^{t^{d+1}}n^{d+1}\right)$ $(d+1)$-size sets of edges $f_1,\ldots,f_{d+1}\in E$ with a non-empty intersection 
$\bigcap_{i=1}^{d+1}\conv(f_i)$. Hence, regardless of the choice of the Colored Tverberg constant $t=\Theta(d)$, the ``bottom-up" argument cannot possibly be used to pierce more than $\eps^{(cd)^{d+1}} n^{d+1}$ simplices in $E$.

%\medskip
%Thus, the poor selection exponent $\beta_d$ in Theorem \ref{Theorem:Main} is mainly due to the use of far-too-general Erd\H{o}s-Simonovits Theorem. For example, if this hypergraph is semi-algebraic of bounded description complexity $(D,s)$ as defined in Section \ref{Subsec:SemiAlgebraic}, then it must contain a complete $(d+1)$-partite hypergraph $K_{t,\ldots,t}$, with vertex classes of enormous size $t=\eps^{O_{d,s,D}(1)}n$. 

\section{Elementary properties of the hyperplane transversals to simplices}\label{App:ElementaryRubin}

%\begin{figure}
   % \begin{center}      
      %  \input{CrossOrderType.pdf_t}\hspace{2cm}\input{NotSurrounded.pdf_t}
        %\caption{\small Left -- Lemma \ref{Lemma:ExtremalHyperplane} in dimension $d=2$: moving a transversal hyperplane $H$ to $\Sigma=\{\Delta_1,\Delta_2,\Delta_3\}$ to a position in which it contains a pair of vertices while still intersecting the closure of each $\Delta_i$. Right: Proposition \ref{Prop:NotSurrounded} for $d=3$: The origin $O$ lies in the triangle $\triangle x_1x_2x_3$, for $x_1\in \Delta_1,x_2\in \Delta_3,x_3\in \Delta_3$, yet it is not surrounded by $\Delta_1,\Delta_2,\Delta_3$. The simplices $\Delta_1$ and $\Delta_2$ are crossed by a line through $O$.}
        %\label{Fig:ExtremalHyperplane}
    %\end{center}
%\end{figure}

In this section we establish several of the elementary properties of hyperplane transversals to families of at most $d+1$ simplices.

\bigskip
\noindent{\bf Proof of Lemma \ref{Lemma:PinHyperplane}.} Let $H$ be any hyperplane that crosses $\Delta_1,\ldots,\Delta_{d+1}$. We continuously translate $H$ in a fixed direction until the first time $H$ meets a vertex $v_1$ of a simplex $\Delta_{i_1}\in \Sigma$. We continuously rotate $H$ in a fixed direction around $v_1$, so the normal to $H$ is moving along a fixed great circle of ${\mathbb{S}}^{d-1}$, until the first time $H$ hits an additional vertex $v_2$ of a simplex $\Delta_{i_2}$ (where $i_1$ is not necessarily different from $i_2$). We repeat this procedure $d-1$ times, each time adding a vertex $v_j$ of some simplex $\Delta_{i_j}\in \Sigma$, for $2\leq j\leq d$, to the hyperplane $H$ which we rotate around the affine hull of the previously added vertices $v_1,\ldots,v_{j-1}$. (This is possible because $|\bigcup_{i=1}^{d+1}V(\Delta_i)|\geq d+1$, whilst $H$ never contains more than $d$ of these vertices.) Once $H$ encounters a new vertex $v_{i_j}$, this vertex never leaves $H$, which reduces by $1$ the available degrees of freedom that are available to the following rotation step. The process terminates when $H$ contains $d$ vertices, so its rotation is no longer possible. 

Notice that, at every step, $H$ meets (the closure of) every simplex $\Delta_{j}$, for the contact between $H$ and $\Delta_j$ cannot be lost before $H$ ``gains" at least one of the $d+1$ vertices of $\Delta_j$. $\Box$

\bigskip
\noindent {\bf Another proof of Lemma \ref{Claim:HyperplaneThroughPoint}.} Let us consider the family
$$
\C=\{C\left(\K\setminus \{\Delta_{i}\}\right)\mid 1\leq i\leq d+1\}.
$$

\noindent 
According to the Nerve Theorem (see, e.g., \cite[Theorem 4.4]{BorsukUlam}), the union $\bigcup \C$ is homeomorphic to the nerve complex $\N(\C)=\{\C'\subseteq \C\mid \bigcap \C'\neq \emptyset\}$. 
Furthermore, since $\K$ is loose, we have that $\bigcap \C=\bigcap_{1\leq i\leq d+1} C\left(\K\setminus \{\Delta_{i}\}\right)\neq \emptyset$. It, therefore, follows that the nerve complex $\N(\C)$ is complete and isomorphic to the $d$-dimensional simplex; hence, the union $\bigcup \C$ must be contractible. 
Since the vertex set $V(\K)$ is in general position, the boundary of $C(\C)=C(\K)$ is comprised of closed $(d-1)$-dimensional simplices, and each of these simplices is also a boundary simplex of some set in $\C$. We, therefore, conclude that $x$ belongs to the union $\bigcup \C=C(\C)=C(\K)$.

Fix any $1\leq i\leq d+1$ so that $x\in C\left(\K\setminus \{\Delta_{i}\}\right)$. It remains to show that the set family, of cardinality $d+1$,
$$
\K'=\{\{x\}\}\cup (\K\setminus \{\Delta_i\})
$$
\noindent is crossed (i.e., {\it not} separated). Indeed, suppose for a contradiction that $\K'$ is separated.
Then Lemma \ref{Lemma:SeparatingHyperplane} yields a hyperplane $H'$ that is tangent to each set in  $\K\setminus \{\Delta_i\}$ while separating this $d$-size family from $\{x\}$, so that $x$ cannot lie in the convex hull $\conv\left(\K\setminus \{\Delta_i\}\right)$. $\Box$

\end{document}

%% file: TightSimplices.pdf_t
\begin{picture}(0,0)%
\includegraphics{TightSimplices.pdf}%
\end{picture}%
\setlength{\unitlength}{1697sp}%
\begingroup\makeatletter\ifx\SetFigFont\undefined%
\gdef\SetFigFont#1#2#3#4#5{%
  \reset@font\fontsize{#1}{#2pt}%
  \fontfamily{#3}\fontseries{#4}\fontshape{#5}%
  \selectfont}%
\fi\endgroup%
\begin{picture}(7854,5538)(381,-5410)
\put(4271,-665){\makebox(0,0)[rb]{\smash{{\SetFigFont{11}{13.2}{\rmdefault}{\mddefault}{\updefault}{\color[rgb]{1,0,0}$\Delta_3$}%
}}}}
\put(4700,-3298){\makebox(0,0)[lb]{\smash{{\SetFigFont{11}{13.2}{\rmdefault}{\mddefault}{\updefault}{\color[rgb]{1,0,0}$x$}%
}}}}
\put(7235,-4959){\makebox(0,0)[lb]{\smash{{\SetFigFont{12}{14.4}{\rmdefault}{\mddefault}{\updefault}{\color[rgb]{0,0,.56}$P_2$}%
}}}}
\put(1577,-4891){\makebox(0,0)[rb]{\smash{{\SetFigFont{11}{13.2}{\rmdefault}{\mddefault}{\updefault}{\color[rgb]{0,0,.56}$P_1$}%
}}}}
\put(4937,-658){\makebox(0,0)[lb]{\smash{{\SetFigFont{12}{14.4}{\rmdefault}{\mddefault}{\updefault}{\color[rgb]{0,0,.56}$P_3$}%
}}}}
\put(7570,-3932){\makebox(0,0)[lb]{\smash{{\SetFigFont{11}{13.2}{\rmdefault}{\mddefault}{\updefault}{\color[rgb]{1,0,0}$\Delta_2$}%
}}}}
\put(1230,-3802){\makebox(0,0)[rb]{\smash{{\SetFigFont{11}{13.2}{\rmdefault}{\mddefault}{\updefault}{\color[rgb]{1,0,0}$\Delta_1$}%
}}}}
\put(1963,-4268){\makebox(0,0)[rb]{\smash{{\SetFigFont{10}{12.0}{\rmdefault}{\mddefault}{\updefault}{\color[rgb]{0,0,.56}$p_1$}%
}}}}
\put(6698,-4359){\makebox(0,0)[lb]{\smash{{\SetFigFont{10}{12.0}{\rmdefault}{\mddefault}{\updefault}{\color[rgb]{0,0,.56}$p_2$}%
}}}}
\put(5064,-1105){\makebox(0,0)[lb]{\smash{{\SetFigFont{10}{12.0}{\rmdefault}{\mddefault}{\updefault}{\color[rgb]{0,0,.56}$p_3$}%
}}}}
\end{picture}%

%% file: Tight3.pdf_t
\begin{picture}(0,0)%
\includegraphics{Tight3.pdf}%
\end{picture}%
\setlength{\unitlength}{1776sp}%
\begingroup\makeatletter\ifx\SetFigFont\undefined%
\gdef\SetFigFont#1#2#3#4#5{%
  \reset@font\fontsize{#1}{#2pt}%
  \fontfamily{#3}\fontseries{#4}\fontshape{#5}%
  \selectfont}%
\fi\endgroup%
\begin{picture}(6343,4388)(1527,-5080)
\put(4651,-4474){\makebox(0,0)[rb]{\smash{{\SetFigFont{9}{10.8}{\rmdefault}{\mddefault}{\updefault}{\color[rgb]{1,0,0}$C_3$}%
}}}}
\put(2410,-4339){\makebox(0,0)[rb]{\smash{{\SetFigFont{9}{10.8}{\rmdefault}{\mddefault}{\updefault}{\color[rgb]{1,0,0}$K_1$}%
}}}}
\put(4658,-1393){\makebox(0,0)[lb]{\smash{{\SetFigFont{9}{10.8}{\rmdefault}{\mddefault}{\updefault}{\color[rgb]{1,0,0}$K_3$}%
}}}}
\put(6835,-4484){\makebox(0,0)[lb]{\smash{{\SetFigFont{9}{10.8}{\rmdefault}{\mddefault}{\updefault}{\color[rgb]{1,0,0}$K_2$}%
}}}}
\put(3766,-2809){\makebox(0,0)[rb]{\smash{{\SetFigFont{9}{10.8}{\rmdefault}{\mddefault}{\updefault}{\color[rgb]{1,0,0}$C_2$}%
}}}}
\put(5492,-2850){\makebox(0,0)[lb]{\smash{{\SetFigFont{9}{10.8}{\rmdefault}{\mddefault}{\updefault}{\color[rgb]{1,0,0}$C_1$}%
}}}}
\end{picture}%

%% file: Tight1.pdf_t
\begin{picture}(0,0)%
\includegraphics{Tight1.pdf}%
\end{picture}%
\setlength{\unitlength}{1697sp}%
\begingroup\makeatletter\ifx\SetFigFont\undefined%
\gdef\SetFigFont#1#2#3#4#5{%
  \reset@font\fontsize{#1}{#2pt}%
  \fontfamily{#3}\fontseries{#4}\fontshape{#5}%
  \selectfont}%
\fi\endgroup%
\begin{picture}(6474,5385)(1419,-5864)
\put(7801,-4864){\makebox(0,0)[lb]{\smash{{\SetFigFont{9}{10.8}{\rmdefault}{\mddefault}{\updefault}{\color[rgb]{1,0,0}$K_2$}%
}}}}
\put(1434,-4822){\makebox(0,0)[rb]{\smash{{\SetFigFont{9}{10.8}{\rmdefault}{\mddefault}{\updefault}{\color[rgb]{1,0,0}$K_1$}%
}}}}
\put(5264,-746){\makebox(0,0)[lb]{\smash{{\SetFigFont{9}{10.8}{\rmdefault}{\mddefault}{\updefault}{\color[rgb]{1,0,0}$K_3$}%
}}}}
\put(5037,-3292){\makebox(0,0)[rb]{\smash{{\SetFigFont{9}{10.8}{\rmdefault}{\mddefault}{\updefault}{\color[rgb]{0,0,.56}$\Delta(\K)$}%
}}}}
\put(2160,-5594){\makebox(0,0)[rb]{\smash{{\SetFigFont{9}{10.8}{\rmdefault}{\mddefault}{\updefault}{\color[rgb]{0,0,.56}$H_2$}%
}}}}
\put(2325,-4227){\makebox(0,0)[rb]{\smash{{\SetFigFont{9}{10.8}{\rmdefault}{\mddefault}{\updefault}{\color[rgb]{0,0,.56}$x_1$}%
}}}}
\put(7466,-3331){\makebox(0,0)[lb]{\smash{{\SetFigFont{9}{10.8}{\rmdefault}{\mddefault}{\updefault}{\color[rgb]{0,0,.56}$H_3$}%
}}}}
\put(6623,-4728){\makebox(0,0)[lb]{\smash{{\SetFigFont{10}{12.0}{\rmdefault}{\mddefault}{\updefault}{\color[rgb]{0,0,.56}$x_2$}%
}}}}
\put(6747,-5755){\makebox(0,0)[lb]{\smash{{\SetFigFont{9}{10.8}{\rmdefault}{\mddefault}{\updefault}{\color[rgb]{0,0,.56}$H_1$}%
}}}}
\put(4839,-1493){\makebox(0,0)[lb]{\smash{{\SetFigFont{9}{10.8}{\rmdefault}{\mddefault}{\updefault}{\color[rgb]{0,0,.56}$x_3$}%
}}}}
\end{picture}%

%% file: Loose.pdf_t
\begin{picture}(0,0)%
\includegraphics{Loose.pdf}%
\end{picture}%
\setlength{\unitlength}{1697sp}%
\begingroup\makeatletter\ifx\SetFigFont\undefined%
\gdef\SetFigFont#1#2#3#4#5{%
  \reset@font\fontsize{#1}{#2pt}%
  \fontfamily{#3}\fontseries{#4}\fontshape{#5}%
  \selectfont}%
\fi\endgroup%
\begin{picture}(5966,4324)(2235,-5000)
\put(3181,-3640){\makebox(0,0)[rb]{\smash{{\SetFigFont{10}{12.0}{\rmdefault}{\mddefault}{\updefault}{\color[rgb]{1,0,0}$K_1$}%
}}}}
\put(2757,-4864){\makebox(0,0)[lb]{\smash{{\SetFigFont{10}{12.0}{\rmdefault}{\mddefault}{\updefault}{\color[rgb]{0,0,.56}$H_2$}%
}}}}
\put(6209,-4835){\makebox(0,0)[rb]{\smash{{\SetFigFont{10}{12.0}{\rmdefault}{\mddefault}{\updefault}{\color[rgb]{0,0,.56}$H_1$}%
}}}}
\put(7177,-2563){\makebox(0,0)[lb]{\smash{{\SetFigFont{10}{12.0}{\rmdefault}{\mddefault}{\updefault}{\color[rgb]{0,0,.56}$H_3$}%
}}}}
\put(4812,-1990){\makebox(0,0)[lb]{\smash{{\SetFigFont{10}{12.0}{\rmdefault}{\mddefault}{\updefault}{\color[rgb]{1,0,0}$K_3$}%
}}}}
\put(6281,-3641){\makebox(0,0)[lb]{\smash{{\SetFigFont{10}{12.0}{\rmdefault}{\mddefault}{\updefault}{\color[rgb]{1,0,0}$K_2$}%
}}}}
\end{picture}%

%% file: Split.pdf_t
\begin{picture}(0,0)%
\includegraphics{Split.pdf}%
\end{picture}%
\setlength{\unitlength}{1579sp}%
\begingroup\makeatletter\ifx\SetFigFont\undefined%
\gdef\SetFigFont#1#2#3#4#5{%
  \reset@font\fontsize{#1}{#2pt}%
  \fontfamily{#3}\fontseries{#4}\fontshape{#5}%
  \selectfont}%
\fi\endgroup%
\begin{picture}(6688,5234)(1338,-5978)
\put(4780,-3210){\makebox(0,0)[rb]{\smash{{\SetFigFont{10}{12.0}{\rmdefault}{\mddefault}{\updefault}{\color[rgb]{0,0,.56}$\Delta$}%
}}}}
\put(2325,-4227){\makebox(0,0)[rb]{\smash{{\SetFigFont{9}{10.8}{\rmdefault}{\mddefault}{\updefault}{\color[rgb]{0,0,.56}$x_1$}%
}}}}
\put(2540,-5255){\makebox(0,0)[lb]{\smash{{\SetFigFont{9}{10.8}{\rmdefault}{\mddefault}{\updefault}{\color[rgb]{0,0,.56}$G_2$}%
}}}}
\put(6316,-5283){\makebox(0,0)[rb]{\smash{{\SetFigFont{9}{10.8}{\rmdefault}{\mddefault}{\updefault}{\color[rgb]{0,0,.56}$G_1$}%
}}}}
\put(6994,-4245){\makebox(0,0)[lb]{\smash{{\SetFigFont{9}{10.8}{\rmdefault}{\mddefault}{\updefault}{\color[rgb]{0,0,.56}$G_1^+\cap G_3^+$}%
}}}}
\put(6623,-4728){\makebox(0,0)[lb]{\smash{{\SetFigFont{9}{10.8}{\rmdefault}{\mddefault}{\updefault}{\color[rgb]{0,0,.56}$x_2$}%
}}}}
\put(7466,-3331){\makebox(0,0)[lb]{\smash{{\SetFigFont{9}{10.8}{\rmdefault}{\mddefault}{\updefault}{\color[rgb]{0,0,.56}$G_3$}%
}}}}
\put(1353,-4707){\makebox(0,0)[lb]{\smash{{\SetFigFont{9}{10.8}{\rmdefault}{\mddefault}{\updefault}{\color[rgb]{0,0,.56}$G_2^+\cap G_3^+$}%
}}}}
\put(4025,-1141){\makebox(0,0)[lb]{\smash{{\SetFigFont{9}{10.8}{\rmdefault}{\mddefault}{\updefault}{\color[rgb]{0,0,.56}$G_2^+\cap G_1^+$}%
}}}}
\put(4839,-1493){\makebox(0,0)[lb]{\smash{{\SetFigFont{9}{10.8}{\rmdefault}{\mddefault}{\updefault}{\color[rgb]{0,0,.56}$x_3$}%
}}}}
\end{picture}%

%% file: CrossOrderType.pdf_t
\begin{picture}(0,0)%
\includegraphics{CrossOrderType.pdf}%
\end{picture}%
\setlength{\unitlength}{1381sp}%
\begingroup\makeatletter\ifx\SetFigFont\undefined%
\gdef\SetFigFont#1#2#3#4#5{%
  \reset@font\fontsize{#1}{#2pt}%
  \fontfamily{#3}\fontseries{#4}\fontshape{#5}%
  \selectfont}%
\fi\endgroup%
\begin{picture}(8559,3323)(-300,-5702)
\put(3791,-5552){\makebox(0,0)[rb]{\smash{{\SetFigFont{9}{10.8}{\rmdefault}{\mddefault}{\updefault}{\color[rgb]{0,0,0}$\Delta_2$}%
}}}}
\put(2669,-4991){\makebox(0,0)[rb]{\smash{{\SetFigFont{9}{10.8}{\rmdefault}{\mddefault}{\updefault}{\color[rgb]{0,0,0}$H$}%
}}}}
\put(991,-5346){\makebox(0,0)[rb]{\smash{{\SetFigFont{9}{10.8}{\rmdefault}{\mddefault}{\updefault}{\color[rgb]{0,0,0}$\Delta_1$}%
}}}}
\put(6517,-3670){\makebox(0,0)[rb]{\smash{{\SetFigFont{9}{10.8}{\rmdefault}{\mddefault}{\updefault}{\color[rgb]{0,0,0}$\Delta_3$}%
}}}}
\end{picture}%

%% file: Pinned.pdf_t
\begin{picture}(0,0)%
\includegraphics{Pinned.pdf}%
\end{picture}%
\setlength{\unitlength}{1697sp}%
\begingroup\makeatletter\ifx\SetFigFont\undefined%
\gdef\SetFigFont#1#2#3#4#5{%
  \reset@font\fontsize{#1}{#2pt}%
  \fontfamily{#3}\fontseries{#4}\fontshape{#5}%
  \selectfont}%
\fi\endgroup%
\begin{picture}(5888,3684)(2313,-4360)
\put(4302,-3243){\makebox(0,0)[lb]{\smash{{\SetFigFont{10}{12.0}{\rmdefault}{\mddefault}{\updefault}{\color[rgb]{0,0,.56}$x$}%
}}}}
\put(7076,-3110){\makebox(0,0)[lb]{\smash{{\SetFigFont{10}{12.0}{\rmdefault}{\mddefault}{\updefault}{\color[rgb]{0,0,.56}$H$}%
}}}}
\put(6379,-3816){\makebox(0,0)[lb]{\smash{{\SetFigFont{10}{12.0}{\rmdefault}{\mddefault}{\updefault}{\color[rgb]{1,0,0}$\Delta_2$}%
}}}}
\put(5052,-1508){\makebox(0,0)[lb]{\smash{{\SetFigFont{10}{12.0}{\rmdefault}{\mddefault}{\updefault}{\color[rgb]{1,0,0}$\Delta_3$}%
}}}}
\put(2837,-3768){\makebox(0,0)[rb]{\smash{{\SetFigFont{10}{12.0}{\rmdefault}{\mddefault}{\updefault}{\color[rgb]{1,0,0}$\Delta_1$}%
}}}}
\end{picture}%

%% file: SelectionJournalReady.bbl
\begin{thebibliography}{99}


%\bibitem{AgMa} P. K. Agarwal and J.Matou\v{s}ek, On range searching with semialgebraic sets, {\it Discrete Comput. Geom.} 11 (1994), 393--418.

%\bibitem{AgMaSa} P. K. Agarwal, J. Matou\v{s}ek, and M. Sharir, On range searching with semialgebraic sets II, {\it SIAM J. Comput.} 42 (2013), 2039--2062.


\bibitem{Levels} P. K. Agarwal, B. Aronov, T. Chan, and M. Sharir, On levels in arrangements of lines, segments, planes, and triangles, {\it Discrete Comput. Geom.} 19 (1998), 315--331.


\bibitem{AlonSelections} N. Alon, I. B\'{a}r\'{a}ny, Z. F\"{u}redi and  D. J. Kleitman,
Point selections and weak $\eps$-nets for convex hulls, {\it Comb. Prob. Comput.} 1 (1992), 189--200.

\bibitem{AlonSemi} N. Alon, J. Pach, R. Pinchasi, R. Radoi\v{c}i\'c, and M. Sharir, Crossing patterns of semi-algebraic sets, {\it J. Combin. Theory Ser. A} 111 (2005), 310--326.


%\bibitem{Alon} N. Alon, 
%A non-linear lower bound for planar epsilon-nets, 
%{\it Discrete Comput. Geom.} 47 (2012), 235--244.

%\bibitem{AlonKalai} N. Alon and G. Kalai,
%Bounding the Piercing Number, {\it Discrete Comput. Geom.} 13 (1995), 245--256.

%\bibitem{AKMM} N. Alon, G. Kalai, R. Meshulam, and J. Matou\v{s}ek. Transversal numbers for hypergraphs arising in geometry, {\it Adv. Appl. Math} 29 (2001), 79-–101.

%\bibitem{AlonChains} N. Alon, H. Kaplan, G. Nivasch, M. Sharir and S. Smorodinsky,
%Weak $\eps$-nets and interval chains, {\it J. ACM} 55 (6) (2008), Article 28.

%\bibitem{AlonKleitman} N. Alon and D. J. Kleitman, Piercing convex sets and the Hadwiger-Debrunner (p,q)-problem, {\it Adv. Math.} 96 (1) (1992), 103--112. 

%\bibitem{NetsRectangles} B. Aronov, E. Ezra and M. Sharir, 
%Small-size epsilon nets for axis-parallel rectangles and boxes, 
%{\it SIAM J. Comput.} 39 (2010), 3248--3282. 

%\bibitem{ManyCellsAMS} B. Aronov, J. Matou\v{s}ek and M. Sharir, On the sum of squares of cell complexities in hyperplane arrange-
%ments, {\it J. Combin. Theory Ser. A.} 65 (1994), 311--321.

%\bibitem{ManyCellsAS} B. Aronov and M. Sharir, Cell Complexities in Hyperplane Arrangements, {\it Discrete Comput. Geom.} 32 (2004), 107-–115.

%\bibitem{APS} B. Aronov, M. Pellegrini and M. Sharir, 
%On the zone of a surface in a hyperplane arrangement, {\it Discrete Comput. Geom.} 9 (2) (1993), 177 -- 186.

%\bibitem{BMS} J. Balogh, R. Morris and W. Samotij, Independent sets in hypergraphs, {\it J. AMS}, 28 (2015), 669 -- 709.

%\bibitem{BS18} J. Balogh and J. Solymosi, On the number of points in general position in the plane, {\it Disc. Analysis} 16 (2018), 1 -- 20.

\bibitem{SelectionPlane} B. Aronov, B. Chazelle, H. Edelsbrunner, L. J. Guibas, Points and triangles in the plane and halving planes in space, {\it Discrete Comput. Geom.} 6(1) (1991), 435 -- 442.

%\bibitem{EstherSearching} B. Aronov, E. Ezra, and J. Zahl,
%Constructive polynomial partitioning for algebraic curves in $\reals^3$ with Applications, {\it SIAM J. Comput.} 49(6) (2020), 1109--1127.


\bibitem{Barany} I. B\'ar\'any, A generalization of Carath\'eodory's Theorem, {\it Discrete Math.} 40(2-3) (1982), 141--152.

\bibitem{BFL}  I. B\'{a}r\'{a}ny, Z. F\"{u}redi and L. Lov\'{a}sz, On the number of halving planes,
{\it Combinatorica} 10 (2) (1990), 175 -- 183.

\bibitem{BaranyKalai} I. B\'{a}r\'{a}ny and G. Kalai, Helly-type problems, {\it Bul. Amer. Math. Soc.} 59 (4) (2022), 471--502.

\bibitem{BaranyValtr} I. B\'ar\'any and P. Valtr, A positive fraction Erd\H{o}s - Szekeres Theorem,
{\it Discrete Comput. Geom.} 19 (1998), 335--342.

\bibitem{BaroneBasu} S. Barone and S. Basu, Refined bounds on the number of connected components of sign conditions on a variety, {\it Discrete Comput. Geom.} 47 (2012), 577--597. 

\bibitem{BasuSurv} S. Basu, Algorithms in real algebraic geometry: A survey, in {\it Real Algebraic Geometry}, vol. 51 of Panor. Synthe\'ses, Soc. Math. France, Paris, 2017, 107--153.

\bibitem{BasuBook} S. Basu, R. Pollack, and M.F. Roy, {\it Algorithms in Real Algebraic Geometry (2nd ed.)}, Springer-Verlag, Berlin, 2006.

%\bibitem{WellSeparation} H. Bergold, D. Bertschinger, N. Grelier, W. Mulzer, and P. Schnider, Well-separation and hyperplane transversals in
%high dimensions, In {\it Proc. 18th Scand. Symp. Work. Alg. Theory (SWAT 2022)}, Article 16, 2022.

\bibitem{BMZ} P. V. M. Blagojevi\'c, B. Matschke, and G. M. Ziegler, Optimal bounds for the colored Tverberg problem, {\it J. Europ. Math. Soc.} 17 (4), 2015, 739-–754.

\bibitem{BorosFuredi} E. Boros and Z. F\"{u}redi, The number of triangles covering the center of an n-set, {\it Geom. Dedicata} 17(1) (1984), 69--77.


\bibitem{BMN} B. Bukh, J. Matou\v{s}ek, and G. Nivasch, Stabbing simplices by points and flats, {\it Discrete Comput. Geom.} 43 (2010), 321--338.

\bibitem{BukhVas} B. Bukh and A. Vasileuski, New bounds for the same-type lemma, {\tt https://arxiv.org/abs/2309.10731}, 2023.




%\bibitem{BroGood} H. Br\"{o}nnimann and M. T. Goodrich, Almost optimal set covers in finite VC-dimension, {\it Discrete Comput. Geom.} 14(4) (1995), 463--479.

%\bibitem{Staircase} B. Bukh, J. Matousek, and G. Nivasch, 
%Lower bounds for weak epsilon-nets and stair-convexity, {\it Israel J. Math.} 182 (2011), 199--228.

%\bibitem{Sphere} Ph. G. Bradford and V. Capoyleas,
%Weak epsilon-nets for points on a hypersphere, {\it Discrete Comput. Geom.} 18(1) (1997), 83 -- 91.

%\bibitem{Chan} T. M. Chan, Optimal Partition Trees, {\it Discrete Comput. Geom.} 47 (2012), 661-–690.

\bibitem{DeCaen} D. de Caen, Extension of a theorem of Moon and Moser on complete subgraphs, {\it Ars Combinatoria} 16 (1983), 5--10.

\bibitem{Cappell} S. E. Cappell, J. E. Goodman, J. Pach, R. Pollack, M. Sharir, and R. Wenger, Common
tangents and common transversals, {\it Adv. Math.} 106 (2) (1994), 198--215.

\bibitem{Castillo} F. Castillo, J. Doolittle, and J. A. Samper, Common tangents to convex bodies, {\tt https://arxiv.org/pdf/2108.13569}, 2021.

\bibitem{Cuttings} B. Chazelle, Cutting hyperplane arrangements, {\it Discrete Comput. Geom.} 6 (1991), 385--406.

%\bibitem{ChazelleWelzl} B. Chazelle and E. Welzl, 
%Quasi-optimal range searching in spaces of finite VC-dimension, {\it Discrete Comput. Geom.} 4 (1989), 467--489.





%\bibitem{ChazelleBook} B. Chazelle, The discrepancy method: randomness and complexity,
%Cambridge University Press, New York, NY, USA, 2000.


%\bibitem{Chazelle} B. Chazelle, H. Edelsbrunner, M. Grigni, L. J. Guibas, M. Sharir and E. Welzl,
%Improved bounds on weak epsilon-nets for convex sets, {\it Discrete Comput. Geom.} 13 (1995), 1--15. Also in {\it Proc. 25th ACM Sympos. Theory Comput. (STOC)}, 1993.

%\bibitem{Cuttings} B. Chazelle and J. Friedman,
%A deterministic view of random sampling and its use in geometry, {\it Combinatorica} 10 (3) (1990), 229--249.

\bibitem{CS}
K. L. ~Clarkson and P.~Shor, Applications of random sampling in
computational geometry, II, \emph{Discrete Comput. Geom.} 4 (1989),
387--421.

\bibitem{TOPP} E. D. Demaine, J. S. B. Mitchell, and J. O'Rourke, The Open Problems Project, {\tt https://topp.openproblem.net/p7}.

\bibitem{Dey} T. K. Dey, Improved bounds for planar $k$-sets and related problems, {\em Discrete Comput. Geom.} 19 (3) (1998), 373--382.

\bibitem{Dutta} K. Dutta, A. Ghosh, B. Jartoux, N. M. Mustafa, {Shallow packings, semialgebraic set systems, Macbeath regions, and polynomial partitioning}, {\it Discrete Comput. Geom.} 61 (2019), 756--777. 

%\bibitem{Doubling} K. L. Clarkson, Nearest neighbor queries in metric spaces, {\it Discrete Comput. Geom.}, 22 (1) (1999), 63--93.

%\bibitem{ManyCells} K. L. Clarkson, H. Edelsbrunner, L. J. Guibas, M. Sharir and E. Welzl,
%Combinatorial complexity bounds for arrangement of curves and spheres, {\it Discrete Comput. Geom.} 5 (1990), 99--160.

%\bibitem{ClaVar} K. L. Clarkson and K. R. Varadarajan, Improved approximation algorithms for geometric set cover, {\it Discrete Comput. Geom.} 37(1) (2007), 430--58.

%\bibitem{ManyCells3d} H. Edelsbrunner, L. J. Guibas, and M. Sharir, The complexity of many cells in arrangements of planes and related problems, {\it Discrete Comput. Geom.} 5 (1990), 197--216.

\bibitem{GeneralPosition} H. Edelsbrunner, E. M\"{u}cke, Simulation of simplicity: a technique to cope with degenerate cases in geometric algorithms, {\it ACM Trans. Graph.} 9 (1) (1990), 66--104.

\bibitem{ErdosSimon} P. Erd\H{o}s and M. Simonovits,  Supersaturated graphs and hypergraphs, {\it Combinatorica} 3 (2) (1983), 181--192. 

\bibitem{EppsteinSelection} D. Eppstein, Improved bounds for intersecting triangles and halving planes, {\it J. Combin. Theory Ser. A} 62 (1993), 176--182.

%\bibitem{Even} G. Even, D. Rawitz and S. Shahar, Hitting sets when the VC-dimension is small, {\it Inf. Process. Lett.} 95(2) (2005), 358--362.

\bibitem{Overlap} J. Fox, M. Gromov, V. Lafforgue, A. Naor, and J. Pach, Overlap properties of geometric expanders, {\it Journal f\"{u}r die reine und angewandte Mathematik (Crelle's Journal)} 671 (2012), 49-83.

\bibitem{Regularity} J. Fox, J. Pach, and A. Suk, A polynomial regularity lemma for semialgebraic hypergraphs and Its applications in geometry and property testing, {\it SIAM J. Comput.} 45(6) (2016), 2199--2223.

\bibitem{Gromov} M. Gromov, Singularities, expanders and topology of maps. Part 2: From combinatorics to topology via algebraic isoperimetry, {\it Geom. Funct. Anal.} 20 (2) (2010), 416--526.

\bibitem{GuthKatz} L. Guth and N. H. Katz, On the Erd\H{o}s distinct distance problem in the plane, {\it Annals Math.} 181 (2015), 155--190. Also in arXiv:1011.4105.

\bibitem{Haussler} D. Haussler, Sphere packing numbers for subsets of the Boolean n-cube with bounded Vapnik-Chervonenkis dimension, {\it J. Comb. Theory, Ser. A} 69 (2) (1995), 217--232.

\bibitem{ArrangementsSurvey} D. Halperin and M. Sharir, Arrangements, Chapter 27 in Handbook of Discrete and Computational Geometry, J.E. Goodman, J. O'Rourke, and C. D. T\'{o}th (ed.), 3rd edition, CRC Press, Boca Raton, FL, 2017.


\bibitem{T3} A. F. Holmsen, Geometric transversal theory: T(3)-families in the plane, Geometry -- Intuitive, Discrete, and Convex, {\it Bolyai Soc. Math. Stud.} 24, J\'anos Bolyai Math. Soc., Budapest, 2013, 187--203.

\bibitem{Jiang} Z. Jiang, A slight improvement to the colored B\'ar\'any's Theorem, {\it Elec. J. Comb.} 21 (4) (2014).

\bibitem{FracHelly} G. Kalai, Intersection patterns of convex sets, {\it Israel J. Math.} 48 (2-3) (1984), 161--174.

%\bibitem{HJ} H. Furstenberg and Y. Katznelson, A density version of the Hales-Jewett Theorem, {\it J. Analyse Math.}, 57 (1991), 64--119, 1991.


%\bibitem{GPW1}
%J.~E.~Goodman, R.~Pollack and R.~Wenger, Geometric transversal
%theory, {\it New Trends in Discrete and Computational Geometry},
%J. Pach (Ed.), Springer Verlag, Berlin, 1993, pp. 163--198.

%\bibitem{HD} H. Hadwiger and H. Debrunner, \"{U}ber eine variante zum Hellyschen satz, {\it Archiv der Mathematik} 8(4) (1957), 309 -- 313.

%\bibitem{Envelopes3D} D. Halperin and M. Sharir, New bounds for lower envelopes in three dimensions, with applications to visbility in terrains, {\it Discrete Comput. Geom.} 12 (1994), 313-€"-326.

%\bibitem{NetShape} S. Har-Peled and M. Jones, Journey to the center of the point set, in {\it Proc. 35th Int. Symposium Comput. Geom. (SOCG 2019)}, Article 41.

\bibitem{HW87} D. Haussler and E. Welzl, $\varepsilon$-nets and simplex range queries, {\it Discrete Comput. Geom.} 2 (1987), 127--151.

%\bibitem{Holmsen} A. Holmsen, Large cliques in hypergraphs with forbidden substructures, to appear in {\it Discrete Comput. Geom.} (2020).

%\bibitem{HolmsenLee} A. Holmsen, D.-G. Lee, Radon numbers and the fractional Helly theorem,  {\tt arXiv:1903.01068}, priprint, 2019.

%\bibitem{Expanders} S. Hoory, N. Linial, and A. Wigderson, Expander graphs and their applications, {\it Bull. Amer. Math. Soc.} 43 (2006), 439--561.

%\bibitem{KS18} C. Keller and S. Smorodinsky, A New Lower Bound on Hadwiger-Debrunner Numbers in the Plane, in {\it Proc. 2020 ACM-SIAM Symposium on Discrete Algorithms (SODA 2020)}, pp. 1155--1169.

%\bibitem{Shakhar} C. Keller, S. Smorodinsky and G. Tardos,
%Improved bounds for Hadwiger-Debrunner numbers,
%{\it Israel J. Math.} 225 (2) (2018), 925 -- 945. (Also in {\it Proc. SODA 2017,} pp. 2254 -- 2263.)

%\bibitem{KPW90} J. Kolm\'{o}s, J. Pach and G. J. Woeginger,
%Almost tight bounds for epsilon-Nets, {\it Discrete Comput. Geom.} 7 (1992), 163--173.


%\bibitem{Trends} J. Matou\v{s}ek, Epsilon-Nets and Computational Geometry, in {\it New Trends in Discrete Computational Geometry}, J. Pach (Ed.), Algorithms and Combinatorics, Berlin, 1993, pp. 69--89.

\bibitem{Karasev} R. Karasev, A simpler proof of the Boros-F\"uredi-B\'ar\'any-Pach-Gromov theorem, {\it Discrete Comput. Geom.} 47 (2012), 492--495.

\bibitem{BoundsPachs} R. Karasev, J. Kyn\v{c}l, P. Pat\'ak, Z. Pat\'akov\'a, and M. Tancer,  Bounds for Pach's selection theorem and for the minimum solid angle in a simplex,
{\it Discrete Comput. Geom.} 54 (2015), 610--636.

\bibitem{ZaraNets} C. Keller and S. Smorodinsky, Zarankiewicz's problem via $\varepsilon$-$t$-nets, preprint, 2023, {\tt https://arxiv.org/pdf/2311.13662.pdf}.


\bibitem{Lovasz} L. Lov\'asz, On the number of halving lines, {\it Ann. Univ. Sci. Budapest Rolando E\"otv\"os Nom., Sec. Math.} 14 (1971), 107--108.


\bibitem{JirkaBook} J. Matou\v{s}ek, Lectures on Discrete Geometry, Springer-Verlag, New York, 2002.

\bibitem{BorsukUlam}  J. Matou\v{s}ek, Using the Borsuk-Ulam theorem, Lectures on Topological Methods in Combinatorics and Geometry,
     Universitext, Springer-Verlag, Heidelberg, 2003, second corrected printing 2008.

\bibitem{PartitionTrees} J. Matou\v{s}ek, Efficient partition trees, {\it Discrete Comput. Geom.} 8 (3) (1992), 315--334.

\bibitem{MatWag04} J. Matousek and U. Wagner,
New constructions of weak epsilon-nets, {\it Discrete Comput. Geom.} 32 (2) (2004), 195--206.

\bibitem{MultiLevel} J. Matou\v{s}ek and Z. Pat\'akov\'a, Multilevel polynomial partitions and simplified range searching, {\it Discrete Comput. Geom.} 54(1) (2015), 22--41.

\bibitem{MirzaeiSuk} M. Mirzaei and A. Suk, A positive fraction mutually avoiding sets theorem, {\it Discrete Math.} 343 (3) (2020), 111730.


%\bibitem{MSW90}	J. Matousek, R. Seidel and E. Welzl,
%How to net a lot with little: small epsilon-nets for disks and halfspaces, {\it Proc. 6th ACM Symp. Comput. Geom.}, 1990, pp. 16--22.


%\bibitem{MoYe} S. Moran and A. Yehudayoff,
%On weak epsilon-nets and the Radon number, {\it Proc. 35th Symp. Comput. Geom. 2019}, pp. 51:1-51:14.
%\bibitem{Mulmuley} K. Mulmuley, Computational Geometry: An Introduction Through Randomized Algorithms, 1st edition, Prentice Hall, 1993.

%\bibitem{MustafaRay} N. H. Mustafa and S. Ray, Weak $\varepsilon$-nets have a basis of size $O(1/\varepsilon \log 1/\varepsilon)$, {\it Comput.
%Geom.} 40 (2008), 84 -- 91.

%\bibitem{HandbookNets} N. H. Mustafa and K. Varadarajan,  Epsilon-approximations and epsilon-nets, Chapter 47 in Handbook of Discrete and Computational Geometry, J.E. Goodman, J. O'Rourke, and C. D. T\'{o}th (ed.), 3rd edition, CRC Press, Boca Raton, FL, 2017.



\bibitem{MustafaRay} N. H. Mustafa and S. Ray, $\varepsilon$-Mnets: Hitting Geometric Set Systems with Subsets, {\it Discrete Comput. Geom.} 57(3) (2017), 625--640.

\bibitem{GabrielSelection} G. Nivasch and Micha Sharir, 
Eppstein's bound on intersecting triangles revisited, {\it J. Comb. Theory, Ser. A} 116(2) (2009), 494--497.

\bibitem{PachTheorem} J. Pach, A Tverberg-type result on multicolored simplices, {\it Comput. Geom.} 10 (2) (1998), 71--76.

\bibitem{PolWen} R. Pollack and R. Wenger, Necessary and sufficient conditions for hyperplane transversals, {\it Combinatorica} 10(3) (1990), 307--311.

\bibitem{RubinRegularity} N. Rubin, An efficient regularity lemma for semi-algebraic hypergraphs, manuscript, 2023.

\bibitem{Rubin2D} N. Rubin, An improved bound for weak epsilon-nets in the plane, {\it J. ACM} 69 (5) (2022), Article 32, 35pp.


\bibitem{Rubin} N. Rubin, Stronger bounds for weak epsilon-nets in higher dimensions, Proceedings of the Annual Symposium on Foundations of Computer Science (STOC 2021), 2021, pp. 62. Also {\tt https://arxiv.org/pdf/2104.12654.pdf}.

\bibitem{ksets4D} M. Sharir,
An improved bound for k-sets in four dimensions, {\it Comb. Probab. Comput.} 20(1) (2011), 119--129.


\bibitem{ksets3D} M. Sharir, S. Smorodinsky, and G. Tardos, An improved bound for k-sets in three dimensions,
{\it Discrete Comput. Geom.} (26) (2001), 195--204.

\bibitem{StoneTukey} A. H. Stone and J. W. Tukey, Generalized sandwich theorems, {\it Duke Math. J.} 9 (1942), 356--359.

\bibitem{Toth} G. Toth, Point sets with many k-sets, {\it Discrete Comput. Geom.} 26 (2001), 187--194.

\bibitem{Warren} H. E. Warren, Lower bound for approximation by nonlinear manifolds, {\it Trans. Amer. Math. Soc.} 133 (1968), 167--178.



%\bibitem{PachTardos} J. Pach and G. Tardos, Tight lower bounds for the size of epsilon-nets, {\it J. AMS} 26 (2013), 645 -- 658.

%\bibitem{FOCS18} N. Rubin,
%An improved bound for weak Epsilon-nets in the plane, in {\it Proc FOCS 2018}, pp. 224--235.

%\bibitem{Tagansky} B. Tagansky,
%A new technique for analyzing substructures in arrangements of piecewise linear surfaces, {\it Discret. Comput. Geom.} 16(4), 455 -- 479 (1996).

%\bibitem{SaxTh} D. Saxton and A. Thomason, Hypergraph containers, {\it Inventiones Mathematicae} 201 (3) (2015), 925 -- 992.

%\bibitem{EnvelopesHigh}	M. Sharir,
%Almost tight upper bounds for lower envelopes in higher dimensions, {\it Discrete Comput. Geom.} 12 (1994), 327--345.

\bibitem{SA}
M. Sharir and P.~K. Agarwal, {\it Davenport-Schinzel Sequences and
Their Geometric Applications}, Cambridge University Press, New York,
1995.


%\bibitem{SzT} E. Szemer\'{e}di and W. T. Trotter, Extremal problems in discrete geometry, Combinatorica 3 (3-4) (1983), 381 -- 392.






%\bibitem{VC1971} V. N. Vapnik and A. Y. Chervonenkis, On the
%uniform convergence of relative frequencies of events to their probabilities, {\it Theory
%Prob. Appls.} 16 (1971), 264--280. 

\bibitem{WengerProgress} R. Wenger, Progress in geometric transversal theory, in {\it Advances in discrete and computational geometry} (South Hadley, MA, 1996), {\it Contemp. Math.}, 223, {\it Amer. Math. Soc.}, Providence, RI, 1999, pp. 375--393.

\bibitem{Wen-surv}
R.~Wenger, Helly-type theorems and geometric transversals, in
\emph{Handbook of Discrete and Computational Geometry}, 2nd Edition
(J.E. Goodman and J. O'Rourke, Eds.), Chapman \& Hall/CRC Press,
2004, pp. 73--96.

\bibitem{ColoredTverberg} R. T. \v{Z}ivaljevi\'{c} and S. T. Vre\'{c}ica, The colored Tverberg's problem and complexes of injective functions, {\it J. Comb. Theory Ser. A}
61 (2) (1992), 309--318.

\bibitem{TopologicalSurvey} R. T. \v{Z}ivaljevi\'{c}, Topological methods in discrete geometry, Chapter 21 in Handbook of Discrete and Computational Geometry, J.E. Goodman, J. O'Rourke, and C. D. T\'{o}th (ed.), 3rd edition, CRC Press, Boca Raton, FL, 2017.


\end{thebibliography}
